%% file: Main.tex
\documentclass[colorinlistoftodos,12pt]{article}

\usepackage[utf8]{inputenc}

\usepackage[utf8]{inputenc}
\usepackage[a4paper]{geometry}
\usepackage{amsmath,amsfonts,amsthm,epsfig,latexsym,graphicx,amssymb,xfrac,faktor}
\usepackage{enumerate} 
\usepackage[english]{babel}

\usepackage{mathtools}
\usepackage{stmaryrd} 

\usepackage[dvipsnames]{xcolor} 
\usepackage{tikz-cd}

\usepackage{lmodern}

\usepackage{graphpap} 
\usepackage[
  colorlinks = true,
  linkcolor = RoyalBlue, 
  citecolor = RoyalBlue 
  ]{hyperref}

\usepackage{sectsty} 
\sectionfont{\color{Salmon!70!red}}
\subsectionfont{\color{Salmon!70!red}}

\DeclareTextFontCommand{\emph}{\color{ForestGreen}\em} 

\usepackage{theoremref}

\newcommand{\RR}{\mathbb{R}}
\newcommand{\NN}{\mathbb{N}}

\newcommand{\LL}{\mathcal{L}}
\newcommand{\PP}{\mathcal{P}}
\newcommand{\Pre}{\mathfrak{P}}
\newcommand{\II}{\mathcal{I}}
\newcommand{\JJ}{\mathcal{J}}

\newcommand{\ZZ}{\mathbb{Z}}
\newcommand{\QQ}{\mathbb{Q}}

\newcommand{\FF}{\mathcal{F}}
\newcommand{\cC}{\mathcal{C}}
\newcommand{\cR}{\mathcal{R}}
\newcommand{\bp}{\mathbf{x_0}}
\newcommand{\bpx}{\mathbf{x}}

\newcommand{\wb}{\overline}
\newcommand{\ub}{\underline}
\newcommand{\wt}{\widetilde}

\newcommand{\lint}{\llbracket}
\newcommand{\rint}{\rrbracket}
\newcommand{\intint}[1]{{\lint #1 \rint}}		

\DeclareMathOperator{\reg}{Reg}
\DeclareMathOperator{\Int}{Int}
\DeclareMathOperator{\Diag}{Diag}
\DeclareMathOperator{\homeo}{homeo}
\DeclareMathOperator{\diam}{diam}

\DeclareMathOperator{\im}{im}
\DeclareMathOperator{\id}{id}

\DeclareMathOperator{\Leaf}{Leaf}

\DeclareMathOperator{\Sing}{Sing}

\newtheoremstyle{colorplain}%
{\topsep}   
{\topsep}   
{\itshape}  
{0pt}	   
{} 
{.}		 
{5pt plus 1pt minus 1pt} 
{\textbf{\textcolor{RoyalBlue}{\textbf{\thmname{#1} \thmnumber{#2}}}}\thmnote{ (#3)}}
{}

\newtheoremstyle{colorremark}%
{\topsep}   
{\topsep}   
{}  
{0pt}	   
{\itshape} 
{.}		 
{5pt plus 1pt minus 1pt} 
{\textcolor{RoyalBlue}{\thmname{#1} \thmnumber{#2}}\thmnote{ (#3)}}
{}

\newtheoremstyle{colordefinition}%
{\topsep}   
{\topsep}   
{}  
{0pt}	   
{} 
{.}		 
{5pt plus 1pt minus 1pt} 
{\textcolor{RoyalBlue}{\textbf{\thmname{#1} \thmnumber{#2}}}\thmnote{ (#3)}}
{}

\makeatletter 
\renewenvironment{proof}[1][\proofname]{
	\par
	\pushQED{\qed}%
	\normalfont \topsep6\p@\@plus6\p@\relax
	\trivlist
	\item\relax
	{\itshape\color{RoyalBlue}
	#1\@addpunct{.}}\hspace\labelsep\ignorespaces
}{%
	\popQED\endtrivlist\@endpefalse
}
\makeatother

\theoremstyle{colorplain}
\newtheorem{theorem}{Theorem}

\numberwithin{theorem}{section}

\newtheorem{maintheorem}{Theorem}

\newtheorem{remark}[theorem]{Remark}
\newtheorem{example}[theorem]{Example}

\newtheorem{lemma}[theorem]{Lemma}

\newtheorem{proposition}[theorem]{Proposition}

\newtheorem{corollary}[theorem]{Corollary}
\newtheorem{question}{Question}

\theoremstyle{colorremark}

\newtheorem{observation}[theorem]{Observation}

\theoremstyle{colordefinition}
\newtheorem{definition}[theorem]{Definition}

\title{From pre-lamination to foliated Plane}
\author{Christian Bonatti, Théo Marty}
\date{}

\begin{document}

\maketitle

\begin{abstract}
	To a singular foliation on the plane corresponds a circular boundary at infinity endowed with a pre-lamination on the circle. We solve the converse direction. We determine which pre-lamination on the circle are boundary at infinity of a foliation, and we build the corresponding (unique) foliation. We consider both regular foliations and a singular foliations with prong singularities.
\end{abstract}



{
	\hypersetup{linkcolor=black}
	\tableofcontents \label{ToC}
}

\input{Introduction}

\input{preliminaires}
\input{planar}

\input{Preliminar}

\input{planaroflamination}

\input{universel}

\input{RegularFoliation}

\input{SingularFoliation}

\addcontentsline{toc}{section}{References}

\bibliographystyle{alpha}
\bibliography{ref}

\noindent 
Christian Bonatti,\\
CNRS and Universit\'e  Bourgogne-Europe, Dijon, France.\\
\emph{Email address}: bonatti@u-bourgogne.fr

\vskip 2mm
\noindent
Th\'eo Marty,\\
Universit\'e  Bourgogne-Europe, Dijon, France.\\
\emph{Email address}: marty.theo.math@gmail.com.

\end{document}

%% file: Introduction.tex

\section{Introduction} \label{sec-intro}
\subsection{Short presentation of the questions and results}

Foliations on the plane and their relations with pre-laminations on the circle recently saw interesting applications, for instance, to the characterization of pseudo-Anosov flows by their bi-foliated planes. 
For almost a century, it has been known that foliations on the plane are characterized by their leaf spaces \cite{KaplanI,KaplanII}, which are non-Hausdorff 1-manifolds with some additional data. More recently, Mather \cite{Mather} build a natural circular boundary at infinity associated to a foliation of $\RR^2$ (more exactly, he considers germs of foliations at a given isolated singular point). Using similar ideas, given a foliation $\FF$  of the plane $\RR^2$, \cite{BonattiCirc} shows that the plane admits a unique compactification as a compact disc by adding a circle at infinity $\partial_\infty\FF\simeq S^1$, in such a way that  any half leaf of $\FF$ converges toward a point in $S^1$, no more than countably many half leaf have the same endpoint, and that the endpoints of leaves are dense in the circle\footnote{The generalized construction allows at most countable families of (singular) foliations with some mild transversality assumptions, and states the uniqueness of the compactification under some natural conditions.}. 

The circle at infinity is naturally endowed with a pre-lamination as follows. Denote by $\Delta$ the diagonal in $S^1\times S^1$. Given a (non-singular) leaf of the foliation~$\FF$, its two endpoints induce an unordered pair of distinct points in $S^1$, that is, a point in $(S^1\times S^1\setminus\Delta)_{/(x,y)\sim(y,x)}$. All these pairs of endpoints form a collection called \emph{pre-lamination of the circle $S^1$} that is a collection of pairs of points that do not cross each other. We call it \emph{the end pre-lamination $L_\FF$ of~$\FF$}. 

 In this article, we complete the picture by solving the converse question:
\begin{question}\label{q.induced}
	Which pre-laminations of the circle are induced by a foliation on the plane? 
\end{question}
More precisely, starting from a pre-lamination $L$ satisfying some properties, we build an associated non-Hausdorff 1-manifold and a foliation $\FF$ on the plane whose boundary at infinity is endowed with $L_\FF=L$.  We answer the question for usual (non-singular) foliations (Theorem~\ref{main-A}) and for a large family of singular foliations: the one with prong singularities, at least three separatrixes per singularity, and that have no connection between singularities (Theorem~\ref{main-A-sing}).


The uniqueness of the compactification in \cite{BonattiCirc} has an important consequence. Every homeomorphism of the plane preserving the foliation~$\FF$ extends in a unique way on the compact disc $D^2_\FF$ and induces by restriction a homeomorphisms of the circle $\partial_\infty \FF$. Thus any group action $\phi\colon G\to\homeo(\RR^2)$ preserving~$\FF$ extends as a group action $\Phi\colon G\to\homeo(D^2_\FF)$ and induces a group action $\varphi\colon G\to\homeo(\partial_\infty \FF)$ on the circle.

An important long-term motivation for this work is the following question.

\begin{question}\label{q.action} 
    Let $\FF$ be a foliation on the interior of $D^2$, and $\varphi\colon G\to \homeo(\partial_\infty\FF)$ be a group action which preserves the end lamination $L_\FF$ of $\FF$. How can we determine when $\varphi$ can be extended 
    as an action $\Phi\colon G\to\homeo(D^2)$ that preserves~$\FF$? 
\end{question}
Let us spend some time discussing Question~\ref{q.action}. In the setting of Question~\ref{q.action}, let $\{g_i\}_{i\in \II}$ be a systems of generators of the group $G$ and let $F_\II$ be the free group generated by $\{g_i\}_{i\in \II}$. One may choose, for every $i\in \II$, a homeomorphism $h_i$ of the compact disc, preserving $\FF$ and inducing $g_i$ on $\partial_\infty\FF$. This defines an action $\psi$ of $F_\II$ on  $D^2$, preserving $\FF$.  For every $f\in F_\II$ in the kernel $Ker(\rho)$ of the natural morphism $\rho\colon F_\II\to G$, the homeomorphism $\psi(f)$ not only preserves $\FF$ but also preserves every leaf of $\FF$. The question is to determine under which condition one may chose the $\psi(g_i)$  so that $\psi$ would be trivial on $Ker(\rho)$. 

We do not answer this question in the present paper but we tackle it by answering the next question. 
In view of Question~\ref{q.action}, we would like to be able to rebuild the foliation~$\FF$ with a smaller data than its whole end pre-lamination: ideally, one would like to be able to rebuild~$\FF$ by considering a $G$-orbit of a single
pair in $S^1\times S^1\setminus \Delta$. Thus we are interested by the following question:

\begin{question}\label{q.completed}
 Given a pre-lamination $\ell\subset(S^1\times S^1\setminus\Delta)_{/(x,y)\sim(y,x)}$, under which assumptions does it exist a pre-lamination $L\supset\ell$ induced by a foliation (or a singular foliation)~$\FF$ and so that $\ell$ corresponds to a dense subset of leaves of~$\FF$ ?
\end{question}

We give in Theorem~\ref{t.non-complete} a complete answer to Question~\ref{q.completed} in the non-singular setting only. 

\vspace{\baselineskip}
Note that the corresponding questions in the setting of two transverse (possibly singular) foliations have recently been solved in \cite{BBM2024}.  However, the arguments cannot be adapted to the case of one foliation. In fact, two transverse foliations provide enough structure to explicitly answer these questions. Given two transverse pre-laminations on the circles, with good properties, the bi-foliated plane can be rebuilt from the intersection of their leaves: a point of the plane corresponds to a pair of crossing leaves. In contrast, when given only one pre-lamination, we do not have enough structure to generalize this construction.


\subsection{Statement of our results: the non-singular case}
Let us now present our results, first in the non-singular cases. 
Theorem~\ref{main-A} below states that the pre-laminations $L$ which are induced by foliations are exactly the pre-laminations whose endpoints are dense in $S^1$ and with no more than countably many leaves ending at a given endpoints, and that satisfies two additional conditions. We express these conditions using \emph{the geodesic realization $G(L)$} of $L$ that is, the family of Euclidian geodesic segments of $\Int(D^2)$ whose pair of endpoints belongs to~$L$:
\begin{itemize}
	\item \emph{shell property}: the pre-lamination $L$ is a \emph{shell pre-lamination} if every complementary region $\Delta$ of $G(L)$ is a \emph{shell}, that is, the interior of $\Delta$ is not empty, and all but exactly one boundary components of $\Delta$ are leaves of $L$. The boundary component that is not a leaf of $L$ is called \emph{the root} of the shell~$\Delta$. 
	\item \emph{no bad-accumulation}: for any sequence of complementary regions that converges toward a geodesic, their roots accumulate on that geodesic. 
\end{itemize}
The shell condition is related to branching of the foliation, that is non-separated leaves. The no bad-accumulation condition reflects the existence of transverse segments of the foliation at the limiting leaf.

\begin{maintheorem}\thlabel{main-A}
	Let $L$ be a pre-lamination of $S^1$ whose endpoints are dense in $S^1$ and with no more than countably many leaves ending at a given endpoints. 
	Then $L$ is the end pre-lamination of a (non singular) foliation of $\Int(D^2)$ if and only if it is a shell pre-lamination and has no bad accumulation. 
	
	Additionally, when it satisfies the previous condition, the foliation is unique up to a homeomorphism of $D^2$ equal to the identity on $S^1$. 
\end{maintheorem}

Later on, a pre-lamination satisfying the hypotheses in Theorem~\ref{main-A} will be called a $\Pre$-pre-lamination. 
A more formal statement is given in \thref{th-good-lam-are-end-fol}. 

As a corollary of Theorem~\ref{main-A}, Theorem~\ref{t.non-complete} gives a complete answer to Question~\ref{q.completed} by providing a necessary a sufficient condition for a pre-lamination $\ell$ to be contained in the end pre-lamination $L_\FF$ of non-singular foliation~$\FF$, and corresponding to a dense subset of leaves of~$\FF$. The main difference with the hypotheses of Theorem~\ref{main-A} is that complementary regions are allowed to have more than one boundary component that is not a leaf. This requires to modify slightly the no bad-accumulation property. The main difference with the conclusion is that the completion $L_\FF$ of $\ell$ is, in general, not unique, and therefore the foliation~$\FF$ is not unique. 

In this introduction, we present a simple setting where Theorem~\ref{t.non-complete} as a simple expression.
A pre-lamination $\ell$ of $S^1$ is said \emph{regular} if:
\begin{itemize}
 \item each leaf of $\ell$ is accumulated on both sides by leaves of $\ell$,
 \item the leaves are pairwise disjoint, as sets of pairs of points of $S^1$.
\end{itemize}
This setting is natural: a leaf of a foliation is said regular when it is separated from the other leaves. The set of regular leaves of a foliation are dense in $Int(D^2)$. Any dense set of regular leaves of a foliation is a regular pre-lamination: that is the motivation of this notion. 

\begin{maintheorem}\label{t.regular} Let $\ell$ be a regular pre-lamination whose endpoints are dense in $S^1$. 
Then there is a foliation~$\FF$ so that $\ell\subset L_\FF$ and $\ell$ corresponds to a dense subset of leaves of~$\FF$. 

However the pre-lamination $L_\FF$ (and therefore the foliation~$\FF$) is not unique, unless~$\FF$ is conjugated to the trivial foliation by horizontal straight lines. 
\end{maintheorem}

The proof of Theorem~\ref{main-A} is based on Kaplan's work (see \cite{KaplanI, KaplanII} for instance) which characterizes a foliation on $\RR^2$ by it leaf space. Let us recall it briefly. Let~$\FF$ be a cooriented foliation of $\RR^2$. The leaf space of the foliation~$\FF$ is a simply-connected (possibly) non-Hausdorff oriented 1-manifold. Its branchings, that is maximal subsets of at least two points that are pairwise non-separable, correspond to subsets of leaves of~$\FF$ that are approached by a common sequence of leaves. Each branching comes with a natural order: the left-right order when viewing the coorientation going upward. The leaf space of~$\FF$ together with these orientations form a so-called \emph{planar structure}. 
Kaplan proved the $[1:1]$ correspondence between cooriented foliation of the plane (up to homeomorphisms) and planar structures.

An important remark is that the branchings of~$\FF$ are in a canonical $1$ to $1$ correspondence with the complementary regions of the the geodesic realization $G_\FF=G_{L_\FF}$ of $L_\FF$.
The shell property means that each complementary region $\Delta$ has a particular form: exactly one boundary component $\gamma$ of $\Delta$ is not a leaf of $G_\FF$, but it is accumulated by leaves $\gamma_n$ of $G_\FF$. The leaves in $\partial \Delta$ correspond to leaves of~$\FF$ which are pairwise non separated and are all accumulated by the leaves of~$\FF$ corresponding to the $\gamma_n$. Together with the no bad-accumulation property, this will allow us to define a non-Hausdorff 1-manifold as being the leaf space of a pre-lamination $L$ of the circle (Theorem~\ref{t.prelam-planar}). 

\subsection{Statement of our results: the general (singular) case}
In this paper, singular foliations are assumed to be as follows. Their singular points are $k$-prong singularities with $k\geq 3$, and there have no connections between singular points\footnote{\thref{main-A-sing} and \thref{mainth-Kaplan-sing} fail when we allow connections between singular points. A generalization seems natural in the case where each leaf carries at most finitely many singular points, but the condition and the precise statement are necessarily more complicated, and the foliation fails to be unique.}. We call such a foliations a \emph{$pA$-foliation} as this setting is motivated by the foliations associated to pseudo-Anosov homeomorphisms. 

The answer to Questions~\ref{q.induced} in the singular case are very similar to the one in the non-singular case: one just needs to adapt the two main hypotheses, the shell property and no bad-accumulations. The main difference is that Kaplan's work has been only stated in the non-singular setting, and an important difficulty of this paper is to provide the singular version of Kaplan's work. 

Let us start with the statement of the answers to Questions \ref{q.induced} and~\ref{q.completed}. 

\begin{itemize}
	\item \emph{shell/star pre-lamination}: any complementary region $\Delta$ of $G(L)$ is either:
	\begin{itemize} 
        \item a \emph{shell region}: $\Delta$ admits exactly one boundary component that is not a leaf of $L$ (the root of $L$),
    	\item or a \emph{star region}: the boundary $\partial \Delta$ is the union of finitely many leaves of $L$ which form an ideal polygon. 
	\end{itemize}
    Additionally, no leaf of $L$ belongs to the boundary of two distinct star regions. 
	\item \emph{no bad-accumulation}: for any sequence $\Delta_n$ of complementary regions that accumulates on a geodesic $\gamma$, we have:
    \begin{itemize}
        \item if the $\Delta_n$ are shell regions, then their roots also accumulates on $\gamma$, 
    	\item if the $\Delta_n$ are star regions, then there exist pairs of successive boundary components $(\gamma_{1,n},\gamma_{2,n})$ of $\Delta_n$ so that the $\gamma_{i,n}$ converge in $n$ toward $\gamma$.
	\end{itemize}
\end{itemize}

\begin{maintheorem}\thlabel{main-A-sing}
	Let $L$ be a pre-lamination of $S^1$ whose endpoints are dense in $S^1$ and with no more than countably many leaves ending at a given endpoints. 
	Then $L$ is the end pre-lamination $L_\FF$ of a $pA$-foliation~$\FF$ if and only if it is a shell/star pre-lamination and has no bad accumulation. Additionally, when it satisfies the previous condition, the singular foliation is unique up to homeomorphism.
\end{maintheorem}

Later on, a pre-lamination satisfying the hypotheses in the theorem will be called a $\Pre^*$-pre-lamination. 
A more formal statement is given in \thref{th-good-lam-are-end-sing-fol}. 


An important part of this paper consists of providing a singular version of Kaplan's work. Given a $pA$-foliations~$\FF$ on $\RR^2$, its leaf space $\Leaf(\FF)$ is a non-Hausdorff 1-manifold\footnote{Here we consider the leaf space of the restriction of~$\FF$ to $\RR^2\setminus \Sing(\FF)$, that is, we remove the singular points}. We equip it with an additional data, called singular planar structure, denote by $(\Leaf(\FF),\preceq_\cdot)$. 
We prove a similar theorem to Kaplan's theorems. 

\begin{maintheorem}\thlabel{mainth-Kaplan-sing}	
	Let $(\LL,\preceq_\cdot)$ be a singular planar structure, there exists a $pA$-foliation $\FF$ of $\Int(D^2)$, and an isomorphism $f\colon(\LL,\preceq_\cdot)\to(\Leaf(\FF),\preceq_\cdot)$ of singular planar structures. 
	
	Additionally, $\FF$ is unique, in the following sense. For any $pA$-foliation $\FF'$ on $\Int(D^2)$ and any isomorphism $f'\colon\left(\LL,\preceq_\cdot\right)\to(\Leaf(\FF'),\preceq_\cdot)$, there is an orientation preserving homeomorphism $h\colon\Int(D^2)\to\Int(D^2)$, conjugating $\FF$ to $\FF'$ and inducing an isomorphism $\wt h\colon(\Leaf(\FF'),\preceq_\cdot)\to(\Leaf(\FF'),\preceq_\cdot)$ that satisfies $\wt h\circ f=f'$.
\end{maintheorem}

As for Question \ref{q.completed}, a similar theorem to the non-singular case is certainly possible in the singular case. But the precise statement that one would need for any reasonable application is unclear. A difficulty is to decide whether a given complementary region should be completed in a shell or a star, which depends on the type of application one could have in mind. Although we do not answer the question, we provide enough tools for it.

\paragraph{Choice of leaf space.} The leaf space of a non-singular foliation is usually seen as a non-Hausdorff 1-dimensional manifold, whereas the leaf space of a singular foliation with no branching is often seen as a $\RR$-tree ($\RR$-trees are Hausdorff by definition). We had a choice of expressing our problem with two point of views: with a generalized non-Hausdorff $\RR$-tree or with a non-Hausdorff 1-manifold. We opted for the later, which we are more used too. It is essentially a matter of taste, and the argument would have been equivalent with the other language. 

\paragraph{Acknowledgment.} Théo Marty gratefully acknowledges the support of a postdoctoral position at the Université de Bourgogne, in Dijon, established in honor of Marco Brunella, which made this collaboration possible.

%% file: preliminaires.tex
\section{Preliminary notions} \label{sec-preliminar}

In this section, we introduce our terminology about ramified cover (used as a tool later on) and about foliations. We also state precisely the compactification of foliations on the plane.

\subsection{Ramified cover}
Let $S$ be a connected surface which is not homeomorphic to a sphere, $X\subset S$ be a closed discrete subset, and $\varphi\colon X\to \NN\setminus\{0,1\}$ be any function.
As $X$ is closed and discrete, we can fix a family of disjoint compact discs $D_x$ centered at $x\in X$, pairwise disjoint.

Fix an origin $\bp\in S\setminus\bigcup_{x\in X}D_x$. Let $G(X,\varphi)$ be the normal subgroup of $\pi_1(S\setminus X,\bp)$ generated (as a normal subgroup) by the generators $\{(\gamma_x)^{\varphi(x)}\}_{x\in X}$ where $\gamma_x$ is a loop around $x$ (the boundary of $D_x$). Clearly, each $\gamma_x$ depends on the choice of a path between the origin $\bp$ and $x$ but, as we consider the normal subgroup, the group $G(X,\varphi)$ does  not depend on the choices.

A normal cover only depends on a normal subgroup of the fundamental group. Thus $G(X,\varphi)$ determines a normal cover $\Pi_{X,\varphi}\colon \widetilde{(S\setminus X)}_\varphi\to S\setminus X$. Any connected component of $\Pi_{x,\varphi}^{-1}(D_x\setminus\{x\})$, for $x\in X$, is a punctured disc, and $\Pi_{x,\varphi}$ induces on that punctured disc a cyclic cover or degree $\varphi(x)$ over $D_x\setminus\{x\}$. Let us consider the completion of $\widetilde{(S\setminus X)}_\varphi$ obtained by adding all the punctures of these punctured discs. It yields a new surface $\widetilde{S_{X,\varphi}}$ and the projection $\Pi_{X,\varphi}$ extends in a unique way as a ramified cover
$$\Pi_{X,\varphi}\colon \widetilde{S_{X,\varphi}}\to S,$$
ramified over $X$ and which is a cyclic ramified  cover of degree $\varphi(x)$ on each connected component of
$\Pi_{X,\varphi}^{-1}(D_x)$.

The fundamental group of $\widetilde{(S\setminus X)}_\varphi$ is by construction the group $G(X,\varphi)$ and is generated by the loop around the punctured discs, connected components of $\Pi_{x,\varphi}^{-1}(D_x\setminus\{x\})$, for $x\in X$. Adding the punctures vanishes all this generators, so that the surface $\widetilde{(S\setminus X)}_\varphi$ is connected, simply connected and not compact, and so a plane. 

\begin{definition}\label{d.ramified} The projection $\Pi_{X,\varphi}\colon \widetilde{S_{X,\varphi}}\to S $  is called the \emph{universal cover of $S$, ramified along $X$ with degree $\phi$}.
\end{definition}

\subsection{Singular foliations: terminology}



Let~$\FF$ be a singular foliation on a surface $\RR^2$. That is~$\FF$ is defined outside a closed set of $S$, called the singular set and denoted by $\Sing(\FF)$. We denote by $\FF^*$ the (regular) foliation obtained as the restriction of~$\FF$ on $S\setminus \Sing(\FF)$.

In this paper we deal with singular foliations whose singular points are isolated, and are \emph{$k$-prongs singular points}, with $k\geq 3$,  equivalently called \emph{saddle points with $k$ separatrices}. Let us describe precisely these singularities.

Consider $(-1,1)^2$ endowed with the horizontal foliation $H$ restricted to $(-1,1)^2\setminus{(0,0)}$, and consider the quotient $O_0=(0,1)^2/(r,s)\simeq(-r,-s)$ endowed with the quotient foliation $H_0$. Denote by $\omega\in O_0$ the image of $(0,0)$ by the projection. The triplet $(O_0, H_0, \omega)$ is called the model \emph{$1$-prong singular point}.

An isolated singular point $x$ of a singular foliation~$\FF$ is a $k$-prong singular point (for $k\geq 1$) if there is a neighborhood $O_x$ of $x$ and a ramified cover $\pi\colon O_x\to O_0 $ with $\pi(x)=\omega$, of degree $k$ on $\omega$, and so that the restriction of~$\FF$ to $O_x\setminus \{x\}$
is the lift by $\pi$ of the foliation $H_0$ on $O_0\setminus \{\omega\}$.  A leaf of $\FF^*$  is \emph{a separatrix} of $x$ if it contains a connected component of the preimage by $\pi$ of the projection of $(0,1)\times\{0\}$ on $O_0$.

A singular foliation on $\RR^2$ is called \emph{pseudo-Anosov like}, shortly \emph{is a pA-foliation}, if the following hold:
\begin{itemize}
	\item every singular point $x\in S$ of~$\FF$ is a $k$-prongs singularity for some $k\geq 3$,
	\item no two distinct singularity share a same separatrix (we say there is no saddle connection).
\end{itemize}

According to the terminology used in pseudo-Anosov homeomorphisms or flows, a \emph{singular leaf} $\ell$ of a pA-foliation~$\FF$ of $\RR^2$ is the union of a singular point $p$ with all its separatrices.
The separatrices in $\ell$ are naturally cyclically ordered and the notion of successive separatrices of $p$ is well-defined. A \emph{face} of the singular leaf $\ell$ is the union of a singular point $\{p\}$ with two of its successive separatrices.

A ray is the image of a proper embedding of $[0,+\infty)$ in $\RR^2$. A \emph{ray of~$\FF$} is a ray contained in a leaf (singular or not).  An \emph{end of leaf} of~$\FF$ is a germ of a ray of~$\FF$ at infinity. That is two rays of~$\FF$ are equivalent at infinity if they contain a common ray. A end of leaf of~$\FF$ is the equivalent class of rays of~$\FF$. For pA-foliations, a classical consequence of Poincar\'e-Bendixson argument implies:

				\begin{itemize}
					\item a non-singular leaf of~$\FF$ has two ends
					\item a separatrix of~$\FF$ has one end
					\item a singular leaf of~$\FF$ has as much ends as separatrices
					\item a face of a singular leaf has two ends.
				\end{itemize}

				An important consequence of the fact that the Poincar\'e-Hopf index of the singular points are all strictly negative is :

				\begin{lemma}\label{l.1point}
					Let $\sigma\colon [-1,1]\to \RR^2\setminus \Sing(\FF)$ be transverse to the~$\FF$. Then, for every (singular or not) leaf $f$ of~$\FF$ the intersection $\sigma\cap f$ consists of at most one point.
				\end{lemma}

				\paragraph{Non-separated leaves, regular leaves.}

				Let~$\FF$ be a $pA$-foliation as above and $f$ be a non-singular leaf of~$\FF$. Let us fix a transverse curve $\sigma\colon\RR\to\RR^2\setminus \Sing(\FF)$ to~$\FF$, with $\sigma(0)\in f$. Up to a countable set of $t\in [0,1]$, the point $\sigma(t)$ and $\sigma(-t)$ belongs to non-singular leaves $f_t$ and $f_{-t}$ of~$\FF$ (the set of singular leaves is at most countable).
				For every $t>0$ for which $f_t$ and $f_{-t}$ are non-singular,  we denote by $U_t(f),U^+_t(f)$, and $U^-_t(f)$ the closures of the connected components of $\RR^2\setminus (f_{t}\cup f_{-t})$, $\RR^2\setminus (f_{t}\cup f)$, and $\RR^2\setminus (f_{-t}\cup f)$ bounded by both $f_t$
				and $f_{-t}$, $f_t$ and $f$, $f_{-t}$ and $f$, respectively.

				We denote
				$$K(f)=\bigcap_{t>0}U_t,\quad K^+(f)=\bigcap_{t>0}U^+_t , \mbox{and } K^-(f)=\bigcap_{t>0}U^-_t$$
				Then $K(f),K^+(f),K^-(f)$ are three closed non-empty~$\FF$-saturated subsets of $\RR^2$ and $K(f)$ is independent of the choice of the transverse segment $\sigma$. Additionally, $K^+(f)$ and $K^-(f)$ only depend on the orientation of $\sigma$: a change of orientation of $\sigma$ just exchanges $K^+(f)$ with $K^-(f)$.

				One says that a non-singular $f$ is \emph{regular} if $f=K(f)$. One defines also the regularity on one side of $f$:   $f$ is regular on the  side of $\sigma((0,1])$ or of $\sigma([-1,0))$ if  $f=K^+(f)$ or $f=K^-(f)$, respectively.

If $f$ is a face of a singular leaf $F$ of the foliation~$\FF$, then one side of $f$ contains $F\setminus f$ and the other side is disjoint from $F$. Let us call that other side \emph{the free side of the face $f$}. Let $\sigma$ be a segment transverse to~$\FF$ so that $\sigma(0)\in f$
and $\sigma(1)$ is contained in the free side of $f$.
For $t\in (0,1]$ so that the leaf $f_t$ through $\sigma(t)$ is non-singular, one defines $U^+_t$ as the closure of the connected component of $\RR^2\setminus (f_{t}\cup f)$ bounded by both $f_t$ and $f$, and we denote
$$K^+(f)=\bigcap_{t\in(0,1]}U_t.$$
Then $K^+(f)$ is independent of $\sigma$. The face $f$ of the singular leaf $F$ is \emph{regular on its free side} if $K^+(f)=f$.

One says that a sequence $f_n$ of non-singular leaves or faces of leaves \emph{accumulates} on a non-singular leaf $f$ if, for $n$ large enough, $f_n$ intersects any transverse segment $\sigma$ with $\sigma(0)\in f$ at points $\sigma(t_n)=\im(\sigma)\cap f_n$ with $t_n\to 0$. This notion does not depend on the choice of the segment $\sigma$.
One says that a sequence $f_n$ of non-singular leaves or faces of  singular leaves \emph{accumulates} on a face $f$ of a singular leaf if, for $n$ large enough, $f_n$  is contained in the free side of $f$ and intersects a transverse segment $\sigma$ with $\sigma(0)\in f$ at points $\sigma(t_n)$ with $t_n\to 0$. This notion for $f_n$ on the free side of $f$ does not depend on the choice of the segment $\sigma$.

Two non-singular leaves or faces of singular leaves $f$, $g$ of the foliation~$\FF$ are \emph{non-separated} if there is a family of non-singular leaves $f_n$ accumulating on both $f$ and $g$. A leaf or face of leaf $f$ \emph{is non-separated} if there if  a leaf or face of leaf $g$ so that $f$ and $g$ are not separated.

As a direct consequence of Lemma~\ref{l.1point} one gets:

\begin{lemma}\label{l.separated}
	If two non-singular leaves or faces of singular leaves $f$, $g$ of~$\FF$ are non-separated, then given any transverse segment $\sigma\colon\RR\to\RR^2\setminus \Sing(\FF)$, the intersection $\sigma\cap (f\cup g)$ consists of at most one point.
\end{lemma}

\begin{proof}
	A transverse segment cutting both $f$ and $g$ would cut the leaves accumulating $f$ and $g$ in $2$ points contradicting Lemma~\ref{l.1point}.
\end{proof}

When $f$ and $g$ are not separated,  any sequence $f_n$ accumulating on both $f$ and $g$ are always contained (for $n$ large) on one side of $f$ and one side of $g$. One says that \emph{$f$ is non separated from $g$ on that side}.

\begin{lemma}\label{l.regular-separated}
	A non-singular leaf $f$ is regular on one side  if and only if, on that side,  it is separated from any other leaf or face of singular leaf.

	A face $f$ of a singular leaf is regular on its free side if and only if it is separated from any other leaf or any face of other singularities.
\end{lemma}

\begin{proof}
	Assume first that  $f$ is non-separated from $g\neq f$. Let  $f_n$ be a sequence of leaves accumulating on both $f$ and $g$ and  contained in one side (by convention the positive side) of $f$. Denote by $\sigma$ a open transversal intersecting $f$ at $0$, and $U_t$ be defined as above. Then, for every $t>0$ and for $n$ large enough, $f_n$ is contained in $U_t^+$. Therefore $g$ is contained in $U_t^+$. Thus, $g\subset K^+(f)$, implying that $f$ is not regular on that side.

	Reciprocally, if $K^+(f)\neq f$ then $K^+(f)$ is saturated for~$\FF$, so that its boundary consists of non-singular leaves or faces of~$\FF$.  Every leaf or face $g$ contained in that boundary is accumulated by the leaves $f_t$ bounding $U^+_t$, which also accumulates on $f$: thus  $f$ is not separated from~$g$.
\end{proof}

\subsection{The circle at infinity of a foliation}

Given a $pA$-foliation~$\FF$ on the plane $\RR^2$, we give its compactification at infinity.

\begin{theorem}[\cite{Mather,BonattiCirc}] \thlabel{th-Bonatti}\label{th-Bonatti}
	Let~$\FF$ be a (possibly singular with $k$-prongs singularities, $k\geq 3$) foliation on $\RR^2$. Then there exists a embedding $\psi\colon \RR^2\to D^2$ with $\psi(\RR^2)=\Int(D^2)$, unique up to a homeomorphism of the disc $D^2$, which satisfies the following:
	\begin{enumerate}
		\item any germ of rays $\rho$ of $\psi(\FF)$ converges toward a unique point $\xi_\rho$ in $\partial D^2$,
		\item given $\xi\in\partial D^2$, there is at most countably many germs of rays $\rho$ of $\psi(\FF)$ with $\xi_\rho=\xi$,
		\item for any non-empty open subset $O\subset\partial D^2$, the set of germs of rays $\rho$ of $\psi(\FF)$ with $\xi_\rho\in O$ is uncountable.
	\end{enumerate}
\end{theorem}

The proof of the theorem consists in building an abstract circle $\partial_\infty \FF$ and putting a natural topology on $\RR^2\coprod \partial_\infty \FF$ so that it is a disc. The circle $\partial_\infty \FF$ is obtained as follows: the set of ends  of leaves of~$\FF$ is naturally cyclically ordered. One identifies two ends if one of the intervals bounded by them (for this cyclic order) is at most countable. This quotient is in an increasing (for the cyclic order) bijection with a dense subset of the circle $S^1$, and this bijection is unique up to an orientation-preserving homeomorphism of $S^1$.

\begin{remark}\label{r.same-end}
	The second item in Theorem~\ref{th-Bonatti} implies:
	\begin{itemize}
		\item two distinct ends of the same leaf (singular or not) tend to distinct points of $S^1.$
		\item two distinct leaves may share at most one limit point in $S^1$.
	\end{itemize}
\end{remark}


We denote by $\pi=\pi_\FF$ the map which associates the pair of endpoints $\pi(f)$ in $(S^1\times S^1 \setminus\Diag )_{/(x,y)\simeq (y,x)}$ to each non-singular leaf or face of singular leaf $f$ of~$\FF$. Remark~\ref{r.same-end} implies that $\pi_\FF$ is well-defined and injective.

We now assume that the foliation~$\FF$ lies on the interior of $D^2$, so that its compactification as above coincide with $D^2$.

\begin{lemma}\label{l.continuous}
	Let $f$ be a non-singular leaf or a face of singular leaf, and $\sigma\colon[0,1)\to \Int(D^2)$ a transversal to~$\FF$ so that $\sigma(0)$ lies in $f$. Let $f_t$ denotes the non-singular leaves through $\sigma(t)$, if exists, and $f_0=f$ . The map
	$$t\mapsto \pi(f_t)\in(S^1\times S^1\setminus\Diag)_{/\simeq}$$
	is continuous at $t=0$ if and only if $f$ is regular on the side of $\sigma$ (and $\sigma$ lies on the free side of $f$ if $f$ is a face of a singular leaves).
\end{lemma}

\begin{proof}
	Consider the sets $U^+_t$ bounded by $f$ and by $f_t$ as defined previously. The adherence of $U^+_t$ on $\partial D^2$ is the union of two intervals, each of them bounded by one end point of $f$ and one end point of $f_t$.
	Assume that $t\mapsto\pi(f_t)$ is continuous, that is the endpoints of $f_t$ tend to those of $f$. Then the boundary of $K^+(f)=\bigcap_t U^+_t$ in $\partial D^2$ is exactly the two end points of $f$. Thus the leaves in $K^+(f)$ share their endpoints with $f$, and this implies that $K^+(f)=f$.

	Conversely, assume $K^+(f)=f$. Denote by $I_t,J_t$ the two intervals in $\partial U^+_t\cap\partial D^2$ as discuss above, so that $I_t$ contains a fixed end point of $f$ (independent on $t$). Then $I_t$ is increasing in $t$, so it converges toward $\cap_tI_t$ when $t$ goes to zero. Assume that $\cap_tI_t$ has a non-empty interior. Then there exists a leaf $g\neq f$ of~$\FF$ that has an end inside the interior of $\cap_tI_t$. The leaf $g$ does not belong to $K^+(f)=f$, so it does not belong to $U^+_t$ for some $t$. It follows that its end do not belong to $I_t$, which contradicts the assumption. Therefore, $\cap_tI_t$ and similarly $\cap_tJ_t$ are both reduced to a point. So $t\mapsto f_t$ is continuous.
\end{proof}



%% file: planar.tex
\section{Planar structure}

We introduce here the notion of planar structure: roughly speaking, a non-Hausdorff 1-manifold with order data that encode a (maybe singular) foliation on the plane.

\subsection{Tame non-Hausdorff $1$-manifold\label{ss.tame}}
A \emph{a non-Hausdorff $1$-manifold} $\LL$ is a topological space which admits a countable basis of its topology and so that every point admits a neighborhood homeomorphic to $\RR$. Any subset of $\LL$ homeomorphic to $\RR$ will be called \emph{a chart of $\LL$} or equivalently \emph{an open interval of $\LL$}. Note that we do not actually assume that $\LL$ is non-Hausdorff (pair of points may not admit disjoint neighborhood), we emphasis on the fact that it may be non-Hausdorff, contrary to usual manifolds.

The example below is what we don't want to consider:

\begin{example}
	Consider the usual diadic Cantor set $\cC\subset\RR$. Let $\varphi$ be a homeomorphism of $\RR\setminus \cC$. 
	Let $\LL$ be the quotient of $\RR\times\{0,1\}$ obtained by identifying $(t,0)$ with $(\varphi(t),1)$ for $t\notin \cC$.
	Then $\LL$ is a non-Hausdorff 1-manifold, and every point in $\cC$ is a non-separated point, in particular there are uncountably many non-separated points.
\end{example}

If $\LL$ is a non-Hausdorff $1$-manifold and $x\in\LL$. A \emph{side of $x$} is a germ at $x$ of connected component of $I\setminus\{x\}$ where $I$ is an open interval that contains $x$: a connected component of $I\setminus \{x\}$ define the same side as a connected component of $J\setminus \{x\}$ if they coincide in a neighborhood of $x$. Thus each point as two sides.

\begin{definition}\label{d.tame}
	A non-Hausdorff $1$-manifold $\LL$ is \emph{tame} if given any two points $x,y\in\LL$ which are not separated one from the other there are charts $I_x,I_y$ that contain respectively $x$ and $y$, so that $I_x\cap I_y$ is a connected component of both $I_x\setminus \{x\}$ and $I_y\setminus \{y\}$.
	In other words, $x$ is not separated from $y$ if they have a common side, their other sides being disjoint. One says that $x$ is not separated from $y$ on their common side.
\end{definition}

On tame non-Hausdorff $1$-manifold, to be non-separated on one side induces an equivalence relation on the set of pairs points/side of the point. A \emph{branching} of $\LL$ is an equivalence class of the non-separating relation on a given side. More precisely: consider the set of pairs $(x,s)$ where $x\in\LL$, and $s$ is a side of $x$. The relation $(x,s)\sim(x',s')$ if $x$ and $x'$ are non separated on their side $s$ and $s'$, is an equivalence relation. Let $X$ be an equivalence class of this relation $\sim$. Then the set $\{x\in\LL, \exists s \text{ side of } x, (x,s)\in X\}$ is \emph{a branching}.
A point $x$ may belong to at most $2$ branchings, corresponding to the two sides of $x$.

\begin{lemma}\label{l.countable}
	Let $\LL$ be a tame non-Hausdorff 1-manifold.
	Then the set of non-separated points of $\LL$ is at most countable.
\end{lemma}

\begin{proof}
	From the tame definition, every non-separated pair of points in $\LL$ is contained in a pair of intervals containing no other non-separated pair of points. One may choose these intervals in a countable basis of the topology of $\LL$. Thus, these set of pairs of intervals is countable, and (a fortiori) so is the set of pairs of non-separated points. Hence the set of non-separated points is at most countable
\end{proof}

Branching have been used to characterize non-separated leaves of (non-singular) foliations. We introduce a similar notion to characterize separatrices of $pA$-foliations.
A \emph{cyclic branching} is a cyclically ordered finite set $\nu=\{x_i\}_{i\in\ZZ/n\ZZ}$, $x_i\in\LL$, so that:

\begin{itemize}
	\item for all $i$, $x_i$ and $x_{i+1}$ are on a common branching $\mu_i$,
	\item $\mu_{i-1}\neq \mu_i$ that is, $\mu_{i-1}$ and $\mu_i$ are branching in distinct sides of $x_i$.
	\item given $j\notin\{i-1,i,i+1\}$, $x_i$ and $x_j$ do not belong to a common branching.
\end{itemize}

One denotes by $\preceq^\circlearrowleft_\nu$ the cyclic order on $\nu$ given by the usual cyclic order of the indexation $\ZZ/n\ZZ$.

The \emph{degree} of a cyclic branching $\nu$ is its cardinal.
A \emph{separatrix} is any point that lies on a cyclic branching.

\subsection{(Non-singular) planar structure}

We introduce in this section and the next the notion of planar structure, that encode the leaf space of a foliation. We use that a connected non-Hausdorff $1$-manifold is simply connected if and only if every $x\in\LL$ disconnects $\LL$ (see \cite[p.113]{haefliger1957}.

\begin{lemma}\label{l.simply-tame}
	If $\LL$ is a simply-connected $1$-manifold, then the intersection of any two charts is either empty or is a chart. As consequences, $\LL$ is tame and the intersection of any two distinct branchings consists in at most one point.
\end{lemma}

\begin{proof}
	Let $I$ and $J$ be two charts of $\LL$ and assume that $I\cap J$ has (at least) two connected components $U,V$. Let $a\in I$ be the endpoint of $U$ on the side of $V$ (as see from inside $I$). In particular $a$ is not inside $J$. Now $I\cup J\setminus \{a\}$ is connected: indeed the two connected components of $I\setminus \{a\}$ are the one of $U$ and the one of $V$, but $U$ and $V$ are contained in the connected set $J$.
	Thus $a$ does not disconnect $\LL$, which contradicts that $\LL$ is simply connected. It follows that $\LL$ is tame.

	If two distinct points $x,y$ belong to two distinct branchings $\mu_1,\mu_2$, then these branchings correspond to the two sides of $x$ and of $y$. Thus both sides of $x$ and $y$ coincide: charts at $x$ and $y$ intersect on both sides, contradicting the first part of the statement.
\end{proof}

\begin{lemma}
	A non-Hausdorff simply connected $1$-manifold has no cyclic branching.
\end{lemma}

\begin{proof}
	Assume that $\{x_i\}_{i\in \ZZ/k\ZZ}$ is a cyclic branching. Then, $x_{0}$ and $x_2$ are in distinct connected component of $\LL\setminus \{x_1\}$. But, for any $j\notin\{0,1\}$, $x_j$ and $x_{j+1}$ belong to the same connected component of $\LL\setminus\{x_1\}$ since they are not separable. Thus all the $x_j$ ($j\neq 1$) belong to the same component leading to a contradiction.
\end{proof}

An \emph{ordered branching} is a branching $\mu$ equipped with a total order $\leq_\mu$.

\begin{definition}
	A \emph{planar structure} is the data $(\LL,\{\leq_\mu,\mu \text{ branching}\})$ where $\LL$ is a connected, simply-connected, non-Hausdorff $1$-manifold $\LL$ and $\leq_\mu$ is a total order on $\mu$ for every branching $\mu$.
\end{definition}

An \emph{isomorphism of planar structures} $(\LL,\{\leq_\mu,\mu\})\xrightarrow{f}(\LL',\{\leq_{\mu'},\mu'\})$ is a homeomorphism $\LL\xrightarrow{f}\LL'$ so that, for every branching $\mu$ of~$\LL$, the restriction $f_{|\mu}\colon(\mu,\leq_\mu)\to (f(\mu),\leq_{f(\mu)})$ is increasing. In the rest, we will write $(\LL,\leq_\cdot)$ for shortness.

\begin{remark}
	Any simply connected non-Hausdorff $1$-manifold is orientable.
\end{remark}

A planar structure $(\LL,\leq)$ is \emph{oriented} if $\LL$ is endowed with an orientation.
The properties of planar structures will be described in next sections as a particular case of singular planar structures.

\subsection{Singular planar structure}

We introduce singular planar structure. Think of them as the leaf space of the $pA$-foliations with $\geq3$-prong singularities and no connection.

\begin{definition}\label{d.singular-planar}
	A \emph{singular planar structure}
	is a tame connected non-Hausdorff $1$-manifold $\LL$, with a total order $\preceq_\mu$ on every branching $\mu$, a total cyclic order on every cyclic branching of $\LL$, which satisfies the following properties:
	\begin{enumerate}
		\item the cyclic branchings are pairwise disjoint,
		\item the cardinal of every cyclic branching is $\geq 3$,
		\item for every cyclic branching $\nu$ and every branching $\mu$, $\mu\cap\nu$ is either empty or it has two points. If $\mu\cap\nu=\{x,y\}$ then $x$ and $y$ are consecutive for the order $\preceq_\mu$ on $\mu$ and for the cyclic order $\preceq^{\circlearrowleft}_\nu$. Furthermore, if $y$ is the successor of $x$ for $\preceq_\mu$ it is also the successor of $x$ for $\preceq^\circlearrowleft_\nu$.
		\item every cyclic branching of order $k$ disconnects $\LL$ in $k$ connected components,
		\item every non-separatrix point $x\in\LL$ (i.e. $x$ not in a cyclic branching) disconnects $\LL$ in two connected components.
	\end{enumerate}
\end{definition}

We start determining the intersection of intervals in a singular planar structure.

\begin{lemma}\label{l.|A|}
	Let $\nu$ be a cyclic branching of a singular planar structure $\LL$. For any non-empty subset $A\subset\nu$, $\LL\setminus A$ has $|A|$ connected component.
	Furthermore, two separatrices $x,y$ of $\nu\setminus A$ are in the same connected component of $\LL\setminus A$ if and only if there are
	$x=x_0,x_1\dots,x_k=y$ in $\nu\setminus A$ so that $x_i$ is adjacent to $x_{i+1}$ in $\nu$.
\end{lemma}

\begin{proof}
	It holds true by definition for $A=\nu$. By definition of $1$ dimensional (no Hausdorff) manifold, any point admits a neighborhood that is an interval. Hence it cuts this interval in two components: thus removing a point adds either $0$ or $1$ components, depending if the two local components belong to the same global component or not. If $|A|=1$, that is $A$ is a unique point $x_i$ of $\nu$, $x_i$ is no separated from $x_{i-1} $ from one side and from $x_{i+1}$ in the other side. Thus $x_i$ does not disconnect $\LL$, but breaks the cyclic order.
	Now, Lemma~\ref{l.|A|} follows from a direct induction on $|\nu|-|A|$: each new point of $|A|$ that one removes adds a component.

	Let $x,y\in \nu\setminus A$ be two adjacent separatrices. They share a side disjoint from $A$, hence they are in the same connected component of $\LL\setminus A$. Thus the same holds if $x,y$ are connected by a chain of adjacent separatrices in $\nu\setminus A$.

	Reciprocally, assume that $x,y$ do not satisfy the hypothesis. Then there exists $z_1,z_2\in\nu\setminus A$ so that for the cyclic order on $\nu$, $x$ and $y$ are on distinct interval $B_1,B_2\subset \nu$ bounded by $z_1$ and $z_2$. It follows from above that $\LL\setminus \{z_1,z_1\}$ has two connected components, each on them contains exactly one of $B_1$ and $B_2$. 
	Thus $x,y$ belong to distinct connected components of $\LL\setminus\{z_1,z_2\}$ and hence of $\LL\setminus A$, concluding the proof.
\end{proof}

\begin{lemma}\label{l.single-intersection}
	Let $\LL$ be a singular planar structure and $\nu$ a cyclic branching on $\LL$. Then any interval of $\LL$ intersect $\nu$ in at most one point.
\end{lemma}

\begin{proof}
	Write $\nu =\{x_1,\cdots,x_n\}$ given with the cyclic order.
	From Lemma \ref{l.|A|}, we know that $\nu$ separates $\LL$ in $n$ connected components, each delimited by two consecutive $x_i$. Let $A_i$ be the connected component of $\LL\setminus\nu$ that contains the a side of both $x_i$ and $x_{i+1}$. A curve that starts in $A_i$ either remains in $A_i$ or intersects $\{x_i,x_{i+1}\}$.

	Let $I$ be an  interval in $\LL$ and assume it intersects $\nu$ in two points $x_i,x_j$. Denote by $\gamma$ the (injective) curve inside $I$ that starts on $x_i$ and ends on $x_j$. Up to changing the order in $\nu$, we may assume that $\gamma$ starts inside $A_i$. Then $\gamma$ has to leave $A_i$ through $x_{i+1}$. Note that any neighborhood of $x_i$ inside $A_i$ is also a neighborhood of $x_{i+1}$ inside $A_i$. Thus $\gamma$ self intersects just after $x_i$ and just before $x_{i+1}$. It contradicts the injectivity of $\gamma$. Therefore $I$ can not intersect $\nu$ twice.
\end{proof}

\begin{lemma}\label{l.intersection-connexe}
	Let $I$ and $J$ be two intervals of a singular planar structure $\LL$. Then $I\cap J$ is either empty or an interval.
\end{lemma}

\begin{proof}
	We argue by contradiction assuming that $I\cap J$ has (at least) two connected components $I_1$ and $I_2$. Take any point $z\in I\setminus (I\cap J)$, between $I_1$ and $I_2$. By construction, $z$ does not disconnecting $I\cup J$ and therefore does not disconnect $\LL$. Thus, by definition of a singular planar structure, $z$ belongs to a cyclic branching $\mu$.
	According to Lemma \ref{l.single-intersection}, $\mu$ intersects each interval $I$ and $J$  on at most one point. Since $\mu$ has at least three points, it has at least one point $p$ outside $I\cup J$.

	Lemma \ref{l.|A|} states that $\{z,p\}$ cuts $\LL$ in two connected components $A_1,A_2$.
	The set $(I\cup J)\setminus\{z\}$ is connected, so it lies inside one of the sets $A_i$, let us say~$A_1$. Since $I\cup J$ is a neighborhood of $z$, the set $A_1\cup\{z\}$ is open. So $\{A_1\cup\{z\},A_2\}$ is a partition of $\LL\setminus\{p\}$ in disjoint open subsets. It follows that $\LL\setminus\{p\}$ is not connected, which contradicts Lemma \ref{l.|A|}.
\end{proof}

\begin{lemma}\label{l.presque-connexe}
	Let $\LL$ be a singular planar structure and $C\subset\LL$ be a connected open set, union of a finite family $J_1,\dots J_k$ of open intervals. Let $I$ be an interval so that $I\cap C\neq \emptyset$. Then $I\cap C$ is of the type of $U\setminus \{x_1,\dots ,x_k\}$ where $U$ is an interval, and $x_i$ are each a point in a cyclic branching $\nu_i$ that satisfies $\nu_i\setminus\{x_i\}\subset C$.
\end{lemma}

An example of the lemma is to take $C$ a neighborhood of a cyclic branching minus one point $x$ on that cyclic branching, and to take $I$ to be a small interval that contains $x$. Then $I\cap C$ is equal to $I\setminus x$.

\begin{proof}
	According to Lemma~\ref{l.intersection-connexe} $I\cap J_i$ is an interval $U_i$ (may be empty). Given a point $x\in I$ between (in $I$) two successive $U_i$, $x$ does not disconnect $C\cup I$, and \emph{a fortiori} neither $\LL$. Hence $x$ is on a cyclic branching. Thus, there are at most countably many possible points $x$. Since the intervals $U_i$ comes in finite amount, there are finitely many points $x$ as above, denote them by $\{x_1\cdots x_k\}$ (potentially $k=0$).

	The set $U=\cup_iU_i\cup\{x_1\cdots x_k\}$ is thus an open interval, and $\cup_iU_i$ is equal to $U\setminus\{x_1\cdots x_k\}$. Denote by $\nu_i$ the cyclic branching that contains $x_i$. As the $\{x_1\cdots x_k\}$ does not disconnect $C\cup I$, the other points of $\nu_i$ belong to $C\cup I$. Since an interval intersect a cyclic branching on at most one point (see Lemma~\ref{l.single-intersection}), $\nu_i\setminus\{x_i\}$ belong to $C$, concluding the proof.
\end{proof}

\begin{remark} \label{r.planar}
	If $\LL$ has at least one cyclic branching, then $\LL$ is not simply connected, as each point in a cyclic branching is non-disconnecting.
\end{remark}

If $J$ is the connected components of $I\setminus x$ that accumulates on a given branching $\mu$, we say that $\mu$ \emph{is a branching on the side $J$} of $x$.

\subsection{Universal planar structure}
An oriented planar structure $(\LL,\preceq_\cdot)$ is \emph{a universal planar structure} if given any oriented planar structure $(\tilde\LL,\tilde\preceq_\cdot)$ there is an injective morphism
$\varphi\colon\tilde\LL\to\LL $, increasing for the orientations of $\tilde\LL$ and $\LL$ and increasing for the orders $\tilde\preceq_{\tilde\mu}$ and $\preceq_\mu$ in restriction to any branching $\mu$ of $\LL$ and $\tilde \mu$ of $\tilde\LL$ with $\varphi(\tilde \mu)\subset \mu$.

Notice that any countable totally ordered set $(X,\leq)$ can be increasingly embedded into $\QQ$. Call a countable ordered set $(X,\leq_X)$ \emph{maximally ordered} (or call $\leq_X$ a maximal order) if there exists an increasing bijection onto $(\QQ,\leq)$.

\begin{definition}\label{d.maximally}
	A planar structure $\LL$ is called \emph{maximally branched} if $\LL$ satisfies the three following properties:
	\begin{itemize}
		\item the union of its branchings is dense,
		\item every non-separated point in $\LL$ belongs to two branching (one on each side),
		\item and every branching is maximally ordered.
	\end{itemize}
\end{definition}

\begin{theorem}\thlabel{th-univ-leaf}
	Any maximally ordered planar structure is universal.
\end{theorem}

We build a universal planar structure in Section \ref{sec-univ-fol}.

We say that $U\subset\LL$ is a \emph{line} if it is an open interval which is maximal for the inclusion. That is no other open interval contains $U$.

\begin{proof}
	Let $\LL_1,\LL_2$ be two oriented planar structures and assume $\LL_2$ to be maximally ordered. We need to prove that there exists an increasing embedding $\LL_1\hookrightarrow\LL_2$.

	For some set $N=\NN$ or $N=\intint{0,N}$, let $(U_i)_{i\in N}$ be a family of lines of $\LL_1$ that covers $\LL_1$, and denote by $f_i\colon U_i\to\RR$ an increasing homeomorphism, for all $i\in N$.
	Denote by $V_i=\cup_{j<i}U_j$, for $i>0$. Up to reordering the family $(U_i)_i$ and removing some elements (while keeping $\cup_iU_i=\LL_1$), we can assume that $V_i$ is connected, it intersects $U_i$ but does not contain $U_i$.
	As a straightforward consequence of Lemma~\ref{l.presque-connexe}, $U_i\cap V_i$ is an open interval.


	We build the map $h\colon\LL_1\to\LL_2$ inductively on $V_i$. Let us denote by $B$ the set of non-separated points in $\LL_2$, and by $A$ the set of non separated points in~$\LL_1$.

	For $i=1$, $V_1=U_0$, take a line $U\subset\LL_2$ and an increasing homeomorphism $f_1\colon U_0\to U\subset\LL_2$ so that $f_1$ sends $A\cap U_0$ into the $B\cap U$. This is possible since $A\cap U_0$ is at most countable and the set $B\cap U$ is a countable dense subset of $U$, and so ordered maximally.

	Assume that an increasing embedding $h_i\colon V_i\to\LL_2$ has been built, sending $A\cap V_i$ inside $B$. Since $V_i$ and $U_i$ are open, $I\coloneqq U_i\cap V_i$ is an open sub-interval of $U_i$, and is not all $U_i$ by assumption.

	Suppose first that for some increasing homeomorphism $f\colon U_i\to\RR$, one has $f(I)=]-\infty,0[$. Denote by $a=f^{-1}(0)$ the boundary point of $I$ inside $U_i$. Since all $U_j$ with $j<i$ are line (so maximal), the sequence $f^{-1}(t)\in I$ with $t\to 0^-$ accumulates on a non-empty subset of a branching $\mu\subset B\subset\LL_1$, and the point $a$ belongs to $\mu\subset B$. Then by assumption on $h_i$, $h_i(V_i\cap\mu)$ is contained in a negative branching $\nu$ of $\LL_2$.

	We first extend $h$ on $V_i\cup\{a\}$ as follows:
	The set $V_i\cap\mu$ is finite since all line of $\LL_1$ intersects $\mu$ at most once, by simply-connectedness of $\LL_1$. Since the order on $\nu$ is maximal, one can find a point $b\in \nu$ which satisfies that the map $g\colon(V_i\cap\mu)\cup\{a\}\to\nu$, defined by $g(a)=b$ and $g(x)=h(x)$ for $x\in V_i\cap\mu$, is an increasing bijection.

	Denote by $W$ a positive ray of $\LL_2$ that start through $b$, that is a maximal subset of $\LL_2$ homeomorphic to $[0,+\infty[$ where $b$ corresponds to $0$. Using the same argument as before, one can find an increasing homeomorphism $U_i\setminus I\to W\subset\LL_2$ that send $a$ to $b$, $A\cap (U_i\setminus I)$ into $B\cap W$, and use it to extends the map $h_i$ to $h_{i+1}\colon V_{i+1}\to\LL_2$. Since $\LL_2$ is simply-connected, $h_{i+1}$ is injective. Therefore, it is an increasing embedding, and by construction it sends $A\cap V_{i+1}$ into $B$.

	When $U_i\setminus I$ is a negative ray or the union of two rays (positive and negative), one proceed similarly to extends $h_{V_i}$ to $V_{i+1}$.

	We build inductively the maps $h_i\colon V_i\to\LL_2$. As $h_i$ and $h_{j}$ coincide on $V_i\cap V_j$, the union of all maps $h_{V_i}$ yields a well-defined increasing embedding $\LL_1\hookrightarrow\LL_2$.
\end{proof}

\subsection{Universal singular planar structure}
Let $\LL$ be a singular planar structure and $U\simeq\RR$ be an oriented open interval of $\LL$ and $x\in U$ a separatrix. Denote by $\nu=(x_k)_{k\in\ZZ/n\ZZ}$, with $x_1=x$, the compatible cyclic branching which contains $x$. We say that
\begin{itemize}\item $\nu$ is \emph{on the left} of $U$ if when $y\in U$ converges toward $x$ from below, $y$ also converges on $x_0$ (think of the cyclic order on $\nu$ as going anti-clockwise).
	\item $\nu$ is \emph{on the right} of $U$ if $y$ also converges toward $x_2$.
\end{itemize}
Notice that it depends only on the germ of $U$ at $x$ and on its orientation and not on the precise choice of interval.

\begin{definition}\label{d.pre-universal-planar} Let $\LL$ be a singular planar structure. Denote by $A\subset\LL$ the union of all branching, $B_k$ ($k\geq3$) the union of all cyclic branching of order $k$, and $B_0 =A\setminus\cup_{k\geq 3}B_k$. singular planar structure $\LL$ is said \emph{pre-universal} if :
	\begin{enumerate}
		\item every element in $A$ is on a branching on its two sides,
		\item for any two points $a_1\prec_\mu a_2$ on a common branching $\mu$ and that are not on a common cyclic branching, the set $\{x\in\mu,a_1\prec_\mu x\prec_\mu a_2\}$ intersects all~$B_i$,
		\item for every branching $\mu$ and every $i$, $\mu\cap B_i$ has no maximum and no minimum for $\preceq_\mu$,
		\item for every oriented charts $U\simeq\RR$ in $\LL$, the set $B_0\cap U$ is dense in $U$,
		\item for every oriented charts $U\simeq\RR$ in $\LL$ and every $k\geq3$, the two sets of separatrices in $B_k\cap U$ that lies on the left/right of $U$ are each dense in $U$.
	\end{enumerate}

\end{definition}

\begin{definition}\label{d.universal-planar}
	A singular planar structure $\LL$ is said \emph{universal} if for any singular planar structure $\LL'$, there exists an embedding $\LL'\hookrightarrow\LL$ preserving the orders of the branchings and the cyclic orders of the cyclic branchings.

\end{definition}

\begin{theorem}\thlabel{th-univ-sing-leaf}
	Any pre-universal singular planar structure is universal.
\end{theorem}
As for universal planar structures, not all universal singular planar structures are pre-universal.

Next Lemma is an easy exercise of topology of $\RR$ and left to the reader:

\begin{lemma}\label{l.denombrable-dense}
	Let $I\subset\NN$ and $(X_i)_{i\in I}$, $(X_i)_{i\in I}$ be two families of dense countable subsets of $\RR$, so that the $X_i$ are pairwise disjoint, and the $Y_i$ too. Then there is an increasing homeomorphisms
	$h\colon\RR\to\RR$ so that $h(X_i)=Y_i$ holds for all $i\in I$.
\end{lemma}

\begin{proof}[proof of Theorem~\ref{th-univ-sing-leaf}]
	Let $\LL,\LL'$ be two singular planar structures, so that $\LL'$ is pre-universal. We build an embedding $\LL\to\LL'$ inductively.

	Denote by $A\subset\LL$, $A'\subset\LL'$, the subsets union of all of branchings, and by $B_n\subset\LL$, $B'_n\subset\LL'$, $n\in \NN\setminus\{1,2\}$ the union of all cyclic branchings of order $n$ for $n\geq 3$ and not in a cyclic branching for $n=0$.

	Let $(U_k)_{k\geq 0}$ be a countable covering of $\LL$ by lines (maximal open intervals), so that for all $i\geq 0$, the set $V_i=\cup_{j\leq i}U_i$ for $i\geq 0$ is connected.

	We will build the restriction $h_i$ of $h$ on $V_i$ inductively on $i$. More precisely:
	\begin{lemma}\label{l.univ-sing-leaf}
		For any $i\in \NN$ there an embedding $h_i\colon V_i\to\LL'$ with the following properties:
		\begin{enumerate}
			\item $h_i$ coincides with $h_{i-1}$ on $V_{i-1}$ (when $i\geq 1$),
			\item $h_i(A\cap V_i)\subset A'$ and $h_i(B_n\cap V_i)\subset B_n$ hold for every $n$,
			\item for any branching $\mu$ of $\LL$, $h_i(\mu\cap V_i)$ is contained in a branching $\mu'$ of $\LL'$ and the restriction of $h_i$ to $(\mu\cap V_i, \prec_\mu)\to (\mu', \prec_{\mu'})$ is increasing,
			\item if $\nu$ is a cyclic branching of $\LL$ that intersects $V_i$, then $h_i(\nu\cap V_i)$ is contained in one cyclic branching $\nu'$ of $\LL'$ which satisfies $h_i^{-1}(\nu')\subset\nu$, and additionally $h_i$ preserve the cyclic orders on $\nu$ and $\nu'$,
			\item given $j\leq i$, an orientation on $U_j$, the induced orientation on $h_i(U_j)$, a cyclic branching $\nu$ that intersects $U_j$ and $\nu'$ as above, $\nu$ is on the right of $U_j$ if and only if $\nu'$ is on the right of $h_i(U_j)$.
		\end{enumerate}
	\end{lemma}

	Assume Lemma~\ref{l.univ-sing-leaf} temporally.

	Defining $h\colon\LL\to\LL'$ as the common values of the $h_i$ where defined yields an embedding $\LL\to\LL'$ that is additionally a morphism of singular planar structures.
\end{proof}

\begin{proof}[Proof of Lemma~\ref{l.univ-sing-leaf}]
	We now prove the claim inductively on $n$, starting with $n=0$.
	To build $h_0$, we arbitrarily choose an open interval $U'_0$ in $\LL'$. We orient $U_0$ and $U'_0$. Then we apply Lemma~\ref{l.denombrable-dense} to $V_0=U_0\simeq\RR$ and $U'_0\simeq\RR$ where the countable dense subsets in the lemma are the sets of points in cyclic branchings of a given order $n$ at the right/left of $U_0$ and $U'_0$, and the sets $B_0\cap U_0$ and $B'_0\cap U'_0$. Lemma~\ref{l.denombrable-dense} provides $h_0$ satisfying items 2,3. The item 1 is empty for $i=0$ and items 4,5 are satisfied since a cyclic branching intersects an interval in at most one point.

	Let us now assume that $h_i$ has been built at rank $i$. We will extend $h_i$ on $U_{i+1}$, that is we define it on $U_{i+1}\setminus V_i$. According to Lemma~\ref{l.presque-connexe} $U_{i+1}\cap V_i$ is of the form $O\setminus C$, where $O$ is an open interval, and $C$ is a finite set of points. Additionally, for any $z\in C$, $z$ belongs to a cyclic branching $\nu$ that satisfies $\nu\setminus\{z\}\subset V_i$.
	Thus, $U_{i+1}\setminus V_i$ consists in one or two semi-open intervals and finitely many points on cyclic branchings.
	Extending the map on open intervals is the role of Lemma~\ref{l.denombrable-dense}. So the difficulty is essentially concentrated on the endpoints of $O$ in $U_{i+1}$ and the points in $C$. We thus proceed in two steps.

	\paragraph{Extension of $h_i$ on $C$.} Fix a points $z\in C$, $\nu$ the cyclic branching of $\LL$ that contains $z$, and $k$ the order of $\nu$. By induction hypothesis, the image $h_i(\nu\setminus\{z\})$ is contained in a cyclic branching $\nu'$ of $\LL'$, of the same order $k$. It follows from item 4 of the induction that there is a unique point $z'$ in $\nu'$ which is not inside $h_i(V_i)$. We set $h_{i+1}(z)=z'$, and extend $h_{i+1}$ on the all set $C$.

	We prove that the extension is continuous on $C$. Fix a point $z\in C$ and $\nu,\nu',z'$ as above.
	Denote by $\nu=(z_k)_{k\in\ZZ/k_t\ZZ}$ and $\nu'=(z'_k)_{k\in\ZZ/k_t\ZZ}$ the points in the cyclic branchings, in cyclical order, so that $z_1=z$ and $z'_1=z'$. We orient $O$ so that $\nu$ is on the right of $O$. Denote by $I^-\subset O\setminus C$ a half-open interval that converges toward $z$ from below, for the orientation on $O$. Then it also converges toward $z_2$. Denote by $U_j$, $j\leq i$, one of the open interval that contains $z_2$. Up to making $I^-$ smaller, we may assume that $U_j$ contains $I^-$. Then $h_i(I^-)$ converges toward two adjacent points $z'_k$ and $z'_s$ in $\nu'$, so that $z'_k$ belongs to $h_i(U_j)$. If $z'_s\neq z'_1$ holds, then $I^-$ converges toward the three points $h_i^{-1}(z'_s)$, $z_1$ and $z_2$, which contradicts the definition of cyclic branchings. Thus $z'_s$ is equal to $z'_1$. It follows that $h_{i+1}(z)=z'_s$, and that $h_{i+1}$ is continuous from below on $z$ (from inside $I^-$). The continuity from above is similar.

	We prove that this extension preserves the left/right position of cyclic branching. The orientation on $O$ induces an orientation on $I^-$ and then on~$U_j$. By choice of orientation on $O$, $\nu$ is on the left of $U_j$. So by induction hypothesis (item 5), $\nu'$ is on the left of $h_{i+1}(U_j)$ and so $h_{i+1}(O)$ is on the right of $\nu'$. It follows that the left/right position of cyclic branching is preserved by $h_{i+1}$.

	We prove that the extension preserves the injectivity. First see that $h_{i+1}$ is injective on $C$ since for any two distinct points $z_1,z_2\in C$, they are separated by any non-separatrix point $x\in O$ that lies in between them. Thus $h_i(x)$ also separates the cyclic branching that contains $h_{i+1}(z_1)$ and $h_{i+1}(z_2)$. So these two image points are different. The injectivity then follows from the fact that $h_{i+1}(C)$ is disjoint from $h_i(V_i)$. The injectivity of $h_{i+1}$ on the all set $U_{i+1}\cup V_i$ is addressed in the next paragraph.

	\paragraph{Extension of $h_i$ on a connected component of $U_{i+1}\setminus O$.}
	Let $W$ be a complementary component of $O$ in $U_{i+1}$. It contains one of the endpoints $w$ of $O$.
	We orient $O$ so that, that is the limit set of $\wt O(t)$ for $t\to+\infty$ given any orientation preserving parametrization $\wt O\colon\RR\to O$. Note that $w$ belongs to $\mu$.

	Since $U_i$ are lines, any sequence in $U_i$ that converges inside $\LL$ admits a limit inside $U_i$. And $V_i$ satisfies the same property. Let $x_n\in O\subset V_i$ be a sequence of points that converges toward $w\in\LL$. It follows from above that $x_n$ converges toward a point in $V_i$, necessarily difference from $w$ as $w$ is not in $V_i$. It follows that $\mu$ contains at least two elements, and thus is a branching of~$\LL$.

	Let us consider a first case where $w$ is not a separatrix. Then by induction hypothesis, $V_i\cap\mu$ lies in a branching $\mu'$ of $\LL'$ and $V_i\cap\mu\xrightarrow{h_i}\mu'$ is increasing. Let $w'$ be a non-separatrix point inside $\mu'\setminus h_i(\mu)$ (note that $|\mu|\leq i$ holds and $\mu'$ is infinite). We extend $h_i$ by setting $h_{i+1}(w)=w'$. Since $\mu'$ is maximally ordered, the points $w'$ can be chosen so that the map $(V_i\cap\mu)\cup\{w\}\xrightarrow{h_{i+1}}\mu'$ is increasing.

	Note that $w'$ disconnect $\LL'\setminus\{w'\}$, and one connected component of $\LL'\setminus\{w'\}$ contains $h_i(V_i)$.
	Take a semi-open interval $W'\subset\LL'$ that starts at $w'$ is disjoint from $h_i(V_i)$. Then we extend $h_i$ to $W$ by using Lemma~\ref{l.denombrable-dense} from $W\setminus\{w\}$ to $W'\setminus\{w'\}$. The extension is continuous on $W$. Note that it preserves the injectivity, since $h_{i+1}(W)$ is made disjoint from $h_i(V_i)$. We should also point out that $h_{i+1}(W)$ is disjoint from the set $h_{i+1}(C)$ defined in the previous step. Additionally, the extension of $h_i$ on the two connected components of $U_{i+1}\setminus O$ are disjoint, since for any separable points $x$ in $h_{i+1}(O)$, $x$ disconnects the two sets $h_{i+1}(W)$ defined above.

	Let us explain how to adapt the previous step when $w$ is a separatrix, and belongs to a cyclic branching $\nu$. We assume that $\nu$ lies on the right of $U_{i+1}$, the other case is similar. Denote by $w_*$ the separatrix in $\nu$ that comes just before $w$, so that $w_*$ also belongs to $\mu$ and $w_*<w$ holds on $\mu$. Since $\LL'$ is universal, there exist two points $w'$ and $w_*'$ in $\mu'$ that are adjacent for the order on $\mu'$, so that $w_*'<w'$ holds on $\mu'$.
	When $w_*$ is outside $V_i$, then we chose $w'$ and $w_*'$ outside $h_i(V_i)\cap\mu$ (which is finite). When $w_*$ belongs to $V_i$, we claim that we may chose $w_*'=h_i(V_i)$ and $w'$ outside $h_i(V_i)$.

	To prove the last claim, denote by $\nu'$ the cyclic branching that contains $h_i(w_*)$, and by $w'$ the element in $\nu'$ that comes immediately after $h_i(w_*)$. Reason by contradiction and assume that $w'$ belongs to $h_i(V_i)$, and denote by $x=h_i^{-1}(w')$. Then by assumption, $x$ belongs to the cyclic branching $\nu$. The points $h_i(w_*)$ and $w'$ are not separated from their side that contains $h_i(O\setminus C)$, so $w_*$ and $x$ are not separated from their side that contains $O$. It follows that $w,w_*,x$ belong to $\nu$, and are all non separated from each other on the same side, which contradicts the definition of cyclic branchings. It follows that $w'$ does not belong to $h_i(V_i)$, and so we can take $w_*'=h_i(V_i)$ as above.

	We set $h_{i+1}(w)=w'$. Note that when $w_*$ belongs to $h_i(V_i)$, then $h_{i+1}$ preserves the cyclic orders on $\nu$ and $\nu'$. Then proceed as in the previous case to extend $h_{i+1}$ on $W\setminus\{w\}$. Note that we need to adapt an argument: $w'$ does not separate $\LL'$ anymore. When $w_*$ does not belong to $h_i(V_i)$, then $\{w',w_*'\}$ separates $\LL'$ and we may use in the extension argument. When $w_*$ does not belong to $h_i(V_i)$, we claim that the separatrix $w''$ in $\nu'$ that comes just after $w'$ does not belongs to $h_i(V_i)$. Indeed if it was the case, since $h_i(V_i)$ is connected and disjoint from $w'$, $h_i(V_i)$ contains all $\nu'\setminus\{w'\}$. But then $V_i$ contains all $\nu\setminus\{w\}$ (using item 2 and 4 in the induction hypothesis). It would follow that $V_i$ contains a neighborhood of $w$ minus $w$ itself. But then $w$ would not be an endpoints of $O$ be would be inside $C$. Since it is not the case, $w''$ does not belongs to $h_i(V_i)$. Thus $\{w',w''\}$ separates $\LL'$, and we can use it in the extension argument.

	The continuity and the injectivity are proven similarly to the first case, with minor adjustments. The item 1,2,3 and 5 follow from the construction. Let us argue about the $4^th$ item. Given a cyclic branching $\nu$ of $\LL$ that intersects $U_{i+1}\setminus V_i$ and a (unique) point $w$. If $w$ belongs to $C$, then by construction, $\nu$ is included in $U_{i+1}\cup V_i$ and $U_{i+1}\cup V_i\xrightarrow{h_{i+1}}\LL'$ preserves the cyclic order on $\nu$. Assume that $w$ is one end point of the interval $O$. Either $w_*$ (as defined above) belongs to $V_i$, then $\nu$ contains $w_*$ and $h_{i+1}(w)$ is defined to belong to the same cyclic branching as $h_i(w_*)$, respecting the cyclic branching. Or $w_*$ does not belong to $V_i$ and $w'$ belong to a cyclic branching that intersects $h_{i+1}(U_{i+1}\cup V_i)$ only on $h_{i+1}(w)$. In both case, item 4 is satisfied. Assume now that $w$ is not in $C$ or a end point of $O$, then item 4 follows from Lemma~\ref{l.denombrable-dense}. Thus item 4 is satisfies, ending the induction and the proof.

\end{proof}

%% file: Preliminar.tex
\section{Pre-laminations of the circle}  

We introduce the main properties of pre-lamination that we will use later.

\subsection{Definitions}\label{s.pre-lam-def}
Denote by $\Diag$ the diagonal in $S^1\times S^1$ where $S^1$ denotes the unit circle boundary of the unit disc $D^2$. Given two pairs $(x,y)$ and $(z,w)$ in $S^1\times S^1\setminus\Diag$, we say that they \emph{cross each other} if the four elements in $S^1$ are distinct and if one of the two tuples $(x,z,y,w)$, $(x,w,y,z)$ is cyclically ordered. Or equivalently if the geodesic $]x,y[$ and $]z,w[$ in $\Int(D^2)$ cross each other.

\begin{definition}
	A \emph{pre-lamination} $L$ of the circle $S^1$ is a subset of $S^1\times S^1\setminus\Diag$ invariant by symmetry $(x,y)\mapsto(y,x)$ and so that no two elements in $L$ cross each other.
\end{definition}

A pre-lamination $L$ is called \emph{dense} if the set $\{x\in S^1|\exists y\in S^1,(x,y)\in L\}$ is dense in $S^1$. In this article, only dense  pre-laminations are of interest. So let us fix a dense pre-lamination $L$ here.

The \emph{geodesic realization $G(L)$} of $L$ is the pre-lamination of the interior of the unit disc $\Int(D^2)$ formed by the Euclidean geodesics $]x,y[$ for all pairs of points $(x,y)\in L$. 
Any geodesic $]x,y[$ for $(x,y)\in L$ is called a \emph{leaf} of $G(L)$. Any geodesic in $\Int(D^2)$ which is accumulated by leaves of $G(L)$, and that is not a leaf of $G(L)$ on its own, is called an \emph{accumulated geodesic}. In the following, we identify $L$ and its realization $G(L)$, and pair $(x,y)\in L$ with the corresponding (oriented) leaf. We often express our arguments on the realization $G(L)$ instead of  in terms of pairs of points of $S^1$. It makes the discussion simpler to picture and to argue, but every argument could be translated purely on the circle.

For the geodesic realization, we consider geodesic for the Euclidean metric. We make that choice for simplicity, but any non-positively curved metric would suffice. In particular, picture will be illustrated using the hyperbolic metric in the Poincaré model, for clarity reasons.

A \emph{complementary region} $\Delta$ of $L$ is a connected component of $\Int(D^2)\setminus G(L)$.
A complementary region $\Delta$ is either an accumulated geodesic in $\Int(D^2)$ (accumulated on both sides), or it is bounded in $\Int(D^2)$ by at least three and at most countably many geodesics (leaves or accumulated geodesics). Note that it is has empty-interior only in the first case.
When $L$ is dense, the intersection of the closure of $\Delta$ with $S^1=\partial D^2$ is totally disconnected.

\begin{definition}
	Let $L$ be a dense pre-lamination on the circle.
	A \emph{shell} is a complementary region $\Delta$ with non-empty interior and so that exactly one boundary component $\gamma$ is not a leaf of $G(L)$. The geodesic $\gamma$ is called the \emph{root} of $\Delta$.

	A \emph{star} is a complementary region $\Delta$ with non-empty interior bounded by an \emph{ideal polygon}, that is by finitely many leaves of $G(L)$. Note that each boundary leaves share its endpoints with two other boundary leaves.
\end{definition}

We will see that the shells play an important role, as they correspond to branching for foliations, and the stars corresponds to the prong singular points.
We illustrates this phenomenon in Figure \ref{fig-branching}.
On the left, a family of accumulated leaves of a foliation induces a shell, in light blue, whose root is in dashed line. On the right, a $3$-prong singularity of a singular foliation induces a star with three boundary components.

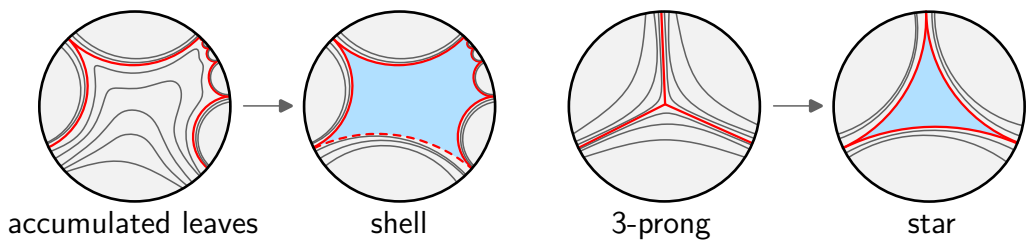
\begin{figure}
	\begin{center}
		\begin{picture}(130,26)(0,0)
			\put(0,0){\includegraphics[width=130mm]{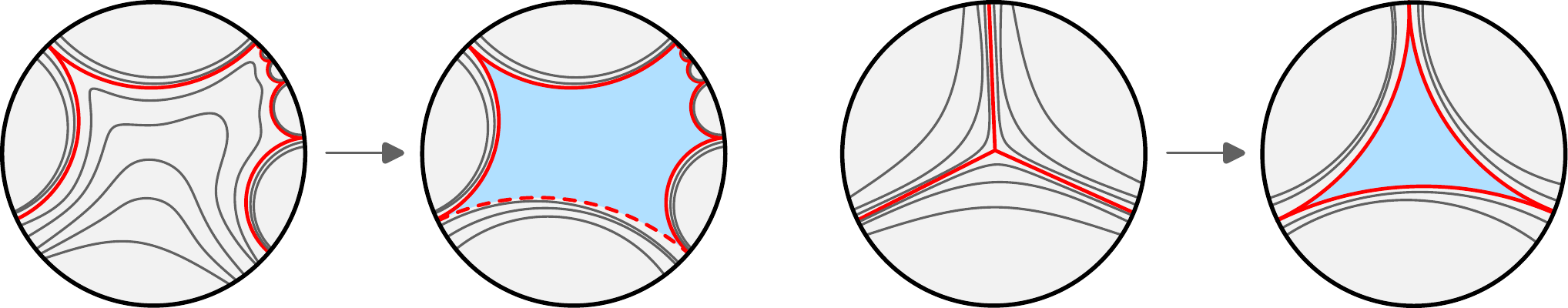}}
			\put(-4,-4){accumulated leaves}
			\put(43.8,-4){shell}
			\put(75.5,-4){3-prong}
			\put(114.5,-4){star}
		\end{picture}
	\end{center}
	\caption{Branchings from foliations to pre-laminations.}
	\label{fig-branching}
\end{figure}


We are interested by a particular pre-lamination, the one induced by foliations.
Let~$\FF$ be a foliation on $\Int(D^2)$ so that, using Theorem~\ref{th-Bonatti}, it boundary at infinity coincide with $S^1$.
Given $f$ an (oriented) non-singular leaf or a face of a singularity, denote by $\pi_\FF(f)=(\xi_{f^+},\xi_{f^-})$ in $S^1\times S^1$ the pair of ends of the leaf (in any order). We define $L_\infty(\FF)$ to be the set of pairs as above. According to Remark~\ref{r.same-end} on has $\xi_{f^+}\neq \xi_{f^-}$. It is also clear that two pairs $(\xi_{f_1^+},\xi_{f_1^-}),(\xi_{f_1^+},\xi_{f_2^-})$ do not cross, as leaves or faces of leaves do not cross. Therefore, $L_\infty(\FF)$ is a dense pre-lamination of $S^1$. We call $L_\infty(\FF)$ the \emph{end pre-lamination of~$\FF$}.

\subsection{Monotonous sequences of leaves}

A sequence $(l_k)_{k\in\NN}$ of leaves of $G(L)$ is called \emph{monotonous} if for any $k$, $l_k$ separates any pair of leaves $(l_i,l_j)$ with $i<k$ and $j>k$. A monotonous sequence $(l_k)_k$ is called \emph{bounded} if there exist a leaf $l$ of $L$ so that all $l_k$ with $k\geq 1$ separates $l_0$ and $l$. It is called \emph{diverging} otherwise. Next lemma is easy and its proof is left to the reader.

\begin{lemma}\thlabel{lem-monotone-leaf-lam}
	Any  bounded monotonous sequence $(l_k)_{k\in\NN}$ of leaves of $G(L)$ converges to either a leaf of $G(L)$ or an accumulated geodesic.

	Any diverging monotonous sequence $(l_k)_{k\in\NN}$ of leaves of $G(L)$ converges to a point of the circle $S^1$.
\end{lemma}

\begin{remark}
	Consider a sequence $(l_k)_{k\in\NN}$ of leaves of $G(L)$ crossing a same geodesic $\gamma$. Then the sequence $l_k$ is monotonous if and only if the sequence of intersection points $x_k=l_k\cap \gamma$ is monotonous inside $\gamma$.
\end{remark}

Let~$\FF$ be a pA-foliation of the plane. A sequence $(f_k)_{k\in\NN}$ of non-singular leaves or faces of singular leaves of~$\FF$  is called \emph{monotonous} if $f_k$ separates the leaves $(l_i,l_j)$ whenever $i<k<j$ holds. This sequence is \emph{bounded } if  there exists a leaf $f$ of~$\FF$ so that all leaf $f_k$, $k\geq 1$ separates $f_0$ and $f$.


\begin{remark}
	Denote by $\pi=\pi_\FF$ the map which associate to a non-singular leaf or a face of~$\FF$ the corresponding leaf of $L_\infty(F)$. Then a sequence $(f_k)_{k\in\NN}$ of leaves or faces is monotonous  if and only if the sequence $l_k=\pi_\FF(f_k)$ is monotonous.
\end{remark}

We also define the opposite notion: a sequence $(f_k)_{k\in I\subset \NN}$ of non-singular leaves or faces of singular leaves of~$\FF$ is said to be \emph{stationary} if they are pairwise distinct and if no $f_k$ separates two other $f_i,f_j$.  One defines in the same way \emph{stationary sequences of leaves of $G(L)$}, where $L$ is a pre-lamination.  A sequence $(f_k)_{k\in I\subset \NN}$ of non-singular leaves or faces of singular leaves
is stationary if and only if the sequence $l_k=\pi_\FF(f_k)$ of leaves of $G(L_\infty(\FF)$ is stationary.

Next remark shows that this notion is natural.

\begin{remark}
	Given a complementary region $\Delta$ of a pre-lamination $L$, the family of leaves of $L$ in the boundary of $\Delta$ is stationary.
\end{remark}

\begin{lemma}\thlabel{lem-leaf-comp}
	Let  $(f_k)_{k\in\NN}$ be a monotonous sequence of non-singular leaves or faces of~$\FF$, and $l_k=\pi_\FF(f_k)$ be the corresponding leaves of $L_\infty(\FF)$. Then $f_k$ is bounded if and only if $l_k$ is bounded. Additionally, the two following holds:
	\begin{itemize}
		\item If $f_k$ and $l_k$ are bounded, let $a,b\in S^1$, $a\neq b$ be the end points of the limit geodesic of the sequence $l_k$. Then the closures $\bar f_k$ of $f_k$ in $D^2$ converges toward a simple curve joining $a$ to $b$, whose intersection with $S^1$ is totally disconnected and whose intersection with $\Int(D^2)$ is the union of at most countably many non-singular leaves or faces of~$\FF$, pairwise non-separated.
		\item If $f_k$ is not bounded, it converges toward a unique point in $D^2$, which lies inside $S^1$ and which is also the limit of the $l_k$.
	\end{itemize}
\end{lemma}

\begin{proof}
	The equivalence is clear.
	Let us prove the first item.
	Let $f$ be the leaf in the definition of bounded monotone sequence: $f_n$ separates $f_0$ from a given leaf $f$.
	Consider the component $U_n$ of $\Int(D^2)\setminus f_n$ containing $f$, and let $D_n$ be its closure inside $D^2$. The sequence $D_n$ is a deceasing sequence (for the inclusion) of compact discs so the set $D_\infty=\bigcap_n D_n$ is compact and not empty. Let $U$ be the connected component of $\Int(D^2)\setminus f$ contained in $U_n$ and $D\subset D_n$ its closure. Since. Then $D$ lies inside all $D_n$ and so inside $D_\infty$ too. It follows that $D_\infty$ has non empty interior and that it intersects $\Int(D^2)$. The open discs $U_n$ are saturated for~$\FF$ and so is $D_n\cap\Int(D^2)$. Thus, $D_\infty\cap \Int(D^2)$ is a closed subset of $\Int(D^2)$ saturated for~$\FF$. Hence, its boundary in $\Int(D^2)$ consists of leaves of~$\FF$, all accumulated by the boundary component $f_n$ of $U_n$. In particular, either it consists of a unique leaf or face, or all these leaves and faces are pairwise non-separated.

	Observe that $D_\infty$ is connected and simply connected. Denote by $\sigma$ the closure of $\Int(D^2)\cap\partial D_\infty$ inside $D^2$. Observe that $\sigma\cap\Int(D^2)$ is the accumulation locus of the leaves $f_n$, and so is the union of at most countably many leaves that are pair-wise non-separated. Let us orient the leaves $f_n$ so that each $f_n$ is on the left if $f_i$ when $i<n$, and on the right otherwise. For any leave $f$ inside $\sigma\cap\Int(D^2)$, we orient $f$ so that al $f_n$ are on the left of $f$. Then all points on the right of $f$ are inside $D_\infty$. So the open interval in $S^1$ bounded by the ends of $f$, and on the right of $f$, is disjoint from $\sigma$. From the construction of $\sigma$, $\sigma\cap S^1$ is the closure of the set of ends of all the leaves that belong to $\sigma\cap\Int(D^2)$. So $\sigma\cap S^1$ is the closure of the set of end points of at most countably many intervals, so it is totally disconnected.

	Denote by $l$ the leaf of $L_\infty(\FF)$ that is limit of $l_n$, and $a,b$ its endpoints. Choose them so that when oriented from $a$ to $b$, $l_n$ is on the left of $l$. The set $\sigma$ is homeomorphic to the interval from $a$ to $b$ in $S^1$, where we identify any leaf $f$ in $\sigma\cap S^1$ to the interval in $S^1$ between the two ends of $f$. Thus $\sigma$ is homeomorphic to a segment.

	We now prove the second item. The sequence $l_n$ converges toward a point $x\in S^1$. Define $D_n$ and $D_\infty$ as above. Then $D_\infty\cap S^1$ is equal to $x$. If there exists a point $y\in D_\infty\cap\Int(D^2)$, the leaf $f$ that contains $y$ lies inside $D_n$ for all $n$, and so $f$ bounds the sequence $f_n$. It contradicts the assumption, so $D_\infty$ is equal to $\{x\}$.
\end{proof}

\subsection{Limit of complementary components}
In this section, length, diameter, and area are considered on $D^2$ endowed with the Euclidian metric.

Let $L$ be a pre-lamination of $S^1$, identified with its geometric realization $G(L)$ on $\Int(D^2)$. Let $\Delta_n\subset\Int(D^2)$ be a sequence of pairwise distinct complementary components of $L$. Let $\bar\Delta_n$ be the closures of $\Delta_n$ in $D^2$.

As the set of compact subset of a compact  metric space set (endowed with the Hausdorff topology) is a compact metric space,
up to taking a subsequence we may assume that $\bar\Delta_n$ converges to a compact subset $\Delta_\infty$ of $D^2$.

\begin{lemma}\label{l.limite-regions}
	Let $\Delta_n$ be a sequence of complementary regions of a pre-lamination $L$, and assume that $\bar \Delta_n$ converges to a compact subset $K\subset D^2$. Then the following hold true:
	\begin{itemize}
		\item if the diameter $\diam(\Delta_n)$ tends to $0$, then $K$ is a point on $S^1$.
		\item if the diameter $\diam(\Delta_n)$ has a strictly positive limit $\delta>0$, then $K$ is a geodesic of length $\delta$. Moreover, for $n$ large enough, there are two distinct boundary components $\gamma_{n,1},\gamma_{n,2}$ of $\Delta_n$ which satisfy $\lim \gamma_{i,n}=K$ and that the supremum of the lengths of the other boundary components of $\Delta_n$ converges toward zero.
	\end{itemize}
\end{lemma}

Let us give another version of Lemma~\ref{l.limite-regions} (both lemmas will be proved together).

\begin{lemma}\label{l.calibrage}
	Let $L$ be a pre-lamination. For any $\delta>0$ and any $\varepsilon\in ]0,\tfrac{\delta}{2}[$, every complementary region $\Delta$ with diameter $\diam(\Delta)>\delta$, up to possibly finitely many of them, has the following properties:
	\begin{itemize}
		\item $\Delta$ has exactly two boundary components, which are geodesics with length larger than $\diam(\Delta)-\varepsilon$
		\item the other boundary components have all length smaller that $\varepsilon$.
	\end{itemize}
\end{lemma}

\begin{proof}
	Let $(\Delta_n)_{n\in\NN}$ be a sequence of distinct complementary regions of $L$.
	The $\Delta_n$ are convex with extremal points on $S^1$, thus their limit is a convex compact subset of $D^2$ with extremal points on $S^1$.  Furthermore, as the $\Delta_n$ have pairwise disjoint interiors,  the limit set $K$ has empty interior. Thus it is a point on $S^1$ or a geodesic segment with endpoints on $S^1$.  Thus if $\diam(\Delta_n)$ tends to $0$ the limit set $K$ is a points on $S^1$ otherwise it is a geodesic segment.

	We now assume that $K$ is a geodesic segment of length $\delta$. Take some $\epsilon$ in~$]0,\tfrac{\delta}{2}[$.
	The interiors of the~$\Delta_n$ are pairwise disjoint, so the sum of their area is inferior to the area of the disc~$D^2$. Thus, the area of $\Delta_n$ tends to $0$.
	Consider a geodesic segment $\sigma$ crossing transversely $K$.  For $n$ large, $\sigma$ crosses the boundary of $\Delta_n$ at two points. These two points are based on distinct boundary components~$\gamma_{1,n}$ and~$\gamma_{2,n}$. As the $\gamma_{i,n}$ get arbitrarily close to $\sigma\cap K$ and do not cross $K$ they make an arbitrarily small angle with the direction of $K$. Hence, their length converges toward the length of $K$, and so it is larger than $\delta-\varepsilon$ for large enough $n$

	No other boundary component of $\Delta_n$ crosses $\sigma$, otherwise one of them should separate the two other, contradicting the connectedness of $\Delta_n$. Thus, the other boundary components are contained in the thin strip bounded by the $\gamma_{i,n}$ and not crossing $\sigma$ thus their length converge uniformly to $0$ with $n$.
\end{proof}

Let $L$ be a pre-lamination. For any $\delta>0$, up to finitely many complementary regions  of diameter larger than $\delta$,  a complementary region $\Delta$ with diameter larger that $\delta$ has exactly $2$ longest boundary components.

Thus the notion of \emph{the two largest boundary components of a complementary region $\Delta$} is well-defined, up to finitely many complementary regions of diameter larger than $\delta$.

\subsection{Properties of the induced pre-laminations:  the conditions in Theorems~\ref{main-A} and  \ref{main-A-sing} are necessary.} \label{sec-end-pre-lam}

We explore the elementary properties satisfied by the end pre-lamination. More precisely, in this section we will see that the hypotheses of Theorems~\ref{main-A} and \ref{main-A-sing} are satisfied for every pre-lamination induced by a non-singular foliation or by a pA-foliation.

\subsubsection{Shell/star pre-lamination}

Recall that  $\pi=\pi_\FF$ denotes the map which associates to a non-singular leaf or a face of~$\FF$ the corresponding leaf of $L_\infty(\FF)$. 
Let $L$ be a dense pre-lamination.
Recall that a complementary region of $L$ is a shell if it has a non-empty interior and if it has exactly one boundary component (its root) that is an accumulated geodesic (i.e. is not a leaf). 
A pre-lamination $L$ is \emph{a shell pre-lamination} when all its complementary regions are shells.

As a corollary of Lemma~\ref{lem-leaf-comp} one gets:

\begin{corollary}\label{c.limites}
	Let $f_n$ be a bounded monotonous sequence of non-singular leaves or faces of singular leaves, and $l_n=\pi_\FF(f_n)$ the corresponding leaves of $L_\infty(\FF)$.
	Let $l_\infty$ be the limit of the $l_n$.  Then
	\begin{itemize}
		\item if $l_\infty$ is a leaf of $L_\infty(\FF)$ then $f_n$ tends to the leaf (or face) $f_\infty$ with $\pi_\FF(f_\infty)=l_\infty$.
		\item if $l_\infty$ is not a leaf of $L_\infty(\FF)$ then let
		      $(g_i), i\in I\subset \NN$ be the family of leaves (or faces) given by Lemma~\ref{lem-leaf-comp} on which the sequence $f_n$ accumulates, and consider $\gamma_i=\pi_\FF(g_i)$. Then $l_\infty$ and the $\gamma_i$ are the boundary components of a complementary region of $L_\infty(\FF)$ which is a shell with root $l_\infty$.
	\end{itemize}
\end{corollary}

\begin{proof}
	First assume that $l_\infty$ is not a leaf of $L_\infty(\FF)$.  Using the notation of the proof of Lemma~\ref{lem-leaf-comp}, $a,b$ are the endpoint of $l_\infty$, $\sigma$ is the path joining $a$ to $b$, limit of the $f_n$. The intersection of $\sigma$ with the interior of $D^2$ is the union of the leaves $g_i$. We write $\gamma_i=\pi_\FF(g_i)$ and $\gamma=(\sigma\cap S^1)\cup\bigcup_i\gamma_i$. Then $\gamma$ is a path in $D^2$ joining $a$ to $b$ and intersecting $\bar l_\infty$ exactly on $\{a,b\}$. Thus $l^\infty\cup \gamma$ is a simple closed curve in $D^2$, bounding an open disc $\Delta$.

	We now check that $\Delta$ is a complementary region of $G(L_\infty(\FF))$, that is the interior of $\Delta$ is disjoint from all leaves of $G(L_\infty(\FF))$. If the interior of $\Delta$ was intersecting a leaf $l$, the $l$ would separate $l_\infty$ from at least one of the leaves $\gamma_i$. Thus the corresponding leaf or face $\pi_\FF^{-1}(l)$ of~$\FF$ would separate every $f_n$ from the leaf $g_i$, contradicting the fact that $f_n$ accumulates on $g_i$.

	Assume now that $l_\infty$ is a leaf of $L_\infty(\FF)$ and let $f_\infty$ be the corresponding leaf (or face) of~$\FF$. Then $f_\infty$ is contained in each of the discs $D_n$ bounded by $f_n$ and therefore is contained in $\bigcap_n D_n$. However, the boundary of $\bigcap_n D_n$ is the segment $\sigma$ joining $a$ to $b$, and $f_\infty$ is joining $a$ to $b$. One easily deduces that $\sigma$ is the closure of $f_\infty$. Thus, $f_\infty$ is the limit of the sequence $f_n$ concluding the proof.
\end{proof}

As a straightforward consequence  of Corollary~\ref{c.limites} one gets:

\begin{corollary}\label{c.accumulee}
	Any geodesic $l\notin L_\infty(\FF)$ accumulated on one side by leaves of $L_\infty(\FF)$ bounds on its other side a complementary region which is a shell.
\end{corollary}

\begin{proof}
	One chooses a monotonous sequences $l_n$ of leaves of $L_\infty(\FF)$ converging to $l$.  This sequence is monotonous and bounded and thus the corresponding sequences $f_n$ of leaves or faces of~$\FF$ is also monotonous and bounded. The second item of Corollary~\ref{c.limites} concludes.
\end{proof}

A pre-lamination $L$ is a \emph{a shell/star pre-lamination} if all its complementary regions are shells or stars.

As a consequence of Lemma~\ref{l.continuous}, Lemma~\ref{l.regular-separated} , Corollary~\ref{c.limites} and Corollary ~\ref{c.accumulee},  Lemma~\ref{l.leaf-boundary} describes all the possible behaviors of a leaf of $L_\infty(\FF)$ :

\begin{lemma}\label{l.leaf-boundary}
	Consider $l\in L_\infty(\FF)$, and $f=\pi_\FF^{-1}(l)$. Consider a side of $l$ and the corresponding side of $f$:
	\begin{enumerate}
		\item If $f$ is a non-singular leaf of~$\FF$ and is regular on that side, then $l$ does not bound a complementary region on that side.
		\item If $f$ is a non-singular leaf  but is not regular on that side, then $l$ bounds a shell complementary region on that side.
		\item If $f$ is a face of a singularity, that side is free and $f$ is regular on that side, then $l$ does not bound a complementary region on that side.
		\item If $f$ is a face of a singular leaf but that side is free and $f$ is not regular on that side, then $l$ bounds a shell complementary region.
		\item If $f$ is a face of a singular leaf and that side is the singular side, then $l$ bounds a star complementary region.
	\end{enumerate}
\end{lemma}

\begin{proof}
	If $f$ is regular on that side, according to Lemma~\ref{l.continuous} one gets that $l$ is accumulated on that side by leaves of $L_\infty(\FF)$, hence $l$ does not bound a complementary region on that side. This proves items 1 and 3.

	If $f$ is non-singular on that side but is not regular on that side, then $f$ is accumulated on that side by a sequence of leaves $f_k$ accumulating at least another leaf or face, according to Lemma~\ref{l.regular-separated}. Then $l$ bounds on that side a shell whose root is the limit of the $\pi_\FF(f_k)$. This proves the items 2 and 4.

	To prove item $5$, just notice that the vertices of the star are the endpoints of the separatrixes of the singular point.
\end{proof}

\begin{lemma}\thlabel{lem-end-lam-comp}
	The pre-lamination $L_\infty(\FF)$ of a non-singular foliation~$\FF$ is a shell pre-lamination.

	The pre-lamination $L_\infty(\FF)$ of a pA-foliation~$\FF$ is a shell/star pre-lamination. Furthermore, no leaf bound a star complementary region on each of its two sides.
\end{lemma}

\begin{proof}
	Let~$\FF$ be a pA-foliation and $\Delta$ be a complementary region of $L_\infty(\FF)$.
	If $\Delta$ has an empty interior, it is a single geodesic that is not a leaf of $L_\infty(\FF)$ but is accumulated on its two sides by two sequences $(l^{+/-}_n)_n$ of leaves of $L_\infty(\FF)$. The preimage leaves (or faces of leaves)  $\pi_\FF^{-1}(l^{+/-}_n)$ of~$\FF$ converge toward a unique leaf $f$ of~$\FF$ which satisfies $\pi_\FF(f)=P$, contradicting the fact that $\Delta$ is not a leaf of $L_\infty(\FF)$.

	So $\Delta$ has non-empty interior. If it has a boundary component that is not a leaf (but accumulated by leaves of $L_\infty(\FF)$, then Corollary~\ref{c.accumulee} implies that it is a shell, and we are done.

	Assume now that every boundary component of $\Delta$ is a leaf of $L_\infty(\FF)$, and let $l$ be one of this leaf, and $f=\pi_\FF^{-1}(l)$.  Then Lemma~\ref{l.leaf-boundary} implies that $f$ is a face of a singular leaf and is singular on the side of $\Delta$, so $\Delta$ is a star.

	Finally, if a leaf $l$ of $L_\infty(\FF)$ bounds a star region on one side, then $\pi_\FF^{-1}(\FF)$ is free on the otherside, so $l$ does not bound a star region on its other side.
\end{proof}

\subsubsection{Bad accumulations}

Let $L$ be a shell/star pre-lamination of $S^1$. 
Recall that, among complementary region with one boundary of length at least $\epsilon>0$, all but finitely many have well-defined two long edges (see Lemma~\ref{l.limite-regions}).

A \emph{bad accumulation of shell regions} of $L$ is a sequence $(\Delta_n)_{n\in\NN}$ of complementary shells that converges  toward a geodesic $l$ of $\Int(D^2)$, and so that the two long edges of $\Delta_n$ are leaves of $L$. In other words, according to Lemma~\ref{l.limite-regions}, the lengths of their roots tend to $0$.

A \emph{bad accumulations of star regions} of $L$ is a sequence $(\Delta_n)_{n\in\NN}$ of complementary stars that converges toward a geodesic $l$ of $\Int(D^2)$, and so that the two long edges of $\Delta_n$ are not adjacent.

The lack of bad accumulations has a simple characterization:

\begin{lemma}
	Let $L$ be a shell/star pre-lamination. Then
	\begin{enumerate}
		\item $L$ has no bad accumulations of shells if and only if for all $\epsilon>0$, for all but maybe finitely many of the complementary shell $\Delta$ that have an edge of length at least $\epsilon$, the root of $\Delta$ is one of its two long edges.
		\item $L$ has no bad accumulations of stars if and only if for all $\epsilon>0$, for all but maybe finitely many of the complementary star $\Delta$ that have an edge of length at least $\epsilon$, the two long edges of $\Delta$ are adjacent.
	\end{enumerate}
\end{lemma}


\begin{lemma}\thlabel{cor-no-acc}
	Let~$\FF$ be a pA-foliation.
	The pre-lamination $L_\infty(\FF)$ has no bad accumulation.
\end{lemma}

\begin{proof}
	We reason by contradiction.
	Assume that $\Delta_n$ is a sequence of complementary regions having a bad accumulation on a geodesic $\gamma$. For $n$ large the two long edges $\gamma_{1,n},\gamma_{2,n}$ of $\Delta_n$ are well-defined and are leaves of $L_\infty(\FF)$ converging to $\gamma$. Up to extracting a subsequence, one may assume that this accumulation is on one side of $\gamma$, that we call its positive side.
	Let $f_{i,n}=\pi_\FF^{-1}(\gamma_{i,n})$. Lemma~\ref{lem-leaf-comp} implies that, if $\gamma$ is a leaf, then the $f_{i,n}$ converge to the leaf $f=\pi_\FF^{-1}(\gamma)$. If $\gamma$ is not a leaf, then it bounds (on its negative side) a shell $\Delta$ and the $f_{i,n}$ accumulate on a leaf or face $f$ whose projection by $\pi_\FF$ is a boundary component of $\Delta$ which  is not the root. In both cases the sequence $f_{i,n}$ accumulate a non singular leaf or a face of singular leaf $f$. Fix an open transversal $\sigma$ to~$\FF$ that intersects $f$ (and disjoint from the singular set). As a consequence for $n$ large enough, $f_{1,n}$ and $f_{2,n}$ cross $\sigma$.

	If the $\Delta_n$ are shells, according to Lemma~\ref{l.separated}, that contradicts the fact that $f_{1,n}$ and $f_{2,n}$ are not separated. Thus $L_\infty(\FF)$ has no bad accumulations of shells.

	If the $\Delta_n$ are stars, the hypothesis of bad accumulations means that $f_{1,n}$ and $f_{2,n}$ are not successive faces of the singular leaf $f_n$ corresponding to $\Delta_n$. Thus $f_{1,n}$ and $f_{2,n}$ intersects only on the singular point. Since $\sigma$ intersects both faces $f_{i,n}$ outside the singular point of $f_n$, $\sigma\cap f_n$ contains at least $2$ points, contradicting Lemma~\ref{l.1point}. Thus $L_\infty(\FF)$ has no bad accumulations of stars.
\end{proof}

In summarizing, we proved:

\begin{proposition}\label{p.necessary}
	Let $L\subset (S^1\times S^1\setminus diag)/(x,y)\simeq (y,x)$ be a pre-lamination on the circle $S^1$. Next conditions are necessary for the existence of a pA-foliation~$\FF$ so that $L_\infty(\FF)=L$:
	\begin{itemize}
		\item (density condition) $L$ is dense
		\item (countability) For every $x\in S^1$ the set of $y\in S^1$ for which $(x,y)\in L$ is at most countable.
		\item $L$ is a shell/star pre-lamination
		\item a leaf may belong to at most one star.
		\item $L$ has no bad accumulation.
	\end{itemize}
	Furthermore the singular leaves of the foliation~$\FF$ (if~$\FF$ exists)  are in one-to-one correspondence with the stars: in particular if $L$ is a shell pre-lamination, the foliation~$\FF$ is non-singular.
\end{proposition}

To prove Theorems~\ref{main-A} and ~\ref{main-A-sing} it remains to prove that these conditions are sufficient for the existent of the foliation~$\FF$ and to check the uniqueness (up to conjugacy) of the foliation. This proof will be much easier in the non-singular case, thanks to Kaplan’s work characterizing (non-singular) foliations up to conjugacy, using their leaf-space. In the singular case, we will need to find an equivalent to Kaplan's work.  For this reason we first present the proof in the non-singular case.

To answer Question \ref{q.induced}, that is characterizing pre-lamination corresponding to a dense subset of leaves of a foliation, we will need an extra condition, substituting the countability condition, in order to get a condition which will be preserved by completion. That is the aim of section~\ref{s.few}.


\subsection{Few common ends}\label{s.few}

The end pre-lamination of a foliation satisfies a last property.
Let $L$ be a pre-lamination. Given $\theta\in S^1$, we denote by $E_L(\theta)$ the set of points $\theta'\in S^1\setminus\{\theta\}$ that satisfy $(\theta,\theta')\in L$. We say that $L$ has \emph{few common ends} if for all $\theta\in S^1$, the set $E_L(\theta)$ is either empty or ordered as a sub-interval of $\ZZ$ (that is finite, ordered as $\NN$, $-\NN$ or $\ZZ$). Here the set $E_L(\theta)$ is naturally ordered by its endpoint other than $\theta$ in $S^1\setminus\{\theta\}$.

\begin{lemma}\thlabel{lem-end-lam-few-end}
	Let~$\FF$ be a pA-foliation of $\Int(D^2)$. Then  $L_\infty(\FF)$ has few common ends.
\end{lemma}

Lemma~\ref{lem-end-lam-few-end} is stated and proved in \cite[Lemma 4.6]{BonattiCirc}] for non-singular foliations, and the proof holds in the pA-setting with no changes. We reproduce it here for completeness.

\begin{proof}
	Up to conjugation, we may assume that $\partial D^2$ is the boundary at infinity of~$\FF$.
	Consider two leaves (or face of singular leaves) $f,g$ of~$\FF$ sharing a same end $\theta\in S^1$. It is enough to prove that only finitely many of them are between $f$ and $g$ for this order.

	For that, one considers a simple path $\sigma\colon[0,1]\to\Int(D^2)\setminus \Sing(\FF)$ joining a point in $f$ to a point in $g$, and made of a concatenation of segments $a_0,b_0,a_1 \dots b_{k-1},a_k$ where the segments $a_i$ are transverse to~$\FF$ and the $b_j$ are leaf-arcs (possibly $k=0$ and $\gamma=a_0$). We denote by $D_\sigma$ the sub-disc of $D^2$ bounded by $\sigma$ and by the two rays contained in $f$ and $g$, starting at $\sigma(0)$ and $\sigma(1)$ and ending at $\theta$.

	If a regular leaf of~$\FF$ belongs to $E_L(\theta)$ and is between $f$ and $g$, then every nearby leaf will belong to $E_L(\theta)$ so that $E_L(\theta)$ would be uncountable. Thus a ray of regular leaf entering $D_\sigma$ needs to exit $D_\sigma$.

	Any element in $E_L(\theta)$ between $f$ and $g$ crosses $\sigma$. Otherwise it would be contained in $D_\sigma$ and should have both ends at $\theta$.

	We will prove by induction on $k$ that the number of element of $E_L(\theta)$ between $f$ and $g$ ($f$ and $g$ excluded) is at most $k-1$.

	Let us see that $k=0$ is not possible. If we had $k=0$, then the segment $\sigma$ is transverse to~$\FF$ and therefore cuts every leaf in at most $1$ point. Thus every leaf crossing $\sigma$ has an end on $\theta$ so that $E_L(\theta)$ is uncountable, contradicting the definition of $L_\infty(\FF)$.

	Assume now  $k=1$.
	Reason by contradiction and assume that there exists a leaf $h\in E_L(\theta)$ between $f$ and $g$. Let $\rho\subset D_\sigma$ be the smallest leaf ray of $h$ that starts on a point $x$ on $\sigma$ and goes to $\theta$. We consider the case when $x$ belongs to $a_0$. The case $x\in a_1$ is similar, and the case $x\in b_0\setminus(a_0\cup a_1)$ is impossible (the ray would contains an endpoint of $a_0$ or $a_1$ and so would not be the smallest).
	Let $t=\sigma^{-1}(x)$ and let $l$ be a leaf ray that starts on the point $\sigma(s)$ for some $s\in[0,t[$ and that enters $D_\sigma$ at that point. Then $l$ is trapped between $f$ and $\rho$, and intersects $a_0$ no more than once, so it ends on $\theta$. It contradicts the countability of $E_L(\theta)$, so there exists no leaf between $f$ and $g$ that ends on $\theta$.

	Assume that the assertion holds up to $k-1\geq 1$. Let $h\in E_L(\theta)$ be between $f$ and $g$ and consider the ray $\rho$ defined above. As explain above, the starting point $x$ of $\rho$ belongs to one of the $a_i$ and that it is the unique intersection point of $\rho$ with $\sigma$.
	For the same reason as above, $x$ can not lie on $a_0$ or $a_k$. So it lies on $a_i$ for some $0<i<k$. Thus $x$ cuts $\sigma$ into two segments $\sigma_0$ which has $i+1$ transverse segments and $\sigma_1$ with $k-i+1$ transverse segments (with possibly one extreme transverse segment reduced to a point). Thus we may apply the induction hypothesis to the pair $f,h$ and to the pair $h,g$,
	getting at most $i-1$ elements of $E_L(\theta)$ between $f$ and $h$ and
	$k-i-1$ between $h$ and $g$.  This shows that $E_L(\theta)$ has at most $k-1$ elements between $f$ and $g$, ending the proof.
\end{proof}

\begin{remark}
	One can obtain a dense pre-lamination $L'$ that has not few common ends by starting from a dense pre-lamination $L$ of $S^1$ that has few common ends, take a segment $I\subset S^1$ which does not contain the two ends of any leaf of $L$, and contracts $I$ into a single point. The resulting space is also a circle, and the image of $L$ is a dense pre-lamination $L'$ that has uncountably many ends inside the point image of $I$. If $L$ is complete or has no bad accumulation, so is $L'$. Part of the theory developed in the next subsection still holds to $L'$.
\end{remark}

%% file: planaroflamination.tex
\section{The planar structure of a shell/star pre-lamination} 

In the whole section, we assume that $L$ is a pre-lamination satisfying a list of properties. Consider the properties:
\begin{enumerate}
	\item (density condition) $L$ is dense,
	\item (countability) For every $x\in S^1$ the set of $y\in S^1$ for which $(x,y)\in L$ is at most countable,
	\item $L$ is a shell pre-lamination,
	\item[3 bis.] $L$ is a shell/star pre-lamination, and a leaf may belong to at most one star,
	\item $L$ has no bad accumulation.
\end{enumerate}

The union of the properties 1 to 4 (without the 3 bis) is call the property~$\Pre$. The union of the properties 1, 2, 3 bis and 4 is called $\Pre^*$. A pre-lamination satisfying $\Pre$ or respectively $\Pre^*$ is called a \emph{$\Pre$-pre-lamination}, or respectively a \emph{$\Pre^*$-pre-lamination}.
They correspond respectively to non-singular foliation on the plane, and $pA$-foliations.

\subsection{Simple $L$-interval}

We develop the technical tools to describe the leaf space of a $\Pre^*$-pre-lamination.
Fix a dense shell/star pre-lamination $L$.

A \emph{separatrix of $L$} is an unordered pair of adjacent boundary components of a star component of $L$. That is if $\Delta$ is a star component, bounded by the leaves $f_1\cdots f_n$ in cyclic order, then $\{f_i,f_{i+1}\}$ is a separatrix.

\begin{definition}
	We call a \emph{$L^*$-leaf} any subset of leaves of $L$ that is either a singleton of a leaf that does not bound a star component, is or a separatrix of $L$.
	We denote by $\Leaf^*(L)$ the set of $L^*$-leaves.
\end{definition}

If $L$ is a shell pre-lamination, then $L=\Leaf^*(L)$.
The set $\Leaf^*(L)$ will later be given a topology, whose open subsets are given by the following notion.

\begin{definition}\label{d.interval}
	Let $L$ be a $\Pre^*$-pre-lamination. A non-empty subset $\II\subset L$ is called \emph{an open $L$-interval} if the following hold:
	\begin{itemize}
		\item given any three leaves $f,g,h\in\II$ one of them separates the two others,
		\item given $f,g$ in $\II$, every leaf $h$ of $L$ separating $f$ from $g$ belongs to $\II$,
		\item given $f$ in $\II$, there exist two leaves $g,h\in\II$ that are separated by $f$,
		\item given $f,g$ in $\II$, either there exists $h\in L$ separating $f$ from $g$, or $\{f,g\}$ is a separatrix of $L$.
	\end{itemize}
\end{definition}

A $L$-interval is illustrated in Figure \ref{fig-interval}, arising from a foliation on $\Int(D^2)$. The complementary regions are represented in ligh blue. The corresponding leaves of the foliation are contained in a blue region, for visual purpose. 

\begin{figure}
	\begin{center}
		\begin{picture}(90,38)(0,0)
			\put(0,5){\includegraphics[width=90mm]{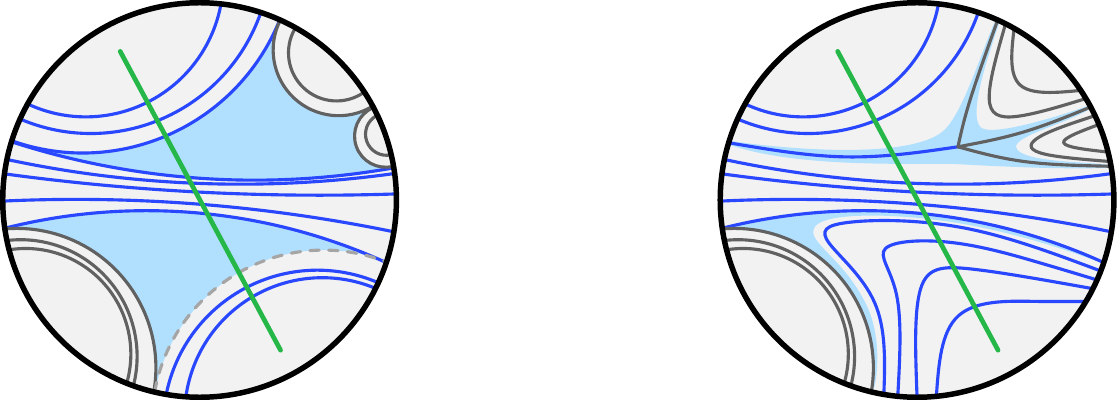}}
			\put(5.5,1){$L$-interval of}
			\put(2,-3){a pre-lamination}
			\put(56.5,1){corresponding leaves}
			\put(62.5,-3){for a foliation}
		\end{picture}
	\end{center}
	\caption{$L$-interval and transverse curve.}
	\label{fig-interval}
\end{figure}
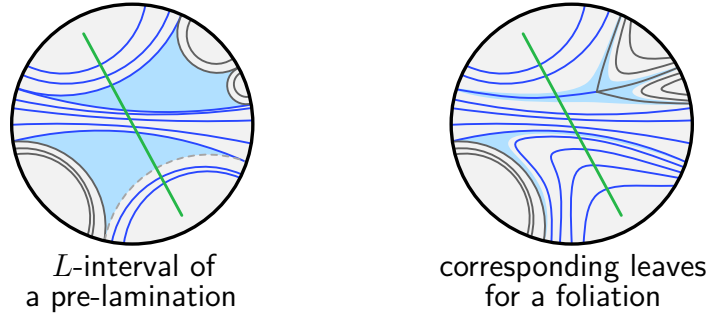

\begin{lemma}\label{l.interval*}
	Let $\II$ be an open $L$-interval and $f\in\II$ be a leaf that bounds a star component $\Delta$. Then there is a unique boundary component $g\neq f$ of $\Delta$ for which $\{f,g\}$ is a separatrix of $L$ and $\{f,g\}$ lies in $\II$.
\end{lemma}

\begin{proof}
	From the item 4, $f$ separates two leaves $h,h'$ in~$\II$. Up to switching them, $h$ belongs on the side of $f$ that also contains $\Delta$. So either $h$ is a boundary component of $\Delta$, and we set $g=h$, or there is another boundary component $g$ of $\Delta$ that separates $f$ from $h$. By definition, $g$ belongs to $\II$, so $\{f,g\}$ lies in $\II^*$. The leaf $g$ is unique since three boundary components of a star component do not separates each others.
\end{proof}

Given a open $L$-interval $\II$, the set of $L^*$-leaves that are included in $I$ is called an \emph{open $L^*$-interval} and will be denoted by $\II^*$.

An open $L$-interval $\II$ is said \emph{simple} if there exists an open geodesic $\delta$, transverse to the leaves of $L$, that satisfies $\II=\{l\in L| l\cap\delta\neq\emptyset\}$. When a $L$-interval $\II$ is simple, we also say that $\II^*$ is a simple open $L^*$-interval. We will see later that all $L$-intervals are simple, but it is not easy to prove right away. So we use this notion as a technical tool for now.

\begin{lemma}\label{l.simple-interval-singulier}
	Let $L$ be a $\Pre^*$-pre-lamination. Every $L^*$-leaf belongs to a simple $L^*$-interval.
\end{lemma}

\begin{proof}
	Given a $L^*$-leaf, we build a simple $L$-interval that contains it. We take a $L^*$-leaf, which is either a non-separatrix leaf $f$, or a separatrix $\{f,g\}$ of $L$.
	Let $\gamma\subset\Int(D^2)$ be a complete geodesic that satisfies the following properties:
	\begin{itemize}
		\item $\gamma$ crosses $f$,
		\item if $f$ bounds a shell component $\Delta$ on any side, $\gamma$ crosses the root of $\Delta$,
		\item if the $L^*$-leaf is a separatrix $\{f,g\}$, then $\gamma$ intersects $g$,
		\item if additionally $g$ bounds a shell on its other side, $\gamma$ crosses the root of that shell.
	\end{itemize}

	Notice that $\gamma$ is transverse to the leaves of $L$.
	Recall that $L$ admits no bad accumulations. So up to shortening $\gamma$ (by cutting it on each side along a leaf of~$L$), we may assume that for any shell $\Delta$ that intersects $\gamma$, $\gamma$ intersects the root of $\Delta$. By shortening it once more (by cutting on each side along the boundary of a star region), we may assume that for any star $\Delta$ that intersects $\gamma$, $\gamma$ intersects two adjacent boundary components of $\Delta$.

	Let $\II$ be the subset of leaves of $L$ that cross $\gamma$. The leaves in $\II$ are totally ordered from $\gamma$, so for any three leaves of $\II$, one of them separates the two others. So item 1 in the definition of $L$-interval is satisfied by $\II$.

	Given $h,h'\in \II$, every leaf separating $h$ from $h'$ crosses $\gamma$ and so belongs to~$\II$. So item 2 is satisfied.

	Let $h$ be a leaf in $\II$ and $S$ be a side of $h$. If $h$ is accumulated by leaves of $L$ on the side $S$, then some of them intersects $\gamma$ (since $\gamma$ is open) and so belongs to $\II$. If $h$ bounds a shell $\Delta$ on the side $S$, then $\gamma$ intersects the root of $\Delta$, and any leaf of $L$ close enough to the root $\Delta$ belongs to $\II$. If $h$ bounds a star $\Delta$ on the side $S$, then $\gamma$ intersects a boundary component $h'$ of $\Delta$ adjacent to $h$, and applying the same argument on the other side of $h'$, there exists a leaf of $\II$ on the side $S$. It follows that on each side of $h$, there is a leaf of $\II$, and these two leaves are separated by $h$. So item 3 is satisfied.

	Take two distinct leaves $h,h'\in\II$ and $D$ the connected component of $D^2\setminus h$ that contains $h'$. If there are leaves in $D$ that accumulate $h$, then they separate $h$ from $h'$ and cross $\gamma$. Otherwise, $h$ bounds a complementary region $\Delta\subset D$ on that side. If $\Delta$ is a shell then the root of $\Delta$ crosses $\gamma$ (by assumption on $\gamma$) and the leaves accumulating that roof cross $\gamma$ and separate $h$ from $h'$. If $\Delta$ is a star, then $\gamma$ crosses the next boundary component $h''$. Either $h'=h''$ holds, that is ${h,h'}$ is a separatrix, or $h''$ separates $h$ from $h'$. Thus item 4 is satisfied. So $\II$ is a simple open $L$-interval.
\end{proof}

Let $\II$ be an open $L$-interval and $\II^*$ the associated open $L^*$-interval. We equip $\II$ and $\II^*$ with two total orders. Assume first that $\II$ is simple, and let $\gamma\subset D^2$ be a geodesic that intersects exactly the leaves of $L$ that belong to $\II$. Let us choose an orientation on $\gamma$, and order $\II$ accordingly to that orientation. That is given two leaves $f,g\in\II$, we set $f\leq g$ if the orientation on $\gamma$ goes from $f$ to $g$. When $\II$ is not simple, we proceed as follows. We fix a based leaf $f_0\in\II$, and choose a co-orientation on $f_0$. For any leaf $f\in\II$, choose the co-orientation on $f$ for which there exists an oriented geodesic that intersects transversely $f$ and $f_0$, with two positive signs. Then given two leaves $f,g\in\II$, We set $f\leq g$ if $g$ is on the positive side of $f$.

Note that $(\II,\leq)$ may admit adjacent points, that is two distinct leaves $f,g\in\II$ that satisfy that no leaf $h\in\II$ has $f<h<g$. This is the case exactly when $\{f,g\}$ is a separatrix, that then belongs to $\II^*$. This has a consequence: no leaf has two adjacent point. So the order on $\II$ induces an well-defined total order $\leq$ on $\II^*$ that admits no adjacent pair of points. Note that for any $f,g,h\in\II^*$, $f<g<h$ holds if and only if $g$ separates $f$ and $h$ inside $D^2$.

\begin{lemma}\label{l.transversal-continuous}
	Let $\II^*$ be an open $L^*$-interval. Then there is an increasing bijection from $(\II^*,\leq)$ to $(\RR,\leq)$.
\end{lemma}

An heuristic when $\II$ is simple, given by a geodesic $\gamma$, is as follows. Given a complementary component $I$ of $\II\cap\gamma$, denote by $\wt I$ the union of $I$ and of the end points of $I$ that lies on $\II$ (potentially only one end points lies on $\II$), and contracts $\wt I$ to a point. Then the resulting topological space is homeomorphic to $\RR$.

\begin{proof}
	Take a pair $f,g\in\II^*$ with $f<g$ and denote by $]f,g[$ the set of $h\in\II^*$ that satisfies $f<h<g$. We build an increasing function $F\colon\RR\to]f,g[$. For that, first observe that $]f,g[$ is not empty, since $\II^*$ admits no adjacent pair of points. And even $]f,g[$ is infinite using the same argument inductively. Let $Q$ be a countable and dense subset of $]f,g[$. By dense we mean that for any leaf $h\in\II$ that lies in between $f$ and $g$, there exists a sequence of points of $Q$ that converges toward $h$ for the Hausdorff topology.
	Note that $Q$ exists since any infinite subset of $D^2$ admits a dense subset.

	Since $Q$ is countable and totally ordered, there exists an increasing bijection $F_{|Q}\colon]f,g[\to\QQ$. Then for any element $h\in]f,g[$, we have
	$$\sup_{x\in]f,h[\cap Q}F_{|Q}(x)\leq\inf_{x\in]f,h[\cap Q}F_{|Q}(x)$$
	and we actually have equality since $F_{|Q}$ has a dense image. We set $F(x)=\sup_{x\in]f,h[\cap Q}F_{|Q}(x)$.
	By construction, $F$ is non-decreasing. The facts that $F$ is non-decreasing and $Q\xrightarrow{F}\QQ$ is bijective implies that $F$ is injective.

	To see that it is surjective, take $x\in\RR\setminus\QQ$ and $y_n\in\QQ$ a sequence that converges toward $x$ from below. Then the $L^*$-leaf $f_n=F^{-1}(y_n)$ converges inside $D^2$ on a geodesic $\gamma$. We consider several cases on $\gamma$.

	Assume first that $\gamma$ is a leaf of $L$. If $\gamma$ does not bound a star component, then $\gamma$ belongs to $\II'$ as it separates $f$ and $g$. Then $F(\gamma)$ is the supremum of $F(f_n)=y_n$, that is $F(\gamma)=x$ holds. When $\gamma$ bounds a star component $\Delta$, there exists another boundary component $\gamma'$ of $\Delta$ that satisfies that $\{\gamma,\gamma'\}$ lies inside $\II^*$, and thus inside $]f,g[$. Then $F(\{\gamma,\gamma'\})=x$ follows from the same argument.

	Assume now that $\gamma$ is not a leaf of $L$. It is then the root of a shell component $\Delta$. One boundary component $h$ of $\Delta$ separates $\gamma$ and $g$. Then $h$ belongs to $\II$, and the same argument as above show that the image of $h$ (or a separatrix that contains $h$) by $F$ is equal to $x$. As a conclusion, $F$ is surjective.

	We proved that given any two distinct $L^*$-leaves $f,g\in\II^*$, the set of leaves of $\II^*$ between $f$ and $g$ is ordered as $\RR$. The item 4 in the definition of open $L$-interval implies that $\II^*$ has no maximum nor minimum. It additionally admits a dense subset, so it is ordered as $\RR$.
\end{proof}

\begin{lemma}
	Any open $L$-interval is simple.
\end{lemma}

\begin{proof}
	Any leaf $f$ of $\II$ ends on two points which delimit in $S^1$ two compact intervals $I_f^+$ and $I_f^-$. We may choose these two intervals uniformly on $f$ so that given two leaves $f_1<f_2$ in $\II$, we have $I_{f_1}^+\subset I_{f_2}^+$ and $I_{f_2}^-\subset I_{f_1}^-$. Then the two intersections
	$$\bigcap_{f\in\II}I_f^+ \text{ and } \bigcap_{f\in\II}I_f^-$$
	are monotonous, so the intersections are non-empty and compact. Denote by $I^+$ and $I^-$ these intersections. We choose a point $\theta^\pm$ inside $I^\pm$. If $I^\pm$ is not reduced to a point, we choose $\theta^\pm$ in its interior. Then denote by $\delta_0$ the open geodesic that goes from $\theta^+$ to $\theta_-$. Assume temporally that the two sets $I^\pm$ are not reduced to a points, then $\delta_0$ intersects all the leaves of $\II$. In that case, denote by $\delta$ the convex hull (in $\delta_0$) of the intersection points of $\delta_0$ with the leaves of $\II$. From the definition of $L$-intervals, $\delta$ is an open interval inside $\delta_0$, and it intersects exactly the leaves of $L$ that are in $\II$. So $\II$ is simple.

	We need to prove the same conclusion when one or two of $I^\pm$ is reduced to a point. Assume first that only $I^-$ is reduced to a point, which is equal to $\theta^-$. We claim that no leaves of $\II$ end on $\theta^-$. Indeed if it was the case, let say $f\in\II$ ends on $\theta^-$, then any leaf $h\in\II$ on the side of $f$ that bounds $I_f^-$ also end on $\theta^-$. Then Lemma~\ref{l.transversal-continuous} implies that there is an uncountable amount of leaves ending on $\theta^-$, which contradicts the property $\Pre^*$. Thus no leaf of $\II$ ends on $I$, and all leaves of $\II$ intersects transversely $\delta_0$. The same argument as above concludes this case.

	Assume now that the two intervals $I^\pm$ are reduced to a point. The same arguments shows that no leaves of $\II$ end on $\theta^\pm$, and that additionally $\theta^-$ and $\theta^+$ are separated by any leaves of $\II$. So the same argument concludes the proof.
\end{proof}

\subsection{The $L^*$-interval topology}

In this section, we equip the set of $L^*$-leaves with a topology, whose basis are the open $L^*$-intervals.

\begin{lemma}\label{l.intersection-intervals}
	Given to open $L^*$-intervals $\II^*_1,\II^*_2$, the intersection $\II^*_1\cap \II^*_2$ is either empty of an open $L^*$-interval itself.
\end{lemma}

\begin{proof}
	Let $\II^*_1,\II^*_2$ be two open $L^*$-intervals, and $\II_1,\II_2$ the corresponding open $L$-intervals. We assume that $\II^*_1\cap \II^*_2$ is not empty.

	Let $F\colon\II^*_1\to\RR$ be an increasing bijection given by Lemma~\ref{l.transversal-continuous}. It follows from the item 2 in the definition of open $L$-intervals that $I=F(\II^*_1\cap\II^*_2)$ is a non-empty interval inside $\RR$. We first prove by contradiction that $I$ is open. So assume that $I$ admits a maximum $x=\max I$ and write $f=F^{-1}(x)$. We denote by $D$ the connected component of $D^2\setminus f$ that contains $F^{-1}(]x,+\infty[)$. Let us consider two cases on $f$ and $D$.

	First assume that $f$ is accumulated by leaves of $L$ inside $D$. From item 3, there exists a leaf $g$ in $D$ that lies inside $\II_2$. Take a leaf $h$ in $D$ very close to $f$. Then it separates $f$ and $g$, so it belongs to $\II_2$. Similarly if $h$ is close enough to $f$, it belongs to $\II_1$. Since $L$ has at most countably many separatrix, $h$ may be chosen to be a non-separatrix. Then $F(h)>F(f)=x$ holds true, which contradicts our assumption. So this case is not possible.

	Secondly, assume that $f$ bounds a shell region $\Delta$ inside $D$. Then both $\II_1$ and $\II_2$ contains any leaf of $L$ that are close enough to the root of $\Delta$. The same argument yields a contradiction in this case too.

	I follows from above that $I$ has no maximum. Similarly, $I$ has no minimum, and it is an non-empty open interval. It easily follows that $\II^*_1\cap\II^*_2$ is an open $L^*$-interval. We let the details to the reader.
\end{proof}

We equip $\Leaf^*(L)$ with the topology generated by open $L^*$-intervals: an open subset is a union of $L^*$-intervals.
An immediate corollary of Lemma~\ref{l.transversal-continuous} and Lemma~\ref{l.intersection-intervals} is the following.

\begin{lemma}\label{l.R}
	Every open $L^*$-interval $\II^*$ is homeomorphic to $\RR$.
\end{lemma}

We start to build a singular planar structure on $\Leaf^*(L)$.

\begin{lemma}\label{l.singular-leafspace}
	Let $L$ be a $\Pre^*$-pre-lamination. Then $\Leaf^*(L)$ is a connected non-Hausdorff $1$-manifold with the following properties:
	\begin{itemize}
		\item every non-separatrix point in $\Leaf^*(L)$ disconnects $\Leaf^*(L)$,
		\item for every star component with $n$ edges, the corresponding set of $n$ separatrixes disconnects $\Leaf^*(L)$ in $n$ components,
		\item two separatrixes (pair of leaves) sharing a leaf are non-separated.
	\end{itemize}
\end{lemma}

\begin{proof}
	Every point in $L^*$ admits a neighborhood which is a a simple open $L^*$-interval, homeomorphic to $\RR$, and one easily checks that $L^*$ has a countable basis of its topology. Thus $\Leaf^*(L)$ is a non-Hausdorff $1$-dimensional manifold.

	We prove that $\Leaf^*(L)$ is connected.
	Let $\II$ be an open $L$-interval. We denote by $D_\II$ the union of the leaves in $\II$, and of the complementary region that are bounded by a leaf of $\II$. Observes that $D_\II$ is connected, since any leaf or complementary region that separates two leaves of $\II$ are actually inside $D_\II$. Also $D_\II$ is open. Indeed for any $x\in D_\II$, if $x$ belongs to a complementary region $\Delta$, then $\Delta$ is a neighborhood of $x$ and it is included in $D_\II$. If $x$ is a leaf of $L$, then any leaf, and also any complementary region, close enough to $x$ lies in $D_\II$.

	Given a connected component $A\subset\Leaf^*(L)$, denote by $D_A$ the union of the sets $D_\II$ for all open $L$-interval $\II$ that are included in $A$. If $X$ is the set of connected components of $\Leaf^*(L)$, then the set of $D_A$ for $A\in X$ yields a partition of $D^2$ in open subsets. It follows from the connectedness of $D^2$ that $\Leaf^*(L)$ is connected.

	We now start to prove item A in the lemma.
	Every non-separatrix $f\in\Leaf^*(L)$ corresponds to a leaf of $L$ that disconnects the disc in two discs $D_1,D_2$. Each $D_i$ containing some leaf of $L$, so denote by $L_{D_i}$ the set of leaves of $L$ that lie inside $D_i$. For any $L$-interval $\II$, either $\II$ lies inside one of the $L_{D_i}$, or $\II\setminus\{f\}$ splits in two $L$-interval, each of them is contained inside $L_{D_i}$. It follows that $L_{D_i}$ projects on an open subset of $\Leaf^*(L)$. Thus $\Leaf^*(L)\setminus\{f\}$ has at least two connected components. As $f$ splits a chart in two components, $f$ splits $\Leaf^*(L)$ in exactly two components.

	The third point in the Lemma~follow from Lemma~\ref{l.simple-interval-singulier}.

	Let $\Delta$ be a star component with $n$ edges. The boundary $\partial\Delta$ splits the disc in $n+1$ connected components, one of them contains no leaf of $L$ (the interior of $\Delta$), and the $n$ other are bounded each by one leaf pf $L$. The same argument as above yields that each of these $n$ components project on an open subset of $\Leaf^*(L)$. Thus the set of separatrixes of $\Delta$ splits $\Leaf^*(L)$ in at least $n$ connected components.
	Let $f_i$ be the leaves on the boundary of $\Delta$, cyclically ordered.
	For each separatrix $(f_i,f_{i+1})$, take an $L$-interval $\II_i$ that contains that separatrix.
	The union $\cup_i\II_i$ is a neighborhood of the set of separatrixes of $\Delta$ inside $\Leaf^*(L)$.
	Observe that $\II_i$ and $\II_{i+1}$ intersects on the free side $f_{i+1}$.
	So the set of $n$ separatrixes disconnects $\cup_i\II_i$ in exactly $n$ open-subsets, one for the free side of each $f_i$.
	It follows that the set of separatrixes of $\Delta$ splits $\Leaf^*(L)$ in $n$ components.
\end{proof}

\begin{corollary}\label{c.leafspace}
	If $L$ is a $\Pre$-pre-lamination, then $\Leaf^*(L)$ is a connected simply connected non-Hausdorff $1$-manifold.
\end{corollary}

\begin{remark}[Addendum to Lemma~\ref{l.singular-leafspace}]\label{r.leafspace}
	In the proof of Lemma~\ref{l.singular-leafspace} we proved that every non-separatrix leaf $l\in \Leaf^*(L)$ cuts $\Leaf^*(L)$ in two open subsets which are the sets of leaves contained in each of the connected components of $\Int(D^2)\setminus l$.
\end{remark}

\subsection{The branchings of $\Leaf^*(L)$}

Let $L$ be a $\Pre^*$-pre-lamination and $\Leaf^*(L)$ its leaf space. Take two $L^*$-leaves $f,g\in\Leaf^*(L)$. By definition, $f$, $g$ are non separated if every neighborhood of $f$ intersects every neighborhood of $g$, or equivalently, every open $L^*$-interval $\II^*_f$,$\II^*_g$ containing $f$ and $g$, respectively, have a non trivial intersection. According to Lemma~\ref{l.intersection-intervals} the intersection of two $L^*$-intervals is an interval: thus one may choose
$\II^*_f$ and $\II^*_g$ so that they coincide on one side of $f$ and $g$, and are disjoint on the other side. One says that $f$ and $g$ are \emph{non-separated on that side}. This shows that $\Leaf^*(L)$ is tame, so that the notion of branchings is well-defined.

\begin{lemma}\label{l.non-separated}
	Let $x,y\in\Leaf^*(L)$ be two distinct $L^*$-leaves. Then $x$ and $y$ are non separated points if and only if there are leaves $f,g$ of $L$ in $x$ and $y$ so that either:
	\begin{itemize}
		\item $f$ and $g$ are boundary component of the same shell,
		\item or $x$ and $y$ are separatrixes and $f=g$.
	\end{itemize}
\end{lemma}

\begin{proof}
	Given a shell $\Delta$, any interval containing a leaf that bounds $\Delta$ also contains the leaves accumulating the root of $\Delta$. This implies that the boundary components of $\Delta$ belongs to a same branching of $\Leaf^*(L)$. Two successive separatrixes of a same star share a leaf $f\in L$. Thus, any $L$-interval that contains $f$ also contains the leaves accumulating $f$, or the leaves accumulating the root of $\Delta$ when $f$ bounds a shell $\Delta$.

	We now prove the converse. Let $x,y\in L^*$ be non-separated points. Then there is a complementary region $\Delta$ that contains leaves in both $x$ and $y$.
	If $\Delta$ is a shell, then the first item in the Lemma~is satisfied.

	Assume that $\Delta$ is a star, so $x$ and $y$ are two separatrixes of $\Delta$. Take two open $L$-intervals $\II_x,\II_y$ that contains $x$ and $y$ respectively, and which are not disjoint since $x$ and $y$ are non-separated. Note that $\II_x$ intersects the boundary of $\Delta$ only on $x$ (see Lemma~\ref{l.interval*}), and $\II_x\setminus\partial\Delta$ is contained in the two discs bounded by $x$ that are disjoint with $\Delta$. Thus if $x$ and $y$ were not successive separatrix, then the corresponding four discs would be disjoint, and so would be $\II_x$ and $\II_y$. But $\II_x$ and $\II_y$ are not disjoint, so $x$ and $y$ are successive separatrixes, and they contain a common leaf of $L$.
\end{proof}

As a straightforward corollary, one gets:

\begin{lemma}\label{l.branching}
	Let $L$ be a $\Pre^*$-pre-lamination. Each branching of $\Leaf^*(L)$ is either:
	\begin{itemize}
		\item the set $\mu$ of leaves or separatrixes which bounds a same shell $\Delta$ (a separatrix bounds $\Delta$ if on of its two leaves bounds $\Delta$),
		\item or a pair of successive separatrixes of a common star.
	\end{itemize}

	Additionally, every cyclic branching is the set of separatrixes of a star.
\end{lemma}

\subsection{Ordering the branchings}

We equip branchings and cyclic branchings of $\Leaf^*(L)$ with orders.

Let $\Delta$ be a star component of $L$. Its separatrixes form a cyclic branching $\nu_\Delta$. The positive (anticlockwise) orientation of the boundary $\Delta$ induces a cyclic order of the leaves in $\partial\Delta$, which induces a cyclic order $\preceq^{\circlearrowleft}_{\nu_\Delta}$ on $\nu_\Delta$.

Let now $\Delta$ be a shell component of $L$, and $\gamma$ be its root. The set of leaves in the boundary of $\Delta$ projects in $\Leaf^*(L)$ on a branching $\mu=\mu_\Delta$. 
Let $I\subset\partial D^2$ be the closed interval that joins the endpoints of $\gamma$ and that contains all the endpoints of the leaves in $\partial\Delta$. The clockwise cyclic orientation of $\partial D^2$ induces an orientation of $I$, which induces a total order $\preceq_\Delta$ on the leaves in $\partial\Delta$. It does not induce yet an order on the $L^*$-leaves that bound $\Delta$, since two adjacent separatrixes that bounds $\Delta$ are not order that way. Take a boundary component $f$ of $\Delta$ that bounds a star $\Sigma$ on its other side. Then $f$ is contained in $2$ separatrixes $\{f,g\}$ and $\{f,g'\}$ of $\Leaf^*(L)$ which are successive for $\preceq^{\circlearrowleft}_{\nu_\Sigma}$. If $\{f,g\}$ comes just before $\{f,g'\}$ for $\preceq^{\circlearrowleft}_{\nu_\Sigma}$, then we will set $\{f,g\}\preceq_{\mu}\{f,g'\}$.

More generally, given two $L^*$-star distinct $x,y\in\mu$, we set $x\preceq_\mu y$ if one of the following holds:
\begin{itemize}
	\item the leaf $x\cap\partial\Delta$ is less than $y\cap\partial\Delta$ for the order $\preceq_\Delta$,
	\item $x,y$ are two successive separatrixes of a common star $\Sigma$, and $x\preceq_{\nu_\Sigma} y$ holds.
\end{itemize}

Notice that $\preceq_\mu$ is a total order. It makes $(\Leaf^*(L),\{\prec_\mu\})$ a singular planar structure, which we say \emph{associated to $L$}. If $L$ has no stars, then $\Leaf^*(L)$ is a (non-singular) planar structure.

\begin{remark}\label{r.induced}
	Let $L_1,L_2$ be two $\Pre^*$-pre-laminations and $\psi\colon S^1\to S^1$ an orientation preserving homeomorphism conjugating $L_1$ to $L_2$. Then $\psi$ maps leaves on leaves, faces on faces, root of shells on root of shells, stars on stars, etc. Thus $\psi$ induces an isomorphism $\varphi\colon \Leaf^*(L_1)\to \Leaf^*(L_2)$ of singular planar structures, called \emph{the isomorphism induced by $\psi$}.
\end{remark}

\subsection{$\Pre^*$-pre-laminations with the same planar structure}

Recall that given a $\Pre^*$-pre-lamination $L$, $\Leaf^*(L)$ is the set $L^*$ endowed with the $L^*$-interval topology and with orders on the branching.

\begin{theorem} \label{t.prelam-planar}
	Let $L_1,L_2$ be two $\Pre^*$-pre-laminations and $\varphi$ an isomorphism of singular planar structures from $\Leaf^*(L_1)$ to $\Leaf^*(L_2)$.
	Then, there is a unique orientation preserving homeomorphism $\psi$ of $S^1$ which satisfies $\psi(L_1)=L_2$ and so that $\varphi$ is the isomorphism induced by $\varphi$ (see Remark~\ref{r.induced}).
\end{theorem}

Let $L$ be a pre-lamination of $S^1$. A ray is a proper embedding of $[0,+\infty[$ in $\Int(D^2)$, that converges toward a point in the boundary at $+\infty$.
We denote by $\cR(L)$ the set of \emph{germs at infinity of rays} contained in a leaf of the geodesic realization $G(L)$. Note that $\cR(L)$ is canonically cyclically ordered, as follows. Given any three distinct rays $r_1,r_2,r_3$ and any $d<1$ close enough to 1, the circle $d\cdot S^1$ (of radius $d$) intersects the three rays $r_i$ at one point $x_i$ each, and the cyclic order of $x_1,x_2,x_3$ does not depend on $d$. The germs of rays are ordered using these cyclic order for any $d$ close enough to 1.
Recall that a $L^*$-leaf is either a leaf or a separatrixes, so it has either two or three germs of rays.

\begin{lemma}\label{l.prelam-planar}
	With the hypotheses of Theorem~\ref{t.prelam-planar}, there is an increasing bijection of $\psi_\cR\colon \cR(L_1)\to\cR(L_2)$ with the following property:

	Given a $L^*$-leaf $x\in\Leaf^*(L_1)$, $\psi_\cR$ maps the germs of rays contained in $x$ to the germs of rays contained in $\varphi(x)\in\Leaf^*(L_2)$.
\end{lemma}

The map $\psi$ above can not easily be used to recover a map between the circles, since several germs of rays can end up on the same point. The next lemma simplify this idea by restricting $\psi$ to a simpler set.

Denote by $\reg(L)\subset S^1$ the set of endpoints of regular leaves (accumulated on both sides) of $L$.

\begin{lemma}\label{l.regular-leaves}
	Given a $\Pre^*$-pre-lamination $L$. Then $\reg(L)$ is a dense in $S^1$ and every point of $\reg(L)$ is the endpoint of an unique leaf in $L$.
\end{lemma}

We will prove the Lemma~further down. For now, let us assume them.

\begin{proof}[Proof of Theorem~\ref{t.prelam-planar} assuming the Lemma~\ref{l.prelam-planar} and ~\ref{l.regular-leaves}]
	The map $\psi_\cR$ given by Lemma~\ref{l.prelam-planar} induces a (cyclically) increasing bijection $\wt\psi$ between $\reg(L_1)$ and $\reg(L_2)$. As each subset $\reg(L_1)$ and $\reg(L_2)$ is dense in $S^1$, $\wt\psi$ extends in a unique way as a homeomorphism $\psi\colon S^1\to S^1$, satisfying all the announced properties. Additional $\psi$ is orientation preserving since the map in Lemma~\ref{l.prelam-planar} is increasing.

	To prove the uniqueness of $\psi$, take $\psi'$ another map as in the conclusion of the theorem. Then $\psi^{-1}\circ\psi$ is an increasing homeomorphism of $S^1$ which induces the identity of $\Leaf^*(L_1)$. It preserves the set of endpoints of any leaf of $L_1$, and thus is equal to the identity since it is increasing.
\end{proof}

\begin{proof}[Proof of Lemma~\ref{l.regular-leaves}]
	Let us first prove that a regular leaf $f$ of $L$ does not share any endpoint with any other leaf. Let $g$ be another leaf of $L$.
	Any leaf close enough to $f$, and on its side that contains $g$, separates $f$ and $g$. Since $f$ is regular, there are infinity many leaves that separates $f$ and $g$, and even uncoutanbly many according to Lemma~\ref{l.R}. And since at most countably many of them share an end with $f$, $g$ does not share an end with $g$.

	We now prove that the endpoints of leaves in $E$ are dense in $S^1$.
	Let $I\subset S^1$ be a non-empty interval.
	Recall that there are at most countably many non-regular leaves.
	We consider two cases.
	Assume first that there exists a leaf having both endpoints in $I$ then its bound in $\Leaf^*(L)$ an open set $\Omega$ of leaves having both endpoints in $I$. Then $\Omega$ is uncountable, it as at most countably many non-regular leaves and thus it contains a leaf in $E$.

	Assume now that no leaves have its two endpoints inside $I$. The endpoints of the leaves of $L$ are dense in $I$, so there exists a monotonous sequence of leaves $f_n$ with an end in $I$, and that accumulates on some leaf or some root of shell of $L$. In the two cases, there exists an open $L$-interval that contains the leaf $f_n$ for all $n$ large enough. Recall that open $L^*$-intervals are uncountable. So there exist unaccountably many leaves that separate two of the leaf $f_n,f_m$ with large $n,m$, and so that end on $I$. It follows from the same argument that there is a leaf of $E$ that ends inside $I$.
\end{proof}

\begin{observation}\label{o.interval}
	Let $\II^*_i$, $i=1,2$ be $L^*_i$-intervals and $\varphi^*\colon \II^*_1\to \II^*_2$ be a homeomorphism mapping bijectively separatrixes in $\II^*_1$ onto separatrixes in $\II^*_2$, and let $\II_1,\II_2$ be the corresponding $L_i$-intervals. Then the following holds:
	\begin{enumerate}
		\item there exists a unique monotonous bijection $\varphi\colon\II_1\to\II_2$ induced by $\varphi^*$. That is for any for any $L_1^*$-leaf $f$ of $\II_1^*$, $\varphi$ maps the (one or two) leaves $f$ onto the leaves of $\varphi^*(f)$.
		\item 
		      There is a unique bijection $\psi\colon\cR(\II_1)\to\cR(\II_2)$ increasing for the cyclic orders and that induces $\varphi$. That is for any leaf $f\in\II_1$, $\psi$ maps the two germs of rays at infinity of $f$ into the ones of the leaf $\varphi(f)$.
		\item let $A\subset \II_1$ be any subset with at least $2$ elements. Assume that the map $\psi_A\colon \cR(A)\to \cR(\varphi(A))$ is an increasing bijection for the cyclic order, and that it induces $\varphi_{|A}$. Then $\psi_A$ is the restriction of $\psi$.
	\end{enumerate}
\end{observation}

As a consequence of Observation~\ref{o.interval} one gets

\begin{lemma}\label{l.homeo-bijection}
	Let $\varphi^*\colon\Leaf^*(L_1)\to \Leaf^*(L_2)$ be an isomorphism of the singular planar structures of two $\Pre^*$-pre-lamination $L_1,L_2$. Then $\varphi$ induces a bijection $\varphi\colon L_1\to L_2$ which maps open $L_1$-intervals onto open $L_2$-intervals as a one-to-one correspondence, and is monotonous in restriction to any $L_1$-interval.
\end{lemma}

\begin{proof}
	Note that $\varphi^*$ maps cyclic branchings onto cyclic branchings, so it maps separatrixes onto separatrixes (see Lemma~\ref{l.branching}).

	Let $\II^*$ be an open $L_1^*$-interval, $\II$ its corresponding open $L_1$-interval, and $\II'$ the open $L_2$-interval that corresponds to $\varphi^*(\II^*)$.
	Observation \ref{o.interval} item 1 asserts that there exists a unique monotonous bijection $\varphi_\II\colon \II\to \II'$.

	Let $\JJ^*$ be another $L_1^*$-interval, and $\JJ$ be as above, so that $\II\cap\JJ$ is not empty. Note that $\II^*\cap\JJ^*$ is not empty. Indeed take a leaf $f\in \II\cap\JJ$. If $f$ does not bound a start, then $\{f\}$ belongs to $\II^*\cap\JJ^*$. If $f$ bounds a star, then $\II^*\cap\JJ^*$ contains the leaves that accumulates on $f$, or on the root of $\Delta$ when $f$ bounds a shell $\Delta$ on its other side. So $\II^*\cap\JJ^*$ is a non empty open $L_1^*$-interval. The uniqueness in Observation \ref{o.interval} item 1 implies that $\varphi_\II$ and $\varphi_\JJ$ coincide on the leaves that appear in $\II^*\cap\JJ^*$.

	There are at most two leaves in $\II\cap\JJ$ that do not appear in $\II^*\cap\JJ^*$. Take $f$ one of them. Then $f$
	lies in two successive separatrixes $\{f,g_\II\}\subset \II^*$ and $\{f,g_\JJ\}\subset\JJ^*$. So $\varphi_\II(\{f,g_\II\})$ and $\varphi_\II(\{f,g_\JJ\})$ are two successive of $\varphi^*(\II^*)$. Denote by $f'$ their common leaf of $L_2$.
	Up to changing the orientation, $f$ is the largest leaf of $\II$ that is smaller that all the leaves in $\II\cap\JJ^*$. Thus $\varphi_\II(f)$ is the largest leaf of $\II'$ that is smaller than all the leaves in $\varphi^*(\II^*\cap\JJ^*)$. Thus we have $\varphi_\II(f)=f'$, and similarly $\varphi_\JJ(f)=f'=\varphi_\II(f)$.

	It follows from above that $\varphi_\II$ and $\varphi_\JJ$ coincide on there common domain. Thus the map $\varphi\colon L_1\to L_2$ defined as the common value of the maps $\varphi_\II$ satisfies the conclusion of the lemma.
\end{proof}

\begin{proof}[Proof of Lemma~\ref{l.prelam-planar}]
	Let $L_1,L_2$ be two $\Pre$-pre-laminations and $\varphi^*\colon\Leaf^*(L_1)\to\Leaf^*(L_2)$ be a homeomorphism which is increasing on every branching for the corresponding orders. According to Lemma~\ref{l.homeo-bijection}, $\varphi^*$ induces a bijection $\varphi\colon L_1\to L_2$ that maps open $L_1$-intervals onto open $L_2$-intervals, and that is monotonous on any open $L_1$-interval.

	For any open $L_1$-interval $\II$, Observation \ref{o.interval} item 2 yields a unique increasing bijection $\psi_\II\colon\cR(\II)\to\cR(\varphi(\II))$ that induces $\varphi_{|\II}$. According to the third item of the observation, for any other open $L_1$-interval $\JJ$, the maps $\psi_\II,\psi_\JJ$ coincide on their common domain. So there exists a map $\psi\colon\cR(L_1)\to\cR(L_2)$ that coincide with all $\psi_\II$, and so in particular that induces $\varphi$. By construction,
	$\psi$ in bijective

	We prove now that $\psi$ is increasing. Take three distinct points $x_i$ in $\cR(L_1)$ for $i=1,2,3$, so that $(x_1,x_2,x_3)$ is positively ordered, and $f_i$ a leaf of $L_1$ that ends on $x_i$. We prove that $(\psi(x_1),\psi(x_2),\psi(x_3))$ is positively ordered by considering several cases. First assume that one the leaves $f_i$, let say $f_1$, separates the two others. We co-orient $f_1$ so that $f_2$ in on the negative side of $f_1$, and $f_3$ on the positive side. We also assume that $x_1$ is on the left of $f_1$ (the other case is similar). It induces a co-orientation of $\varphi^*(f_1)$ for which $\varphi^*(f_2)$ is on the negative side, and $\varphi^*(f_3)$ on the positive side. Note that by construction of $\psi_\II$, for any open $L_1$-interval $\II$ that contains $f_1$, $\psi(x)$ lies on the left of $\varphi^*(f_1)$. It follows that $(\psi(x_1),\psi(x_2),\psi(x_3))$ is positively ordered.

	Assume now that two of the leaves, let say $f_1$ and $f_2$, are equal. Choose a co-orientation of $f_1$ so that $x_1$ is on the left of $f_1$, and $x_2$ on the right. Also assume that $x_3$ is on the positive side (the other case is similar). The same argument as above yields that these relative positions are preserved by $\varphi$ and $\psi$. Thus $(\psi(x_1),\psi(x_2),\psi(x_3)$ is positively ordered.

	Assume now that the three leaves are distinct, and non of theme separates the other. We claim that there exists a complementary region $\Delta$ that separates the three leaves $f_i$ in three connected components. Consider the set of leaves $A$ that separates $f_1$ and $f_2$. It is ordered as follows: $g,h\in A$ satisfy $g<h$ if $g$ if $f_1$ is closer to $g$ than to $h$. For $i=1,2$, let $A_i$ be the set of leaves $f\in A$ that separates $f_i$ and $f_3$. Since a leaf $f$ in $A$ have of the leaves $f_i$ on on side, and the two other leaves on its other side, $\{A_1,A_2\}$ is a partition of $A$. If $A_1$ is empty, then $f_1$ is not accumulated by leaves of $L$ on its side that contains $f_2$. Thus It bounds a complementary region $\Delta$. For $j=2,3$, let $\gamma_j$ be the boundary component of $\Delta$ that separates $f_1$ and $f_j$.
	If $\gamma_2=\gamma_3$ holds, then $\gamma_2$ can not be in $A$ since otherwise it would be in $A_1=\emptyset$. Thus $\Delta$ is a shell and $\gamma_2$ is its root. It follows that any leaf close enough to $\gamma_2$ separates $f_1$ and $f_3$, which contradicts $A_1=\emptyset$.
	Thus $\gamma_2$ and $\gamma_3$ and distinct, and $\Delta$ separates the three leaves $f_i$.

	When $A_1$ is not empty, it admits a supremum, which is a geodesic $\wt f_1$ in $\Int(D^2)$. The same argument as above applied to $\wt f_1,f_2,f_3$ yields a complementary region that separates $\wt f_1,f_2,f_3$, and thus $f_1,f_2,f_3$ too. Thus the claim holds true in any cases.

	The complementary region $\Delta$ corresponds either branching or a cyclic branching of $\Leaf^*(L_1)$ (see Lemma~\ref{l.branching}). So it corresponds a complementary region $\Delta'$ of $L_2$ that separates the leaves $\psi(x_1),\psi(x_2),\psi(x_3)$.
	By assumption, $\varphi$ preserves that order on the branching and cyclic branching. It follows immediately that $(\psi(x_1),\psi(x_2),\psi(x_3))$ is positively ordered. Therefore $\psi$ is increasing.
\end{proof}

\subsection{Few common ends}

The aim of this section is to show Lemma~\ref{l.few} below. Recall that $L$ has few common ends if of any $\theta\in S^1$, the set $E_L(\theta)$ of leaves ending on $\theta$ is ordered as an interval of~$\ZZ$.

\begin{lemma} \label{l.few}
	Let $L$ be a $\Pre^*$-pre-lamination. Then $L$ satisfies the few common ends property.
	Furthermore, two successive leaves in $E_L(\theta)$ either
	\begin{itemize}
		\item are two successive boundary components of a shell, and any shell contains at most $2$ boundary components in $E_L(\theta)$.
		\item or they form a separatrix of a star, and a star has at most one separatrix in $E_L(\theta)$.
	\end{itemize}
\end{lemma}

\begin{proof}
	Take $\theta$ in $S^1$.
	Observe that $E_L(\theta)$ intersects an open $L^*$-interval in at most one $L^*$-leaf. Indeed if a $L^*$-interval $\II$ was intersecting in two $L^*$-leaves $x,y$, the uncountably many leaves in $]x,y[\subset\II$ would end on $\theta$, which would contradicts the countable assumption in the definition of $\Pre^*$-pre-lamination.

	As a consequence, leaves in $E_L(\theta)$ cannot accumulate on a geodesic (either a leaf or a root of a shell) because any converging sequence contains infinitely many leaves in a same interval of $\Leaf^*(L)$. So $E_L(\theta)$ contains at most finitely many leaves between $f_0$ and $f_1$. It follows that $E_L(\theta)$ is ordered as an interval of $\ZZ$.

	Two leaves that are successive in $E_L(\theta)$ are not separated by any leaf of $L$. According to Lemma~\ref{l.separated} this implies that they are either:
	\begin{itemize}
		\item two boundary components of the same shell $\Delta$, and as they share an endpoint, they are successive for $\prec_\Delta$.
		\item or two boundary components of the same star, and again, as they share an endpoint they are successive and the pair is a separatrix.
	\end{itemize}

	Finally, given $3$ distinct leaves in $E_L(\theta)$, one of them separates the two others, which prevents them from bounding the same complementary component.
\end{proof}

\subsection{Orientations of a $\mathfrak{P}$-pre-lamination}

Given an oriented pair $(a,b)\in S^1\times S^1$ and $f$ its geodesic realization (oriented from $a$ to $b$), its positive side is the side on the left of $f$. The other component is its negative side. We co-orient $f$ from its negative side to its positive side. With this choice, the orientation of the disc is the orientation of $f$ concatenated with the co-orientation of $f$.

\begin{definition}
	Let $L$ be a pre-lamination. \emph{An orientation of $L$} is a subset $\overrightarrow{L}\subset S^1\times S^1\setminus\Diag$ that contains exactly one oriented leaf $\overrightarrow{f}$ for any non-oriented leaf $f$ of $L$, and so that for every open $L$-interval $\II$ of $L$, for every $f_1,f_2\in\II$, the positive (respectively negative) sides of $\overrightarrow{l_1}$ and $ \overrightarrow{l_2}$ have non-empty intersection.
\end{definition}

\begin{corollary}\label{c.orientation}
	A $\Pre$-pre-lamination has exactly two orientations.
\end{corollary}

\begin{proof}
	The orientation of a leaf determines the orientation of the leaves in any intervals containing it. The connectedness of $\Leaf^*(L)$ (see Corollary \ref{c.leafspace}) ensures that there is at most two orientations, and the simple connectedness of $\Leaf^*(L)$ ensures that an orientation of $L$ is uniquely determined by the co-orientation of any leaf of $L$.
\end{proof}

\begin{remark}
	An orientation of $L$ induces a co-orientation of each leaf, which induces an orientation of each open $L^*$interval, and thus induces an orientation of $\Leaf^*(L).$ Reciprocally an orientation of $\Leaf^*(L)$ induces a co-orientation of every leaf, hence an orientation of every leaf, and this orientation is coherent on each interval and so induces an orientation of $L$.
\end{remark}

%% file: universel.tex
\section{Universal foliations and planar structures} \label{sec-univ-fol}

Kaplan prove that for any planar structure, there is a foliation on the plane that induces that planar structure. We need to prove the same result for singular planar structure. For that, we develop the notion of universal (singular or not) foliation, that contains all other foliations of the plane. We reprove Kaplan's result in the process.

\subsection{Universal foliations}\label{s.universal-foliation}
We say that a co-oriented foliation~$\FF$ on $\RR^2$ is \emph{a universal foliation} if, given any co-oriented foliation $\mathcal{G}$ of $\RR^2$, there is an orientation preserving embedding $\psi\colon\RR^2\to \RR^2$ that satisfies
\begin{itemize}
	\item $\mathcal{G}=\psi^{-1}(\FF)$,
	\item the image $\psi(\RR^2)$ is an open subset of $\RR^2$ saturated for~$\FF$,
	\item $\psi$ maps the co-orientation of $\mathcal{G}$ on the one of~$\FF$. 
\end{itemize}
In other words, $\mathcal{G}$ is conjugated to the restriction of~$\FF$ to a connected simply connected,~$\FF$-saturated region of $\RR^2$.

We define now a foliation which we prove universal later.
Let $C\subset[0,1]$ be the standard dyadic Cantor set. 
We define a set $K\subset\RR^2$ as the union
$$K= \bigcup_{\substack{p,n\in\ZZ\\ q\in\NN_{>0}\\ |n|\geq q}}\left(C+2n,\tfrac{p}{q}\right).$$
It can be described as: for all $\tfrac{p}{q}\in\QQ$, the intersection of $K$ and of the line $\{y=\tfrac{p}{q}\}$ is the union of $\ZZ$-copies of the Cantor set $C$, all shifted by a multiple of 2, with a hole centered at $(x,y)$ of size $2n$.
The condition $|n|\geq q$ is chosen so that the following holds.

\begin{lemma}
	$K$ is closed inside $\RR^2$.
\end{lemma}

\begin{proof}
	Let $(x_k,y_k)$ be a sequence in $K$ that converges toward a point $(x,y)\in\RR^2$. We write $x_k=z_k+2n_k$ and $y_k=\tfrac{p_k}{q_k}$ with $z_k\in C$ and $p_k,q_k,n_k$ as in the definition of $K$. 
	By construction, we have:
	$$q_k\leq|n_k|=\tfrac{|x_k-z_k|}{2}\leq\tfrac{|x|+1}{2}$$
	when $k$ is large enough. So $q_k$ is bounded. It follows that the sequences $q_k,n_k,p_k$ are also bounded. So up to extraction, they converges. It immediately implies that $x$ lies in $K$.
\end{proof}


The subset $\RR^2\setminus K$ is connected and open. Equip $\RR^2\setminus K$ with the foliation by horizontal leaves.
Denote by $S$ the universal cover of $\RR^2\setminus K$, endowed with the lifted foliation $\FF_0$. The main objective of this section is :

\begin{theorem}\label{t.universal} 
	The foliation $\FF_0$ defined above is a universal foliation. 
\end{theorem}

As a straightforward corollary one gets next result, which may be classical but maybe not so well known: 

\begin{corollary}
	Any foliation on the plane is the pullback of the horizontal foliation of $\RR^2$ by a submersion. 
\end{corollary}

\subsection{Planar structure of a (non singular) foliation} \label{sec-leaf-space}

We denote by~$\FF$ a (non singular) foliation on $\RR^2$. It is always  orientable and co-orientable, so let us fix an orientation  and a co-orientation for~$\FF$, so that the orientation followed with the co-orientation gives back the usual orientation of $\RR^2$. 

The leaf space of the foliation~$\FF$, denoted by $\Leaf(\FF)$, is a well studied object (see \cite{KaplanI, KaplanII} for instance). As any transverse segment crosses each  leaf in at most $1$ point, one gets that every point of $\Leaf(\FF)$ admits a neighborhood which is a segment: in other words,  $\Leaf(\FF)$ is a (potentially) \emph{non-Hausdorff 1-manifold},  that is it admits a countable covering by open subset homeomorphic to~$\RR$ (see \cite{KaplanI}). Moreover, the transverse orientation of~$\FF$ induced an orientation of $\Leaf(\FF)$. For convenience, we will speak about non-Hausdorff 1-manifold even when the 1-manifold is Hausdorff. Finally, the simple connectedness of $\RR^2$ implies that $\Leaf(\FF)$ is simply-connected.

Given a connected simply-connected non-Hausdorff 1-manifold $\LL$, Kaplan proved that $\LL$ is the leaf space of some foliation of $\RR^2$, but the foliation is not unique in general. One gets back the uniqueness by adding an order at each branching.


Let~$\FF$ be a foliation and $\mu$ be a branching of~$\FF$. There are two kind of connected components of $\RR^2\setminus \mu$: the components bounded by one leaf in $\mu$ and a unique component bounded by all the leaves in $\mu$.  Take a leaf $f$ of~$\FF$ that belongs to the connected component of $\RR^2\setminus\mu$ that is bounded by all the leaves in $\mu$. Denote by $R$ the simply connected region of $\RR^2$ that is bounded by $\mu\cup l$. Given two distinct leaves $f_1,f_2$ in $\mu$, we set $f_1\preceq_\mu f_2$ if $(f, f_1, f_2)$ is clockwise cyclically ordered, as boundary components of the region $R$. One can see that it does not depend on the choice of $f$. It yields a total order $\preceq_\mu$ on the branching $\mu$. Call \emph{left-right order}  the order $\preceq_\mu$.   Equip the leaf space $\Leaf(\FF)$ with the planar structure $(\Leaf(\FF),\preceq_\cdot)$, where $\preceq_\cdot$ denotes the collections of the orders $\preceq_\mu$ for all the branchings $\mu$ of $\Leaf(\FF)$.

\begin{theorem}[Kaplan \cite{KaplanII}]\thlabel{th-Kaplan}	
	Let $(\LL,\preceq_\cdot)$ be a planar structure, there is a co-oriented foliation~$\FF$ of $\RR^2$ and an isomorphism $f\colon\LL\to\Leaf(\FF)$ of planar structures. 
	
	Additionally $(\FF,f)$ is unique up to conjugation. To be more precise, for any co-oriented foliation $\FF'$ on $\RR^2$ and for any isomorphism $f'\colon\LL\to\Leaf(\FF')$ of planar structures, there is an orientation preserving homeomorphism $h\colon\RR^2\to\RR^2$ conjugating ~$\FF$ to $\FF'$ and inducing an isomorphism $\wt h\colon(\Leaf(\FF),\preceq_\cdot)\to(\Leaf(\FF'),\preceq_\cdot)$ of planar structures so that $\wt h\circ f=f'$.
\end{theorem}

Next lemma is very natural and left to the reader:

\begin{lemma}\thlabel{lem-Kaplan-realisation}
	Let $\LL_1\xhookrightarrow{f}\LL_2$ be an embedding of planar structures, where $\LL_2=\Leaf(\FF)$ is the planar structure of a foliation~$\FF$ on~$\RR^2$. Then the union of leaves of~$\FF$ 
	$$P=\bigcup_{l\in\LL_1}f(l)\subset\RR^2$$
	is homeomorphic to $\RR^2$ and there exists an isomorphism $h\colon\Leaf(\FF_{|P})\to\LL_1$ of planar structure that satisfies $h\circ f=\id_{\LL_1}$. 
\end{lemma}
\subsection{Universal planar structure and universal foliation}

\begin{lemma} 
	A oriented, transversely oriented foliation~$\FF$ of $\RR^2$ is universal if and only if its planar structure $(\LL,\preceq)$ is universal.  
\end{lemma}

\begin{proof}
	Assume that~$\FF$ is universal.  Consider a planar 
structure $(\tilde\LL,\tilde\preceq_\cdot)$ and $\tilde\FF$ the oriented and transversely oriented foliation associated to it by Theorem~\ref{th-Kaplan}. Let $\psi$ be the inclusion of $\tilde\FF$ in a~$\FF$-saturated connected simply connected region of $\RR^2$, given by the universality of~$\FF$. It induces an inclusion of the leaf space $\tilde \LL$ of $\tilde \FF$ into the leaf space $\LL$ of~$\FF$, preserving the orientation and the orientation of the branchings. This shows that $(\LL,\preceq_\cdot)$ is universal. 

Conversely assume that $(\LL,\preceq_\cdot)$ is universal and consider an oriented and transversally oriented foliation $\tilde \FF$ and its planar structure $(\tilde \LL,\tilde\preceq_\cdot) $ that we may assume to be included in $(\LL,\preceq_\cdot)$ by the universality of $(\LL,\preceq_\cdot)$.  Consider the set of leaves of~$\FF$ which are points of $\tilde \LL\subset\LL$.  It is a connected simply connected region $U$ of $\RR^2$, foliated by~$\FF$ and the corresponding planar structure is $(\tilde \LL,\tilde\preceq_\cdot) $. Now Theorem~\ref{th-Kaplan} provides an orientations preserving conjugacy of $\tilde F$ with the restriction of~$\FF$  to $U$, proving the universality of~$\FF$. 
\end{proof}

\subsection{The planar structure of $\FF_0$ is universal}
Recall that $C\subset[0,1]$ denotes the standard Cantor set. Notice for later that $[0,1]\setminus C$ is a countable union of open intervals, so that the set of complementary interval, for the order induced by the usual order on~$\RR$, is in increasing bijection with the ordered set $(\QQ,\leq)$. 

We denote by $\LL_0=\Leaf(\FF_0)$ the planar structure of the foliation $\FF_0$ defined in Section~\ref{s.universal-foliation}.

\begin{lemma}\thlabel{lem-Lo-univ}
	The planar structure $\LL_0$ is maximally branched (see Definition~\ref{d.maximally}).
\end{lemma}

\begin{proof}
	Every branching of $\FF_0$ projects down in $\RR^2$ to the complement of $\bigcup_{|n|\geq q} (C+2n)\times\{\tfrac{p}{q}\}$ inside $\RR\times \{\frac pq\}$ for some $p\in\ZZ,q\in\NN, p\wedge q=1$. So the induced branching of $\LL_0$ corresponds to the connected component of that complement. It follows that it naturally ordered (for the induced order on $\RR\times\{\tfrac{p}{q}$) as $\QQ$.
\end{proof}

Theorem~\ref{th-univ-leaf} and  Lemma~\ref{lem-Lo-univ} imply that $\LL_0$ is universal. We can now reprove Kaplan's realization theorem.


\begin{proof}[Proof of the realization part of \thref{th-Kaplan}] 
	Consider any planar structure $(\LL,\preceq_\cdot) $. As $\LL_0$ is universal, we can assume that $\LL$ it as a subset of $\LL_0$. Consider the union of leaves of $\FF_0$ corresponding the points in $\LL$ and one gets a foliation of a simply connected region of $\RR^2$  whose planar structure is $\LL,\preceq_\cdot$.
\end{proof}

%% file: RegularFoliation.tex
\section{From $\mathfrak{P}$-pre-lamination to foliation} \label{sec-lam-fo}
The aim of this section is to prove the sufficient part of \thref{main-A}, the necessary part has been proven in Section~\ref{sec-end-pre-lam}. Given a $\Pre$-pre-lamination $L$ with no star, its leaf space $\Leaf^*(L)$ is a (non-singular) planar structure. This planar structure corresponds to a foliation~$\FF$, and we will see that $L$ is conjugated to the end pre-lamination $L_\infty(\FF)$ of $\FF$.


\subsection{For $\Pre$-pre-lanimations}

Let $S_1,S_2$ be two circles, $L_1$ a pre-lamination of $S_1$, and $f\colon S_1\to S_2$ be a homeomorphism. Then
$$L_2=(f\times f)(L_1)\subset S_2\times S_2\setminus\Delta$$
is a pre-lamination of $S_2$. We denote by $L_1\xrightarrow{\wb f}L_2$ and $\Leaf^*(L_1)\xrightarrow{\ub f}\Leaf^*(L_2)$ the \emph{induced maps}.
\thref{th-good-lam-are-end-fol} below is the main step for \thref{main-A} :

\begin{theorem}\thlabel{th-good-lam-are-end-fol}
	Let $L$ be a $\Pre$-pre-lamination of $S^1$.
	\begin{itemize}
		\item Existence:
		      There exists a foliation~$\FF$ of $\RR^2$ and an increasing homeomorphism $g\colon\partial_\infty\FF\to S^1$ conjugating the pre-lamination $L_\infty(\FF)$ to $L$.
		\item Uniqueness: For any foliation $\FF'$ of $\RR^2$ and any increasing homeomorphism $g'\colon\partial_\infty\FF'\to S^1$ conjugating $L_\infty(\FF')$ to $L$, there is an orientation preserving homeomorphism $h$ of $\RR^2$ conjugating~$\FF$ to $\FF'$ so that the induced homeomorphism $\wt h\colon \partial_\infty\FF\to \partial_\infty\FF'$ satisfies $g=g'\circ\wt h$.
	\end{itemize}
\end{theorem}

The theorem allows one to go from pre-laminations on the circle to foliations on the plane. The proof goes from pre-laminations to planar structre, then to foliation.

Take a (non-singular) foliation $\FF$ on $\Int(D^2)$. Given a leaf $f$ of $\FF$, it corresponds a unique leaf of $L_\infty(\FF)$, which is an element of $\Leaf^*(L)$ since $L_\infty(\FF)$ has no star. It describes a natural map from $\Leaf(\FF)$ to $\Leaf^*(L_\infty(\FF))$, which we denote by~$\pi_\FF$.

\begin{proposition}\thlabel{prop-end-leaf-hom}
	For any (non singular) foliation~$\FF$ on $\Int(D^2)$, the map
	$$\Leaf(\FF)\xrightarrow{\pi_\FF}\Leaf^*(L_\infty(\FF))$$
	is an isomorphism of planar structure.
\end{proposition}

\begin{proof}
	Up to conjugating $\FF$ with an orientation preserving homeomorphism, we may assume that $\partial D^2$ is the boundary at infinity of $\FF$.

	Let us first prove that $\pi_\FF$ is a homeomorphism. It is a bijection by construction.
	We prove that $\pi_\FF^{-1}$ is continuous.
	Given a close curve $\sigma\colon[0,1]\to\RR^2$ transverse to $\FF$, denote by $A_\sigma\subset\Leaf(\FF)$ the set of leaves that intersects $\sigma(]0,1[)$.
	The sets $A_\sigma$ form a basis of the topology of $\Leaf(\FF)$. Denote by $f_0$ and $f_1$ the leaves of $L_\infty(\FF)$ that contains the image by $\pi_\FF$ of $\sigma(0)$ and $\sigma(1)$. The image of $\pi_\FF(A_\sigma)$ is the set of leaves of $L_\infty(\FF)$ that separates $f_0$ and $f_1$, so it is an open $L_\infty(\FF)$-interval. It follows that $\pi_\FF(A_\sigma)$ is open.

	Conversely, let $\II$ be an open $L_\infty(\FF)$-interval and $f,g,h$ be leaves of $\FF$ in $\pi_\FF^{-1}(\II)$. By definition of open $L$-interval, we may choose $g,h$ so that they are separated by $f$. Take an open curve $\sigma$ transverse to $\FF$ and that intersects $f$. Take a leaf $f'$ very close to $f$. Then it separates $g$ and $h$, so $\pi_\FF(f')$ separates $\pi_\FF(g)$ and $\pi_\FF(h)$. It follows from the definition of open $L$-intervals that $\pi_\FF(f')$ belongs to $\II$. Hence, $\pi_\FF^{-1}(\II)$ is a neighborhood of $f$. It follows that $\pi_\FF^{-1}(\II)$ is open, and $\pi_\FF$ is continuous. Hence, it is a homeomorphism.

	Given a branching $\mu$ of $\Leaf(\FF)$, we prove that $\pi_\FF$ is increasing on $\mu$.
	Take a leaf $f$ in the connected component of $\Int(D^2)\setminus\mu$ that is adjacent to all the leaves of $\mu$, and $D$ the region in $\Int(D^2)$ that is bounded by $\mu\cup\{f\}$. Recall that the order on $\mu$ is given by the cyclic order on the boundary component of $D$.
	Denote by $\Delta$ the shell of $L_\infty(\FF)$ that corresponds to $\mu$. Also take $I\subset S^1$ the closed interval whose ends are the endpoints of the root of $\Delta$, and so that $I$ contains all the endpoints of $\Delta$. Given two leaves $g,h$ in $\mu$, we have $g<h$ if and only if $(f,g,h)$ is positively cyclically ordered, if and only if the end points or $g$ lies before the ones of $h$ inside $I$, if and only if $\pi_\FF(g)<\pi_\FF(h)$ holds true. Thus, $\pi_\FF$ preserves the order on the branchings.
\end{proof}

\begin{proof}[Proof of \thref{th-good-lam-are-end-fol}]
	Let us prove that existence first.
	Let $L$ be a $\Pre$-pre-lamination of $S^1$. Then $\Leaf^*(L)$ is a (non-singular) planar structure so it is isomorphic to the planar structure of a foliation~$\FF$ on $\RR^2$ (unique up to orientation preserving homeomorphism).
	Let $\ub u\colon\Leaf(\FF)\to\Leaf^*(L)$ denote the isomorphism of planar structure.

	According to Proposition~\ref{prop-end-leaf-hom}, the planar structure $\Leaf^*(L_\infty(\FF))$ is isomorphic to $\Leaf(\FF)$, and hence to $\Leaf^*(L)$ too. According to Theorem~\ref{t.prelam-planar}, this implies that there is an orientation preserving homeomorphisms $g\colon S^1\to S^1$ mapping $L_\infty(\FF)$ onto $L$, ending the proof.

	We now prove the uniqueness.
	Let $\FF'$ be a foliation on $\RR^2$ and $g'\colon\partial_\infty\FF'\to S^1$ be an increasing continuous map that induces a bijection $L_\infty(\FF')\xrightarrow{\wb g'}L$. Then the induced maps $\Leaf^*(L_\infty(\FF))\xrightarrow{\ub g}\Leaf^*(L)$ and $\Leaf^*(L_\infty(\FF')\xrightarrow{\ub g'}\Leaf^*(L)$ are isomorphism of planar structures. Thus, $\ub {g'}^{-1}\circ\ub g$ is an isomorphism of planar structure, and Kaplan (see \thref{th-Kaplan}) proved that this isomorphism is induced by an orientation preserving homeomorphism of foliations $h\colon(\RR^2,\FF)\to(\RR^2,\FF')$. Then the uniqueness in Theorem~\ref{t.prelam-planar} implies that $h$ induces a map $\wt h\colon \partial_\infty\FF\to \partial_\infty\FF'$ satisfying $g=g'\circ\wt h$.
\end{proof}

\subsection{Completing pre-laminations} 

We answer the following question: given a pre-lamination $L_1$ on the circle, is there a pre-lamination $L_2\supset L_1$ that satisfies the condition of \thref{th-good-lam-are-end-fol}, meaning it is the end pre-lamination of a foliation~$\FF$, and so that $L_1$ corresponds to a dense subset of leaves of $\Leaf(\FF)$. Recall that the two long boundary components of a complementary region with diameter larger than some $\epsilon>0$ is well-defined except maybe for finitely many complementary regions.

\begin{theorem}\label{t.non-complete}
	Let $L_1$ be a pre-lamination of $S^1$. Then there exists a $\Pre$-pre-lamination $L_2\supset L_1$, so that the image of $L_1$ into $\Leaf^*(L_2)$ is dense if and only if $L_1$ satisfies the following properties:
	\begin{itemize}
		\item $L_1$ is a dense pre-lamination,
		\item any complementary region of $L_1$ has at least one boundary component which is not a leaf of $L_1$,
		\item $L_1$ has no bad accumulation, 
		\item $L_1$ has few common ends.
	\end{itemize}
\end{theorem}

Next example explains why we use here of the "few common ends" condition instead of the "at most countably many leaves shares the same endpoint".

\begin{example}
	Take $L=\{(1,\theta=e^{2q\pi})\in S^1\times S^1, q\in\QQ\cap ]0,2\pi[\}$. Then $L$ is a dense pre-lamination, the boundary of each complementary component consists in two accumulated geodesics, $L$ has no bad accumulations, and the set of leaves having a common end is countable as $L$ is countable. However, $L$ has not few ends, and $L$ is not contained in a $\Pre$-pre-lamination.
\end{example}



\begin{proof}[Proof of the direct implication.]
	We build $L_2$ by adding leaves to $L_1$. First add to $L_2$ all the geodesics that are complementary regions with empty interior of $L_1$. Then for any complementary region $\Delta$ of $L_1$ with non-empty interior, add to $L_2$ all but one the boundary component of $\Delta$. We can choose the excluded boundary component to be one of the two long boundary components, expect for at finitely many possible $\Delta$. Then $L_2$ is a dense pre-lamination. It has no accumulation since the complementary regions of $L_2$ are exactly the complementary regions of $L_1$ with non-empty interior, with maybe some boundary components removed. Additionally, $L_2$ has no star component. Note that by construction, all the leaves in $L_2\setminus L_1$ are accumulated by leaves in $L_1$.

	We prove that $L_2$ has few common ends, which implies that it is a $\Pre$-pre-lamination with no star. Take $\theta\in S^1$ and recall that $E_{L_i}(\theta)$ is the set of leaves of $L_i$ that ends on $\theta$.
	Take $l^-,l^+$ two distinct leaves of $L_2$ that have a common endpoint $\theta\in S^1$.

	We prove by contradiction that there are at most finitely many leaves of $L_2$ that separates the $l^-$ and $l^+$. So let us assume that it is not the case.
	Any leaf $l\in E_{L_2}(\theta)$ that separates $l^-$ and $l^+$ is either a leaf of $L_1$, or accumulated by leaves of $L_1$. In the second case, it is accumulated by leave of $L_1$ that separates $l^-$ and $l^+$, so they also are in $E_{L_1}(\theta)$. It follows that $E_{L_1}(\theta)$ is not finite.

	Note that $E_{L_1}(\theta)$ is ordered as an interval of $\ZZ$, and it is not finite by assumption. Let us then denote by $l_i$ these leaves, with $i$ in $\NN$ or $\ZZ$, so that $i\mapsto l_i\in E_{L_1}(\theta)$ is increasing. The sequence $l_i$ is monotonous and bounded by either $l^-$ or $l^+$, so it converges when $i$ goes to $+\infty$, and its length is bounded below away from zero. Note that $l_i$ and $l_{i+1}$ bound a common complement region $\Delta_i$. When $i$ goes to $\infty$, the area of $\Delta_i$ goes to zero, so the length of all its boundary components outside $l_i$ and $l_{i+1}$ go to zero. Hence, for $i$ large enough, $l_i$ and $l_{i+1}$ are the two long edges of $\Delta_i$. It follows that $\Delta_i$ is a bad accumulation of $L_1$, which contradicts the hypothesis. Therefore, there are at most finitely many leaves between $l^-$ and $l^+$, so $L_2$ has few common ends.
\end{proof}

\begin{remark}
	In the previous proof, we only add to $L_2$ some accumulated boundary of complementary regions. This is not the only possibility. For instance if a complementary region $P$ has four boundary components, two of which are not leaves of $L_1$, we can add one diagonal of $P$ that separates the two accumulated boundary components of $P$. When $P$ has more boundary components, wilder completion are possible.
\end{remark}

This remark makes the converse trickier to prove that one may expect.

\begin{proof}[Proof of the converse]
	Assume that such a pre-lamination $L_2$ exists. Then $L_1$ has few common ends since it is a subset of $L_2$.

	Assume (arguing by contradiction) that there is a complementary region $\Delta$ of $L_1$, with non-empty interior and whose boundary geodesic are all leaves of $L_1$. Since $L_2$ satisfies $\Pre$ with no stars, $\Delta$ is not a complementary region of $L_2$. So there exists a leaf $l_1$ of $L_2$ which lies in the interior of $P$. Let $\II$ be an open $L_2$-interval containing $l_1$. By definition of open $L_2$-intervals, $\II$ contains another leaf $l_2$ of $L_2$ contained in $\Delta$. Then the $L_2$-interval $[l_1,l_2]$ inside $\Leaf^*(L_2)$ has non-empty interior and is disjoint from the image of $L_1$. It contradicts the density assumption on the leaves of $L_1$. Thus, all complementary region of $L_1$ with non-empty interior have at least one accumulated boundary.

	We now prove that $L_1$ has no bad accumulation. Assume (arguing by contradiction) that there exist $\epsilon>0$ and a sequence of $\epsilon$-large complementary regions $\Delta_n$ of $L_1$ whose two long boundary components are leaves of $L_1$, that accumulates on a geodesic $\gamma$. Either $\gamma$ is a leaf of $L_2$, or it is the root of a complementary shell in $L_2$. In the two cases, there exists an open $L_2$-interval $\II$ contains all the leaves close enough to $\gamma$. Thus $\II$ contains the two large boundary components of $\Delta_n$ for all large enough $n$. Similarly as above, there are two leaves of $L_2$ inside a common region $\Delta_n$ ($n$ large) that intersect $\II$. And similarly to above, the interval of leaves of $L_2$ between them does not contain any leaf of $L_1$, which contradicts the density hypothesis.
\end{proof}

%% file: SingularFoliation.tex
\section{Singular planar structures and $pA$-foliations} \label{sec-lam-fo}

In this section we consider $pA$-foliations. Next remark motivates our hypothesis of non-existence of connection between singular points.

\subsection{The singular planar structure of a $pA$-foliation}

Let~$\FF$ be a $pA$-foliation of $\RR^2$, and $\Sigma$ its singular set. We denote by $\Leaf(\FF)$ the set of leaves of $\RR^2\setminus\Sigma$, equipped with the quotient topology. Given an open curve $\sigma\subset\RR^2\setminus\Sigma$ transverse to $\FF$, every leaf of $\FF$ cuts $\sigma$ in at most $1$ point (see Lemma~\ref{l.1point}). It yields an embedding from $\sigma$ to $\Leaf(\FF)$. As the images of such curves cover the whole space $\Leaf(\FF)$, this proves:

\begin{lemma}
	$\Leaf(\FF)$ is a connected non-Hausdorff $1$-dimensional manifold.
\end{lemma}

We regroup simple facts about $pA$-foliations:

\begin{enumerate}
	\item If two leaves $\ell_1,\ell_2$ of~$\FF$ on $\RR^2\setminus\Sigma$ are non-separated, then there are segments
	      $\sigma_i\colon [-1,1]\to \RR^2\setminus\Sigma$, $i=1,2$ transverse to~$\FF$, so that $\sigma_i(0)\in \ell_i$, and so that $\sigma_1(t)$ and $\sigma_2(t)$ lie on the same leaf for $t\in [-1,0[$, and no leaf cuts both $\sigma_1(]0,1])$ and $\sigma_2(]0,1])$. This implies that $\Leaf(\FF)$ is a tame $1$-dimensional manifold (see Definition~\ref{d.tame}).
	\item Every non-singular leaf $l$ of~$\FF$ cuts $\RR^2\setminus\Sigma$ in two connected components. Thus $l$ cuts $\Leaf(\FF)$ in two connected components.
	\item Let $\ell_i$, $i=1,\dots,k$ be the separatrixes of a singular point. Then $\ell_i,\ell_{i+1}$ are on a same branching $\mu_i$, and $\mu_{i-1}\neq\mu_i$. Furthermore, $\ell_{i-1}\cup \ell_{i+1}$ disconnects $\ell_i$ from the other separatrixes. One deduces that $\ell_i$, $i=1,\dots,k$	form a cyclic branching.
	\item As a leaf in a cyclic branching does not disconnect $\Leaf(\FF)$, a leaf that is not a separatrix cannot belong to a cyclic branching.
\end{enumerate}

Every branching $\mu$ of $\Leaf(\FF)$ is ordered as follows: let $R$ be the (unique) connected components of $\RR^2\setminus(\mu\cup\Sigma)$ that contains a sequence of leaves that converges on all the leaves of $\mu$. Denote by $l$ a (non-separatrix) leaf in $R$ and $R'$ the components of $R\setminus l$ that is adjacent to $\mu$. For two leaves $f_1,f_2\in\mu$, we set $f_1\preceq f_2$ if $(l,f_1,f_2)$ is clockwise cyclically ordered as boundary components of $R'$ (or if $f_1=f_2)$. It yields a total order $\preceq_\mu$ on $\mu$, independent of the choice of the leaf $l$, and called the \emph{left-right order}.

\begin{lemma}\thlabel{lem-cyrc-branch-fol} \label{l.cyclic-singular}
	Let~$\FF$ be a $pA$-foliation of $\RR^2$. The cyclic branchings of $\Leaf(\FF)$ are the sets of separatrixes of the singularities, ordered clockwise or anti-clockwise.
\end{lemma}

In the proof we will use the following observation.

\begin{lemma} \label{l.loop}
	For any cyclic branching $\nu=\{x_i\}_{i\in \ZZ/n\ZZ}$ and any compact subset $K\subset\Leaf(\FF)$ disjoint from $\nu$, there exists a continuous loop $\gamma\subset\Leaf(\FF)\setminus K$ that intersects each point in $\nu$ exactly once, and transversely.
\end{lemma}

\begin{proof}
	To build $\gamma$, simply concatenate small enough segments around the points in $\nu$.
\end{proof}

\begin{proof}[Proof of Lemma~\ref{lem-cyrc-branch-fol}]
	We take $\nu=(x_i)_{i\in\ZZ/n\ZZ}$ a cyclic branching of $\LL$. By definition, no leaf $x_i$ disconnects $\RR^2\setminus\Sigma$, so they all are separatrixes of~$\FF$.

	Denote by $p\in\Sigma$ the based point of the separatrix $x_i$, and let $y$ be a separatrix of~$\FF$ based on $p$. If $y$ doesn't belong to $\nu$, Lemma~\ref{l.loop} provides a continuous loop $\gamma$ that intersects every point in $\nu$ exactly once and transversely, and that avoids $y$.
	But then $\gamma\setminus\{x_i\}$ connects the two distinct connected components of $\RR^2\setminus(x_i\cup y\cup\{p\})$. This contradiction implies that all separatrixes based at $p$ are contains in $\nu$.

	Since $x_i$ is separated from all $x_j$, $j\neq i$, excepts for $x_{i-1}$ and $x_{i+1}$, these two separatrixes are adjacent to $x_i$. And by induction on $i$, $\nu$ is exactly the set of separatrixes of $p$ in clockwise or anti-clockwise order.
\end{proof}

As a $pA$-foliation does not admit saddle connections, it follows that:

\begin{corollary}
	The cyclic branching of $\Leaf(\FF)$ are pairwise disjoint.
\end{corollary}

Observe that since two adjacent separatrixes, at a common singularity $p$, are not separable from one side, they are ordered for the left/right order on the branching that contains them.
For any cyclic branching $\nu$ of~$\FF$, the anti-clockwise cyclic ordering of its separatrixes is compatible with the left/right order on branchings of~$\FF$.

\begin{lemma}
	Let~$\FF$ be a $pA$-foliation of $\RR^2$. The set $\Leaf(\FF)$ equipped with the left-right order on branching and the anti-clockwise cyclical order on cyclic branching is a singular planar structure.
\end{lemma}

\begin{proof}
	It follows from \thref{lem-cyrc-branch-fol} that the cyclic branching in $\Leaf(\FF)$ are disjoint and of order at least 3. Any non-separatrix leaf disconnects $\RR^2\setminus\Sigma$ in two connected components. So it also disconnects $\Leaf(\FF)$ in two connected components. Similarly, every cyclic branching $\nu$ separates $\RR^2\setminus\Sigma$ into $|\nu|$ connected components, and therefore $\Leaf(\FF)$, in $|\nu|$ connected components.
\end{proof}

\begin{lemma}\label{l.droite-plongee}
	Let $\LL$ be a $pA$-foliation, and $f\colon\RR\to\LL$ be a continuous map that is locally injective. Then $f$ is injective.
\end{lemma}

\begin{proof} A classical argument implies that a locally injective map $f$ can be realized as a path $\wt f\colon \RR\to \RR^2\setminus \Sigma$ topologically transverse to~$\FF$. Now the lemma follows from Lemma~\ref{l.1point}.
\end{proof}

\subsection{Generalization of Kaplan's theorem to the singular case} 

The aim of the next sections is to prove the following generalization of Kaplan's Theorem to $pA$-foliations.

\begin{theorem}\thlabel{th-Kaplan-sing}
	Let $(\LL,\preceq_\cdot)$ be a singular planar structure. Then there exists a $pA$-foliation~$\FF$ of $\RR^2$, and an isomorphism $f\colon(\LL,\preceq_\cdot)\to(\Leaf(\FF),\preceq_\cdot)$ of singular planar structures. This foliation is unique in the following sens:

	Let $\FF'$ be a prong foliation and $f'\colon(\LL,\preceq_\cdot)\to\left(\Leaf(\FF'),\preceq_\cdot\right)$ be an isomorphism a singular planar structure. There exist an orientation preserving homeomorphism $h$ of $\RR^2$ sending~$\FF$ onto $\FF'$, whose induced isomorphism $\wt h$ from $(\Leaf(\FF),\preceq_\cdot)$ to $(\Leaf(\FF'),\preceq_\cdot)$ satisfies $\wt h\circ f=f'$.
\end{theorem}

The strategy for proving Theorem~\ref{th-Kaplan-sing} is the same as for the non-singular case. We will build a $pA$-foliation $\FF_{univ}$ whose singular planar structure is universal, in the sense that it contains any singular planar structures (see Definitions~\ref{d.pre-universal-planar} and~\ref{d.universal-planar}). Then any singular planar structure is the singular planar structure of the $pA$-foliation obtained as the set of leaves of $\FF_{univ}$ corresponding to the sub-planar structure.

This construction is much easier if in restriction to orientable foliations: in that case the universal oriented $pA$-foliation is a ramified cover over a non-singular universal foliation. The non-orientable case is obtained from the orientable case by surgeries, removing some quadrants of singularities.
Such surgeries need some preliminary works, which will be done in the next sections.


\subsection{Origin/opposite separatrixes} \label{s.origin-quadrant}

Let~$\FF$ be a prong foliation on $\RR^2$, and $\bp$ a based point in $\RR^2$ that does not lie on a separatrix. Let $p$ be a singularity of~$\FF$ and $S_p$ be the union of the separatrixes at $p$ (that is, the singular leaf through $p$). We call \emph{quadrant} based at $p$ the connected components of $\RR^2\setminus S_p$, and \emph{origin quadrant} the one that contains $\bp$. We call \emph{origin separatrixes} of $p$ the two separatrixes that bound the origin quadrant at $p$, and \emph{non-origin separatrix} every separatrix that is not an origin separatrix.

\begin{lemma}\label{l.opposite}
	Any path $\gamma\colon\RR\to\RR^2\setminus\Sigma$ transverse to~$\FF$ intersects at most one non-origin separatrix.
\end{lemma}

\begin{proof}
	Denote by $\Sigma_\gamma$ the set of singularities that admits a separatrix that intersects $\gamma$. For every $x\in\Sigma_\gamma$, let $f_{x,0}$ be the separatrix based on $x$ and crossed by $\gamma$. Also denote by $f_{x,-1},f_{x,1}$ the two separatrixes of $x$ adjacent to $f_{x,0}$. Then $f_{x,-1}\cup\{x\}\cup f_{x,1}$ bound two half planes, denote by $\Delta_x$ that one that is disjoint from $f_{x,0}$. Since $\gamma$ intersects at most one separatrix of a given singularity, $\gamma$ is disjoint from the separatrixes $f_{x,-1}$ and $f_{x,1}$. So $\gamma$ remains in $\Delta_x$.

	It implies that the sets $\Delta_x$ are pairwise disjoint. The separatrix $f_{x,0}$ is non-origin when $\bp$ belongs to $\Delta_x$. As the $\Delta_x$ are pairwise disjoint, $\bp$ belongs to at most one of the $\Delta_x$. Thus, at most one separatrix $f_{x,0}$ is non-origin.
\end{proof}

Lemma~\ref{l.opposite} prevent any accumulation point in $\RR^2$ of non-origin separatrixes. That is the union of the non-origin quadrants is closed. In particular one deduces:

\begin{corollary}\thlabel{cor-non-or-sep}
	There exists a family of pairwise disjoint open subsets $U_s\subset\RR^2$, where $s$ varies over the set of non-origin separatrixes, so that $U_s$ is a neighborhood of $s$ (minus the singular point).
\end{corollary}

Corollary~\ref{cor-non-or-sep} will allow us to cut and past simultaneously along all non-origin separatrixes and ensure that the results is still a surface (and then a plan) by using the neighborhoods given above for the set of charts after cut-and-pasting.

\subsection{Construction by ramified cover} 

Let us recall two notations from Section \ref{s.universal-foliation}.
Denote by $C\subset[0,1]$ the standard cantor set (with $\{0,1\}\subset C$), and $K\subset\RR^2$ the closed set
$$K=\bigcup_{\substack{p,n\in\ZZ\\ q\in\NN_{>0}\\ |n|\geq q}}\left(C+2n,\frac{p}{q}\right).$$

We build a $pA$-foliation by taking a ramified covering of $\RR^2\setminus K$ endowed with the horizontal foliation. We chose the ramification points to be on a dense subset of non-separated leaves. 

The set $\pi_0([0,1]\setminus C)$ of connected components of $[0,1]\setminus C$, inherit from $\RR$ an order, and this order is in increasing bijection with $\QQ$.
In other words, one choose an increasing indexation
$$\pi_0\left([0,1]\setminus C\right) = \{c_{s}, s\in\QQ\}.$$

We choose a (noncontinuous) function $\psi\colon \QQ\to \NN\setminus \{1,2\}$ so that, for any $i$ the set $\QQ_i=\psi^{-1}(i)$ is dense in $\QQ$: for instance, $\psi(p/q)= 0$ if $q$ is even or $q=1$ and $\psi(p/q)= 1+i\geq 3$ if $q$ is odd and the lower prime number of $q$ in the decomposition of $q$ into prime numbers is the $i^{th}$ prime number ($3$ is the second prime number).

Fix a rationnal $r\in\QQ$, which we write $r=\tfrac{p}{q}$ with $p\wedge q=1$. Denote by $K_r= K\cap(\RR\times\{r\})$ the slice of $K$ at level $y=r$. We denote by $d_r$ the connected component $d_r=]-2q+1,2q[\times\{r\}$ of $(\RR\times\{r\})\setminus K_r$. Define a map
$$\varphi_r\colon \pi_0((\RR\times\{r\})\setminus K_r)\to\NN\setminus \{1,2\}$$
by:
\begin{itemize}
	\item if $c=d_r$ holds, we set $\varphi_r(c)=\psi(r)$.
	\item if $c=(2n+c_{s})\times\{r\}$ holds for some $n\in\ZZ$ and $s\in\QQ$, then we set $\varphi_r(c)=\psi(s)$. Note that we have $n\geq q$ in that case.
	\item if $c=]2k-1,2k[\times\{r\}$ holds for some $k\leq q$ or $k\geq q+1$, we set $\psi_r(c)=0$.
\end{itemize}

We let the reader check that the possible intervals $c$ given above are indeed connected components of $(\RR\times\{r\})\setminus K_r$, and that the three cases cover all connected components.

Given any component $c\in \pi_0(\RR\times\{r\}\setminus K_r)$ for which $\varphi_r(c)\neq 0$, we choose a point $x_c\in c$ as follows:

\begin{itemize}
	\item if $c=d_r$ holds, we set $x_{c}=(-2q+2,r)$ if $p$ is even and $x_{c}=(2q-1,r)$ if $p$ is odd.
	\item otherwise $x_c$ is the middle point of $c$.
\end{itemize}

\begin{remark}\label{r.density-degree-lr}
	As the central component $d_r$ is arbitrarily large for a dense subset of value of $r$, one deduces the following property.
	Given any open curve $\sigma\colon\RR\to\RR^2\setminus\Sigma$ transverse to the horizontal foliation, and given any $i\geq 3$, there is a dense set of values of $t\in\RR$ so that if we write
	$\sigma(t)=(x,r)$, we have the following properties:
	\begin{itemize}
		\item $r\in\QQ$ and $\psi(r)=i$,
		\item $\sigma(t)\in d_r$,
		\item $x_{d_r}$ is at the right of $\sigma(t)$.
	\end{itemize}
	The same statement holds for $x_{d_r}$ on the left of $\sigma(t)$.
\end{remark}

We denote $$X_r=\{x_c, c\in \pi_0(\RR\times\{r\}\setminus K_r) \mbox{ and }\varphi_r(c)\neq 0\} \quad \mbox{ and }\quad X=\bigcup_{r\in \QQ} X_r.$$
We still denote by $\varphi_r$ the restriction of $\varphi_r$ to $X_r$ and by $\varphi\colon X\to \NN\setminus\{1,2\}$ the map that coincides with $\varphi_r$ on $X_r$ for any $r\in\QQ$.
Notice that any accumulation point of $X$ belongs to $K$, so $K\cup X$ is a closed subset of $\RR^2$, and $X$ is a closed discrete set in the open surface $\RR^2\setminus K$.
We denote by $\theta\colon X\to \NN\setminus\{0,1\}$ the map defined by $\theta(x)=\tfrac{1}{2}\varphi(x)$ if $\varphi(x)$ is even and by $\theta(x)=\tfrac{1}{2}(\varphi(x)+1)$ if $\varphi(x)$ is odd.

We consider the universal cover $\Pi_{\theta}\colon\widetilde{\RR^2\setminus K}_{\varphi}\to \RR^2\setminus K$ of $\RR^2\setminus K$ ramified along $X$ with ramification degree $\theta$ (see Definition~\ref{d.ramified}). To simplify the notation, write $\PP_e=\widetilde{\RR^2\setminus K}_{\theta}$ ($e$ for even). Then $\PP_e$ is diffeomorphic to the plane $\RR^2$. We endow $\PP_e$ with the lift $\FF_\theta$ of the horizontal foliation on $\RR^2\setminus K$. Its singular is $\Sing(\FF_\theta)=\Pi_{\theta}^{-1}(X)$.

\begin{remark} The foliation $\FF_\theta$ has only singularities with an even number of separatrixes (prongs) and therefore is an orientable and co-orientable foliation.
	In particular, its leaf space $\Leaf(\FF_\theta)$ is an orientable $1$-dimensional non-Hausdorff manifold.
	Proposition~\ref{t.orientable-universal} below can be understood as: $\Leaf(\FF_\theta)$ is universal for orientable singular planar structures.
\end{remark}

Because of the even number of prongs, the planar structure $\Leaf(\FF_\theta)$ is not universal. We will remove by surgery on some separatrixes, in order to get a universal singular planar structure.
Recall that a singular planar structure is pre-universal (see Definition \ref{d.pre-universal-planar}) if roughly speaking all types of branchings appear as dense subsets. We give below the technical tool that allows use to prove that the singular planar structure of the foliation, after surgery, is pre-universal.

Let $A$, $B_i$, $i\geq 3$, and $B_0$ denote the following subsets of $\Leaf(\FF_\theta)$: $A$ is the union of the branchings, $B_i$ is the union of cyclic branchings of degree $i\geq 3$, and $B_0=A\setminus\bigcup_{i\geq 3}B_i$.

\begin{proposition} \label{t.orientable-universal}
	The planar structure $\Leaf(\FF_\theta)$ satisfies the following properties:
	\begin{enumerate}
		\item every element in $A$ is on a branching on its two sides,
		\item if $i$ is odd, then $B_i=\emptyset$,
		\item if $i\geq 2$ then $B_{2i}$ splits in $B_{2i}=B_{2i}^+\coprod B_{2i}^-$ where $\varphi\equiv2i$ on $B_{2i}^+$ and $\varphi\equiv2i-1$ on $B_{2i}^-$,
		\item for any two points $a_1\prec_\mu a_2$ on a branching $\mu$ and that are not on a common cyclic branching, the set $\{x\in\mu,a_1<_\mu x<_\mu a_2\}$ intersects $B_0$ and $B_{2i}^\pm$ for all $i\geq 2$.
		\item for every branching $\mu$ and every $i$, $\mu\cap B_{2i}^\pm$ has no maximum and no minimum for $\preceq_\mu$,
		\item for every open subset $U$ of $\Leaf(\FF_\theta)$, the set $B_0\cap U$ is dense in $U$,
		\item for every oriented open interval $U\simeq\RR$ in $\Leaf(\FF_\theta)$ and every $k\geq 2$, the two sets of separatrixes in $B_{2k}^\pm\cap U$ that lies on the left/right of $U$ are dense in $U$.
	\end{enumerate}
\end{proposition}

\begin{proof}
	Take two leaves $l_1,l_2\in \Leaf(\FF_\theta)$ in a branching $\mu$. There are transverse curves $\sigma_1,\sigma_2\colon[-1,1]\to\PP_e$ with $\sigma_i(0)\in l_i$ and so that for every $t<0$, $\sigma_1(t)$ and $\sigma_2(t)$ belong to the same leaf.
	Then the projection by $\Pi_{\theta}$ of $\sigma_1(t)$ and $\sigma_2(t)$ are on the same horizontal line in $\RR^2$.
	Thus, the projection of $l_1$, $l_2$ also lie on the same horizontal line $\RR\times \{r\}$ with $r\in \QQ$.

	Conversely, given $r\in\QQ$, and given any compact segment $I\subset\RR$ there is $\varepsilon>0$ so that $(K\cup X)\cap (I\times [-\varepsilon, \varepsilon])= (K\cup X)\cap I\times \{r\}$.
	It implies that $\Pi_{\theta}$ is a homeomorphism in restriction to any connected component of $\Pi^{-1}_{\theta}((I\times [-\varepsilon,0])\setminus (K\cup X))$ and of $\Pi^{-1}_{\theta}((I\times [0,\varepsilon])\setminus K)$. We deduce that :

	\begin{itemize}
		\item every point in a branching belongs to a branching on both sides;
		\item $\Pi_{\theta}$ induces an increasing bijection between any branching $\mu$ with a subset $(\RR\times\{r\})\setminus(K\cup X)$ for some $r\in\QQ$.
	\end{itemize}

	Between any two connected components of $(\RR\times\{r\})\setminus K$, the map $\varphi$ takes all the possible values. Take a component $c$ of $\RR\times\{r\} \setminus K$ for which $\varphi\neq 0$. It contains two components of $\RR\times\{r\} \setminus (K\cup X)$ (which are delimited by $x_c$) which lift to the $2\theta(c)$ separatrixes of a cyclic branching of degree $\theta(c)$.

	This proves items 1 to 5.

	Consider now an oriented open interval $U$ in $\Leaf(\FF_\theta)$. It lifts to an open curve $\gamma\subset\RR^2\setminus(K\cup X)$, transverse to the foliation. Up to making $U$ smaller, we may assume that $\gamma$ remains in a compact subset of $\RR^2$. Then according to Remark~\ref{r.density-degree-lr}, there is a dense set of points $(x,r)\in\gamma$ with $r\in\QQ$, $\psi(r)=2i$, and so that $x_{d_r}$ lies on the left (respectively right) of $(x,t)$. The preimage of that set in $U$ lies in $B_{2i}$, and the corresponding cyclic branching are on the left (respectively right) of $U$.
	This proves items 5 and 6, ending the proof.
\end{proof}

\subsection{Construction by surgeries}

For convenience, we denote by $\theta$ and $\varphi$ the lifted maps $\theta\circ \Pi_{\theta}$ and $\varphi\circ \Pi_{\theta}$ from $\Sing(\FF_\theta)$ to $\NN_{\geq 3}.$ Furthermore, for $x\in \Sing(\FF_\theta)$ one has :
\begin{itemize}
	\item if $\varphi(x)$ is even, then $x$ has $2\theta(x)=\varphi(x)$ separatrixes.
	\item if $\varphi(x)$ is odd, then $x$ has $2\theta(x)=\varphi(x)+1$ separatrixes.
\end{itemize}
In particular every point in $\Sing(\FF_\theta)$ as an even number of separatrixes, and hence an even number of quadrants (see Section~\ref{s.origin-quadrant} for the notion of quadrants).

We will now cut a quadrant of each singular point $x$ in $\Pi_{\theta}^{-1}(x)$ for which $\varphi(x)$ is odd.
Choose a based point $\bp\in \PP_e$ that does not lie on a separatrix, which determines for each singularity its origin quadrant (see Section~\ref{s.origin-quadrant}). Any singular point $x$ has an even number of quadrants, and these quadrants are naturally clockwise cyclically ordered. We call the \emph{opposite quadrant} of $x$ its quadrant that is opposite to its origin quadrant. We also call \emph{odd-opposite quadrant} the opposite quadrant of a singularity $x$ that satisfies that $\phi(x)$ is odd.

Note that two quadrants with non-empty intersection but with disjoint interior are successive quadrants of the same singular point. Thus, two opposite quadrants are either disjoint or their interiors have non-empty intersection. Since quadrants are bounded by two adjacent separatrixes, two quadrants that intersects on their interior are included one in the other. A \emph{maximal odd-opposite quadrant} is an odd-opposite quadrant that is not included in the interior of any another odd-opposite quadrant.

Lemma~\ref{l.opposite} implies that odd-opposite quadrants do not accumulate on the plane. One deduces:

\begin{lemma}
	Any odd-opposite quadrant is included in a maximal odd-opposite quadrant.
\end{lemma}

\begin{proof}
	Reason by contradiction and assume that there exists an infinite increasing sequence $Q_n$ of odd-opposite quadrants. Denote by $Q_n^c$ the closure of the complement of $Q_n$. Then $Q_n^c$ is a decreasing sequence which contains $\bp$, so it is not empty. In particular $\bigcap_{n}Q_n^c$ is a saturated set containing $\bp$. The boundary of $\bigcap_{n}Q_n^c$ is accumulated by $\partial Q_n^c$, which is the union of two non-origin separatrixes and a singular point. Which contradicts Lemma~\ref{l.opposite}, that is these non-origin separatrixes have no accumulation in $\PP_e$.
\end{proof}

As a consequence of \thref{cor-non-or-sep} one gets:

\begin{lemma}\thlabel{lem-quadrant}
	Any compact subset of $\PP$ intersects at most finitely many odd-opposite quadrants.
\end{lemma}

Let $\PP_e'$ be the plane $\PP_e$ minus the interior of all maximal odd-opposite quadrants. Given a singular point $x\in \PP_e'$ with odd $\varphi(x)$ (in fact, $x\in\partial \PP_e'$), denote by $l^-,l^+$ the two separatrixes of $x$ that bound the odd-opposite quadrant (in $\PP_e)$ based at $x$. We glue these two separatrixes using any homeomorphism $f\colon l^-\cup\{x\}\to l^+\{x\}$. Denote by $\PP_{\varphi}$ the topological space obtained after these identifications for all singularity $x\in\PP_e'$ with odd $\varphi(x)$. Also denote by $\pi\colon\PP_e'\to \PP_{\varphi}$ the projection. Denote by $\FF_\varphi$ the image of the foliation $\FF_\theta$ after identification. Next lemma immediately follows from \thref{lem-quadrant}.

\begin{lemma}
	The space $\PP_{\varphi}$ is homeomorphic to a plane, and $\FF_\varphi$ is a prong foliation. The singularities of $\FF_\varphi$ are exactly the image of the singularities of $\FF_\theta$ that do not lie in the interior of any odd-opposite quadrant. Additionally, the degree of a singularity $x$ in $\FF_\varphi$ is equal to $\varphi(x)$.
\end{lemma}

\subsection{Universal $pA$-foliation}

A $pA$-foliation is said \emph{universal} if its singular planar structure is universal.

\begin{theorem}\thlabel{th-univ-sing-fol-ex}
	The $pA$-foliation $\FF_\varphi$ is universal.
\end{theorem}

\begin{proof}
	We will prove that $\FF_\varphi$ is pre-universal as defined in Definition \ref{d.pre-universal-planar}, which implies that it is universal.

	Let us describe the branchings of $\FF_\varphi$. Let $\ell$ be a leaf of $\FF_\varphi$ and $\sigma\colon ]-1,1[\to \PP_{\varphi}$ be a transverse curve, with $\sigma(0)\in \ell$. Consider $\pi^{-1}(\sigma)\subset \PP_e'$. As transverse curves may cross at most one odd-opposite separatrix, up to shrinking $\sigma$, one may assume that $\pi^{-1}(\sigma([-1,0[])$ and $\pi^{-1}(]0,1])$ are disjoint from all the odd-opposite quadrants. Thus, either $\pi^{-1}(\sigma)$ is a transverse segment disjoint from the odd-opposite quadrants, or $\pi^{-1}(\sigma(0))$ consists in two points on the two separatrixes bounding a same maximal odd-opposite quadrant.
	In the two cases, the leaves passing through $\pi^{-1}(\sigma(t))$ for $t\in]-1,0[$, are not in odd-opposite quadrant, so their limit as $t\to 0$ are not in the interior of odd-opposite quadrants. This shows that the leaf $\ell$ belongs to a branching for $\FF_\varphi$ on the negative side of $\sigma$ if and only if the leaf passing through
	$\pi^{-1}(\sigma(0_-))$ belongs to a branching, and $\pi$ induces an increasing bijection between these branching. Thus, item 3 in the definition of pre-universal is satisfied.

	If $\pi^{-1}(\sigma(0))$ is not in the boundary of an odd-opposite quadrant, one gets that $\ell$ is either in no branching or in two branchings, one on each side. If $\pi^{-1}(\sigma(0))$ is two points on the two separatrixes bounding a maximal odd-opposite quadrant, then $\ell$ belongs to a branching on each side, as discuss above. So item 1 is satisfied.

	Denote by $A$ the union of the branchings of $\Leaf(\FF_\varphi)$.
	It follows that any point in $A$ belongs to two branchings, one in each side.
	Furthermore, every branching $\mu$ of $\FF_\varphi$ is in increasing bijection with a branching $\wt \mu$ of $\FF_\theta$ contained in $\PP_e$. Hence, item 3 is satisfied.

	Proposition~\ref{t.orientable-universal} implies that between any two points in $\wt \mu$ which are not successive points, there are points in $\wt \mu$ with all possible values of $\varphi$. This remains true for $\mu$. However, now, points $x\in \Leaf(\FF_\varphi)$ with $\varphi(x)\geq 3$ belong to a cyclic branching of degree $\varphi$. Notice that the cyclic order on the cyclic branching are still compatible with the order on the branchings.

	As already said, any interval $U$ in $\Leaf(\FF_\varphi)$ crosses at most one separatrix obtained by identification of two odd-opposite separatrixes. As we are concerned with a dense subset of $U$, we can assume that $U$ is disjoint from these separatrixes, and hence corresponds to a transverse segment in $\Leaf(\FF,{X,\theta})$. For every $i$, Proposition~\ref{t.orientable-universal} provides points in $U$ with all possible values of $\varphi=i$ (hence points in cyclic branching of degree $i$), at the right as well as at the left of the chart. It proves item 4 and 5.
\end{proof}

\subsection{Proof of \thref{th-Kaplan-sing}} 

We can now prove the singular version of Kaplan's theorem: any singular planar structure is the leaf space of a unique $pA$-foliation.

\begin{proof}[Proof of \thref{th-Kaplan-sing}, existence] Let $\FF_\varphi$ be the universal $pA$-foliation defined in \thref{th-univ-sing-fol-ex}. Let $(\LL,\preceq)$ be a singular planar structure.
	By definition of universality, there exists an embedding of singular planar structure $f\colon\LL\to\Leaf(\FF_\varphi)$. The union of leaves of $\FF_\varphi$ in $f(\LL)$ and of the singularities of $\FF_\varphi$ for which all separatrixes are in $f(\LL)$, is an open subset of $\RR^2$ homeomorphic to a plane equipped with a $pA$-foliation whose singular planar structures is isomorphic to~$\LL$.
\end{proof}

\begin{proof}[Proof of \thref{th-Kaplan-sing}, uniqueness]
	We now prove the uniqueness. Take two foliations $\FF_1,\FF_2$ of $\RR^2$ and an isomorphism $f\colon\Leaf(\FF_1)\to\Leaf(\FF_2)$ of singular planar structures. We will build a homeomorphism $h\colon\RR^2\to\RR^2$ that send $\FF_1$ to $\FF_2$, and whose induced map $\Leaf(\FF_1)\xrightarrow[]{h^*}\Leaf(\FF_2)$ is equal to $f$.

	Let $S_i\subset \RR^2$ be the set of singular points of $\FF_i$. We denote by $P_i$ the universal covering of $\RR^2\setminus\Sigma_i$ and $\wt\FF_i$ the lift of $\FF_i$ to $P_i$, which is a foliation without singularity. The covering map $P_i\to\setminus\Sigma_i$ induces a covering map $\Leaf(\wt\FF_i)\to\Leaf(\FF_i)$. Since $\wt\FF_i$ is a foliation on a plane, $\Leaf(\wt\FF_i)$ is connected and simply connected, so that $\Leaf(\wt\FF_i)\to\Leaf(\FF_i)$ is the universal covering map of $\Leaf(\FF_i)$. It is also clear that this map preserves the ordering at the branchings.

	Chose an origin $\bpx_i\in \RR^2$, not on a singular leaf of $\FF_i$, and $\ell_i$ the leaf that contains $\bpx_i$, so that $f(\ell_1)=\ell_2$. We choose lifts $\wt\ell_1,\wt \ell_2$ of $\ell_1,\ell_2$ on $\Leaf(\wt\FF_1)$ and $\Leaf(\wt\FF_2)$, respectively. Now the map $f$ lifts uniquely on the universal cover as a map
	$\wt f\colon \Leaf(\wt\FF_1)\to\Leaf(\wt\FF_2)$ so that $\wt f(\wt \ell_1)=\wt \ell_2$ and this map is a homeomorphism which is increasing for the order, that is, is an equivalence of the planar structures of $\wt\FF_1$ and $\wt\FF_2$. According to Kaplan theorem, there is a homeomorphism $\wt \Phi\colon P_1\to P_2$ with $\Phi(\wt\FF_1)=\wt\FF_2$ and so that $\Phi$ induces $\wt f$ on the planar structures.


	For every singular point $p$ of $\FF_1$ one chooses a non-origin separatrix $\ell_p$, and thus $f(\ell_p)$ is a non-origin separatrix for $f(p)$. Notice that:
	$$\Delta_1 =\RR^2\setminus \bigcup_{p\in \Sing(\FF_1)} \ell_p$$
	is connected, and simply connected, that is, is diffeomorphic to $\RR^2$. Let $\wt\Delta_1$ be the lift of $\Delta_1$ on $\PP_1$ that contains $\bpx_1$.
	Note that its boundary is the union of two lifts of each of the leaves $\ell_p$.

	Let $\wt\Delta_2=\Phi(\wt\Delta_1)$ be the corresponding set on $\PP_2$.
	It is a connected, simply connected, and its contains exactly one lift of each leaf of $\wt\FF_2$ excepts for the leaves $f(\ell_p)$. Thus, $\wt\Delta_2$ projects down on the set $\Delta_2=\RR^2\setminus \bigcup_p f(\ell_p)$.

	The restriction $\wt\Delta_1\xrightarrow{\Phi}\wt\Delta_2$ induces a homeomorphism $h_0\colon\Delta_1\to\Delta_2$ that conjugates the foliations $\FF_i$. It has no reason to extend continuously on the whole set $\RR^2$. To extend it, we first deform $h_0$ in neighborhoods of the leaves $f(\ell_p)$.

	Denote by $\ell^+_p,\ell^-_p$ the lifts of $\ell_p$ that bound $\wt\Delta_1$. Since they lie above the same separatrix on the base space, there is an element $\gamma_p$ in the fundamental domain of $\Leaf(\FF_1)$ that satisfies $\gamma_p\cdot\ell^+_p=\ell^-_p$. The map $f$ induces an isomorphism $f^*$ between the fundamental groups of $\Leaf(\FF_1)$ and $\Leaf(\FF_2)$.	So we have $f^*(\gamma_p)\cdot\Phi(\ell^+_p)=\Phi(\ell^-_p)$.
	We can extend continuously $h_0$ to the leaf $\ell_p$ if and only if
	$\Phi\circ\gamma_p$ coincides with $f^*(\gamma_p)\circ\Phi$ on $\ell^+_p$. These two maps are isotopic through homeomorphism from $\ell^+_p$ to $f(\ell^-_p)$.
	Corollary~\ref{cor-non-or-sep} asserts that the opposite separatrixes $\ell_p$, for all singular point $p$, admits disjoint neighborhoods $U_p$. Let $\wt U_p$ be the preimage of $U_p$ inside $\wt\Delta_1$. Then one can build a homeomorphism $g_p$ of $\wt U_p$ which is equal to the identity on $\partial\wt U_p\setminus\ell^+$, which preserves the leaves of the foliations, and so that $\Phi\circ\gamma_p\circ g_p$ coincides with $f^*(\gamma_p)\circ\Phi\circ g_p$ on $\ell^+_p$.

	We define $h\colon\RR^2\to\RR^2$ to be equal to the projection of $\Phi\circ g_p$ inside $U_p$, for all singularity $p$, to be equal to $h_0$ on $\Delta_1\setminus\cup_p U_p$, and that sends the singularity of $\ell_p$ to the singularity of $f(\ell_p)$. From the discussion above, it is well-defined. It is a homeomorphim which conjugates $\FF_1$ to $\FF^2$, and it induces $f$ on the leaf spaces.
\end{proof}

\subsection{From $\Pre^*$-pre-lamination to $pA$-foliation} 

Let~$\FF$ be a $pA$-foliation on a plane, with singularity locus $\Sigma$. Denote by $\partial_\infty\FF$ the circle at infinity of~$\FF$, and by $L_\infty(\FF)$ the \emph{end pre-lamination} of~$\FF$ (see the definition in Section~\ref{s.pre-lam-def}).

We prove the following theorem, which is a more precise version of the sufficient part of \thref{main-A-sing} (the necessary part is already proved).

Given a pre-lamination $L$ of the circle, a \emph{realization} is a pair $(\FF,g)$ where~$\FF$ a $pA$-foliation and $g\colon\partial_\infty\FF\to S^1$ is a cyclically non-decreasing continuous map inducing a bijection $L_\infty(\FF)\xrightarrow{\wb g}L$.

\begin{theorem}\thlabel{th-good-lam-are-end-sing-fol}
	Let $L$ be a $\Pre^*$-pre-lamination $L$ of $S^1$. Then there exists a $pA$-foliation $\FF$ on the plane and an increasing homeomorphism $g\colon\partial_\infty\FF\to S^1$ that sends the end lamination $L_\infty(\FF)$ of $\FF$ onto $L$.

	Moreover, if $(\FF',g')$ is another pair as above, then there exists a homeomorphism $h\colon(\RR^2,\FF)\to(\RR^2,\FF')$ whose induced homeomorphism $\wt h\colon \partial_\infty(\FF)\to\partial_\infty(\FF')$ satisfies $g'\circ\wt h=g$ and that sends $L_\infty(\FF)$ onto $L_\infty(\FF')$.

\end{theorem}

Let $L$ be a $\Pre^*$-pre-lamination, and $\Leaf^*(L)$ be the planar structure associated to~$L$.
According to Theorem~\ref{th-Kaplan-sing}, there is a $pA$-foliation~$\FF$ whose planar structure $\Leaf(\FF)$ is isomorphic to $\Leaf^*(L)$.
Let $L_\infty$ denote the end pre-lamination of~$\FF$. It is a $\Pre^*$-pre-lamination.

Let~$\FF$ be a $pA$-foliation and $L_\infty$ be its induced pre-lamination on $\partial_\infty\FF$.
Given a non-singular leaf of a face $f$ of $\FF$, denote by $\psi(f)$ the leaf of $L_\infty$ between the two endpoints of $f$ in $\partial_\infty\FF$.
Define the map $\psi^*\colon\Leaf(\FF)\to\Leaf(L_\infty)$ on a leaf $f\in\Leaf(\FF)$ as follows:
\begin{itemize}
	\item if $f$ is a non-singular leaf of~$\FF$ then $\psi^*(f)=\psi(f)$,
	\item if $f$ is a separatrix of~$\FF$, then $f$ is the intersection of two adjacent faces $f_1,f_2$ of~$\FF$, and $\psi^*(f)$ is the separatrix $\{\psi(f_1),\psi(f_2)\}$ of $L_\infty$.
\end{itemize}

\begin{theorem}\label{t.identification-planaire}
	Let~$\FF$ be a $pA$-foliation and $L_\infty$ its induced pre-lamination on $\partial_\infty\FF$.
	Then the map $\psi^*\colon\Leaf(\FF)\to\Leaf(L_\infty)$ given above is an isomorphism of singular planar structures.
\end{theorem}

Let us assume the theorem for now.

\begin{proof}[Proof of Theorem~\ref{th-good-lam-are-end-sing-fol}] According to
	Theorem~\ref{t.identification-planaire} the planar structure $\Leaf(L_\infty)$ is canonically isomorphic to $\Leaf(\FF)$ and hence to $\Leaf(L)$. Note that $\Leaf(L_\infty)$ has few common ends. Thus Theorem~\ref{t.prelam-planar} provides a non-decreasing map of topological degree $1$ of the circle which is a semi conjugacy of $L_\infty$ ,with $L$. Furthermore, if $L$ has few common ends, then this semi-conjugacy is a conjugacy, ending the proof of Theorem~\ref{th-good-lam-are-end-sing-fol}
\end{proof}

\begin{lemma}
	A subset $X\subset\Leaf(\FF)$ is an open interval if and only if $\psi^*(X)$ is an open $L_\infty^*$-interval.
\end{lemma}

\begin{proof}
	We denote by $Y$ the set of non-separatrix leaves in $X$ union the faces of $\FF$ that contains a separatrix of $X$. Then $\psi(Y)$ is the set of leaves of $L_\infty$ that appear in $\psi^*(X)$. So $\psi^*(X)$ is an open $L_\infty^*$-interval if and only if $\psi(Y)$ is an open $L_\infty$-interval, which then induces $\psi^*(X)$.

	First assume that $X$ is an open interval embedded in $\Leaf(\FF)$, and prove that $\psi(Y)$ is an open $L_\infty$-interval.
	Let $\sigma\in\RR^2\setminus \Sing(\FF)$ be an open curve transverse to~$\FF$ so that $Y$ is the set of non-singular leaves and faces that cross $\sigma$. Using $\sigma$, one easily see that the items 1, 2 and 3 in Definition \ref{d.interval} hold true.

	We prove the fourth item.
	Let $f,h$ be distinct leaves or faces in $Y$. If $f$ and $h$ intersects $\sigma$ in two distinct points, then there exists a leaf $g\in Y$ that intersects $\sigma$ in between $f$ and $h$. Then $\psi(g)$ separates $\psi(f)$ and $\psi(h)$.
	If they intersect $\sigma$ in the same point, then there are adjacent faces of $\FF$, and $\psi(\{f,h\})$ is the separatrix of $L_\infty$ image of the separatrix $f\cap h\in\Leaf(\FF)$. Thus item 4 holds. It follows that $\psi(Y)$ is an open $L_\infty$-interval, and $\psi^*(X)$ is an open $L_\infty^*$-interval.

	Conversely, assume now that $\psi^*(X)$ is an open $L_\infty^*$-interval. Then $\psi(Y)$ is an open $L_\infty$-interval. We claim that $X$ is open. Assume the claim for now.
	Given any two $f,h\in Y$, it follows from item 1 that any leaf or face $g$ that separates $f$ and $h$ also belong to $Y$. Since $X$ is open, it implies that $X$ is connected. Indeed, if it was not, some $f\in\partial X$, not in $X$, would separate two connected components of $X$, contradicting item 1. Then item 4 implies that $X$ is Hausdorff. Therefore, $X$ is a Hausdorff, connected and open subset of a non-Hausdorff 1-manifold, so $X$ is homeomorphic to an open interval.

	We now prove the claim. Take a leaf $f$ in $X$ and an open and transversal curve $\tau\colon\RR\to\RR^2\setminus \Sing(\FF)$ with $\tau(0)\in f$. It follows from item 3 in Definition \ref{d.interval} that there exists two non-separatrix leaves $g,h$ in $X$ that are separated by $f$. Denote by $l_t\in\Leaf(\FF)$ the leaf that contains $\tau(t)$. When $t$ goes to zero, $l_t$ accumulates on $f$, so it separates $f$ from either $g$ or $h$ (depending on the sign of $t$). It follows from item 2 that, up to shortening $\tau$, $l_t$ belongs to $X$ for all $t$. The set of $l_t$ is open, so $X$ is open.
\end{proof}

As a straightforward corollary one gets:

\begin{corollary} The bijection $\psi\colon \Leaf(\FF)\to\Leaf(L_\infty)$ is a homeomorphism.
\end{corollary}

\begin{proof}[Proof of Theorem~\ref{t.identification-planaire}]
	Let $\psi\colon\Leaf(\FF)\to\Leaf(L_\infty)$ be the homeomorphism constructed above.
	It remains to prove that $\psi$ preserves the order of the branching and the cyclic order of the cyclic branching.

	Let $\nu$ be a cyclic branching. The separatrixes of a singular point of~$\FF$ are in one to one correspondence by $\psi$ with the separatrixes of a star of $L_\infty$, and the two families are ordered using the orientation on $\partial_\infty\FF$. As we took the same ordering convention, the map $\psi$ on $\nu$.

	Let $\mu$ be a branching of~$\FF$. If $\mu$ is a pair of adjacent separatrixes (of a common singularity), then the order on $\psi(\mu)$ is induced by the ordering on the cyclic branching that contains $\psi(\mu)$. If follows from the previous case that $\psi$ is increasing on $\mu$.

	Assume now that $\mu$ is not reduced to a pair of adjacent separatrixes. It corresponds to a shell $\Delta$ of $L_\infty$. The leaves (or faces) in $\mu$ are ordered from the point of view of a leaf $f$ in the common connected component of $\RR^2\setminus\mu$. One can take $f$ so that $\psi(f)$ is close to the root of $\Delta$. The leaf of $L$ in $\Delta_\mu$ are ordered from the point of view of the root. So the order on $\mu$ and on $\psi(\mu)$ given by the same orientation, and $\psi$ is increasing on $\mu$.
\end{proof}